\documentclass[a4paper,11pt,draft, reqno]{amsart}

\usepackage[margin=18mm]{geometry}
\usepackage{amssymb, amsmath, amsthm, amscd}

\usepackage[T1]{fontenc}
\usepackage{eucal,mathrsfs,dsfont}
\usepackage{color}

\renewcommand{\le}{\leqslant}
\renewcommand{\ge}{\geqslant}

\definecolor{mno}{rgb}{0.5,0.1,0.5}

\newcommand{\R}{\mathds R}
\newcommand{\M}{\mathscr M}

\newcommand{\vv}{\varphi}

\newcommand{\Pp}{\mathds P}
\newcommand{\Ee}{\mathds E}

\newcommand{\I}{\mathds 1}
\newcommand{\w}{\omega}
\newcommand{\N}{\mathds{N}}

\newcommand{\Z}{\mathds Z}

\newcommand{\sL}{\mathcal{L}}

\newcommand{\sD}{\mathcal{D}}
\def\<{\langle}
\def\>{\rangle}

\newtheorem{theorem}{Theorem}[section]
\newtheorem{lemma}[theorem]{Lemma}
\newtheorem{proposition}[theorem]{Proposition}
\newtheorem{corollary}[theorem]{Corollary}

\theoremstyle{definition}

\newtheorem{example}[theorem]{Example}
\newtheorem{remark}[theorem]{Remark}
\newtheorem{assumption}[theorem]{Assumption}

\numberwithin{equation}{section}

\begin{document}
\allowdisplaybreaks

\title[random conductance models with long-range jumps]
{\bfseries Quenched local limit theorem for  random conductance models with long-range jumps}

\begin{abstract}
We establish the quenched local limit theorem for reversible random walk
on $\Z^d$ (with $d\ge 2$)
among stationary ergodic random conductances  that permit jumps of arbitrary length. The proof is based on the weak parabolic Harnack inequalities and on-diagonal heat-kernel estimates for long-range random walks on general ergodic environments. In particular, this
partly solves \cite[Open Problem 2.7]{BCKW}, where the quenched invariance principle was obtained. As a byproduct, we
prove the maximal inequality with an
extra tail term for long-range reversible random walks, which in turn yields the everywhere sublinear property for the associated corrector.

\medskip

\end{abstract}

\author{Xin Chen, \quad Takashi Kumagai\quad \hbox{and}\quad Jian Wang}
\thanks{\emph{X.\ Chen:}
   School of Mathematical Sciences, Shanghai Jiao Tong University, 200240 Shanghai, P.R. China. \texttt{chenxin217@sjtu.edu.cn}}
   \thanks{\emph{T.\ Kumagai:}
Department of Mathematics,
Waseda University, Tokyo 169-8555, Japan.
\texttt{t-kumagai@waseda.jp}}
  \thanks{\emph{J.\ Wang:}
    School of Mathematics and Statistics \& Key Laboratory of Analytical Mathematics and Applications (Ministry of Education) \& Fujian Provincial Key Laboratory
of Statistics and Artificial Intelligence, Fujian Normal University, 350007 Fuzhou, P.R. China. \texttt{jianwang@fjnu.edu.cn}}

\maketitle

\section{Introduction and main result}
Random conductance model on $\Z^d$ is a model of reversible random walk on random media
determined by a collection $\{C_{x,y}(\omega): x,y\in \Z^d\}$ of non-negative numbers with
$C_{x,y}(\omega)=C_{y,x}(\omega)$ for all $x,y\in \Z^d$ (called the conductance)
on a probability space $(\Omega, {\mathcal F}, \Pp)$ that governs the randomness
of the media. Transition probabilities of the random walk (Markov chain) are determined by
$P(x,y):=C_{x,y}/\sum_{z\in \Z^d}C_{x,z}$ for all $x,y\in \Z^d$.
During last decades, much efforts have been
devoted to the study of quenched invariance principle (QIP) for
the nearest neighbor random walk (NNRW)
in random conductance models.
Here, QIP is the convergence of the associated scaled processes to
Brownian motion for a.s. sample of random conductances,
and the nearest neighbor property means that $C_{x,y}(\w)$ is positive only when $x$ is the nearest neighbor to $y$. Sidoravicius
and Sznitman \cite{SS} proved  the
QIP for the simple random walk (that corresponds to the case that the conductance $C_{x,y}(\w)\equiv1$ for all $x,y\in \Z^d$ such that $C_{x,y}(\w)\neq 0$) on the supercritical Bernoulli percolation cluster on $\Z^d$ with $d\ge 4$, and
for the NNRW in uniformly-elliptic nearest-neighbor
{\it i.i.d.} random conductances.
With the aid of
a special kind of ergodic shift maps, the QIP for the simple random walk
on the supercritical Bernoulli percolation cluster
on $\Z^d$ with $d\ge 2$
was established by Berger and Biskup
 \cite{BB} and also by Mathieu and Piatnitski \cite{MP}.
Moreover, Barlow and Deuschel \cite{BD},
Biskup and Prescott \cite{BisP} and
Mathieu \cite{M}
have studied the QIP for the NNRW on
more general random conductances
by various different methods. We refer the readers to  Andres, Barlow, Deuschel and Hambly \cite{ABDH}
for the latest progress on this problem, where the QIP was proved under general conditions on
{\it i.i.d.} random
conductances
including
supercritical Bernoulli percolation cluster.

Beyond the scope of mutually independent conductances, the QIP for the NNRW on ergodic random conductance models
also attracts lots
of interests and it is usually  tackled
under the following moment condition on the conductance:
\begin{equation}\label{e1-0}
C_{x,y}(\w)\in L^p(\Omega;\Pp),\,\, \frac{1}{C_{x,y}(\w)}\in L^q(\Omega;\Pp),\quad  x,y\in \Z^d\ {\rm with}\ x\sim y
\end{equation}
with $p,q\ge 1$, where $x\sim y$ means that $x$ is the nearest neighbor of $y$.
Andres, Deuschel and Slowik \cite{ADS} proved
the QIP for the NNRW in ergodic random conductance models
with the moment condition $1/p+1/q<2/d$ in \eqref{e1-0}. The key step in \cite{ADS} is to establish the maximal inequality for the corrector equation, which in turn yields
the everywhere sublinear property of the corresponding corrector.
Bella and Sch\"affner \cite{BS1} then improved the moment condition above into $1/p+1/q<2/(d-1)$, which was illustrated to be nearly
optimal for the everywhere sublinearty of the corrector by a counterexample
given
in \cite[Theorem 2.5]{BCKW}. Under the same moment conditions, the quenched local limit theorem for the NNRW on ergodic random conductance models
was shown by Andres, Deuschel and Slowik \cite{ADS1} and Bella and Sch\"affner \cite{BS}, respectively. Recently, the QIP
for the time-inhomogeneous NNRW in space-time ergodic random
media
also has been systematically investigated by
Andres, Chiarini, Deuschel and Slowik \cite{ACDS}, Andres, Chiarini and Slowik \cite{ACS} and Biskup and Rodriguez \cite{BR}.

For the long range random walk (LRRW) on ergodic random conductance models, i.e., $C_{x,y}(\w)$ may not be identically zero when $x$ is not the nearest neighbor of $y$, one main difficulty to prove the QIP is to establish the everywhere sublinearty of the associated corrector, due to the effects of
long range jumps. For the LRRW with $L^2$ integrable jumping kernel in ergodic random conductance models, Biskup, Chen, Kumagai and Wang \cite{BCKW} have proved
the QIP by introducing the time-change arguments as well as an idea of locally modifying environments.
Flegel, Heida and Slowik \cite{FHS} and Piatnitski and Zhizhina \cite{PZ2} studied the stochastic homogenization
for   the LRRW  with $L^2$ integrable jumping kernel in ergodic environments by applying $L^2$ sublinearty of corresponding corrector.
See \cite{BH, Fa} for the recent progress on related topics.
For the LRRW with heavy tails in the sense that the jumping kernel is not $L^2$ integrable,
the limit process is no longer
Brownian motion. The readers are referred to \cite{XCKW,XCKW1,CS1,CS2} for more details.
We also
remark that the simple random walk on
the open cluster of long range percolation belongs to the class of random conductance models with long range jumps. See \cite{AN,Ba2-1,Ba2,Ba1,Bab,Be,B1,CCK,CS1,CS2,DFH,DS} and references therein for more details
concerning various geometric and analytic
aspects on long range percolation.

Although there are several results concerning
the QIP (or the stochastic homogenization) for the LRRW in ergodic random conductance
models, the quenched local limit theorem
is still unknown. In particular, due to the effects of long range jumps,
the methods adopted in \cite{ADS,ADS1,ACS} can not
be
used to
establish
the maximal inequality for the caloric functions and the parabolic Harnack inequality,
which are two crucial ingredients to prove
the quenched local limit theorem. In this paper, we are going to
address this problem,
which will partly solves
\cite[Open Problem 2.7]{BCKW}.
Actually, for the LRRW in random conductance models the classical parabolic Harnack inequality (even the elliptic Harnack inequality)
does not hold in general;
see \cite[Example 4.7]{XCKW}. Instead  we will establish the weak parabolic Harnack inequality in the present setting, which somehow reflects the behavior of long range jumps for the LRRW. We emphasize that, different from the classical parabolic Harnack inequality, there is  an extra tail term in the weak parabolic Harnack inequality, which in turn requires new ideas and
techniques
especially when the conductances are degenerate that are allowed in our paper. More
precisely, in contrast to known
results on the weak parabolic Harnack inequality for non-local operators (see e.g.\ \cite{FK}), we adopt the $L^p$ norm of the tail instead of the $L^\infty$ norm (see \eqref{t4-1-2} below) that is more efficient for the LRRW in general ergodic media;
to the best of our knowledge, this is the first time in the literature that $L^p$ norm of the tail
is considered.
\medskip

We now give the framework more precisely.
Let $\{\tau_x: x\in \Z^d\}$ be a group of measurable transformations on a probability space $(\Omega,\mathscr{F},\Pp)$ such that $\tau_0=\I$
and $\tau_{x+y}=\tau_x \circ \tau_y$ for every $x,y\in \Z^d$, where  $\I$ is the identity map on $\Omega$. Throughout the paper, we assume that
$d\ge 2$, and
$\{\tau_x: x\in \Z^d\}$ is stationary and ergodic on $(\Omega,\mathscr{F},\Pp)$ in the sense that
\begin{itemize}
\item [(1)] $\Pp(\tau_x A)=\Pp(A)$ for every $A\in \mathscr{F}$ and $x\in \Z^d$;

\item [(2)] if there exists $A\in \mathscr{F}$ such that $\tau_x A=A$ for every $x\in \Z^d$,
then either $\Pp(A)=0$ or $\Pp(A)=1$;

\item [(3)] the map $(x,\omega)\mapsto \tau_x \omega$ is $\mathscr{B}(\Z^d)\times \mathscr{F}$ measurable.
\end{itemize}

Let $\{C_{x,y}(\w):x,y\in \Z^d\}$ be a class of non-negative symmetric random variables on $(\Omega,\mathscr{F},\Pp)$
such that
$$
C_{x,y}(\w)=C_{y,x}(\w),\quad
C_{x+z,y+z}(\w)=C_{x,y}(\tau_z \w),\quad x,y,z\in \Z^d,\ \w\in \Omega,
$$ and
$$
\pi_x(\w):=\sum_{y\in \Z^d}C_{x,y}(\w)\in (0,+\infty),\quad x\in \Z^d, \,\,\Pp {\,\,\rm a.s.}\,\, \w\in \Omega.
$$
Due to the symmetry of $(x,y)\mapsto C_{x,y}(\w)$, in the literature the common value
$C_{x,y}(\w)$ is called the (random) conductance of the unordered edge $\langle x,y\rangle$.

Denote by $|z|$ the standard Euclidean norm for every $z\in \Z^d$.
Let $\lambda$ be the counting measure on $\Z^d$. For any fixed environment $\w\in \Omega$, consider the operator $\sL^\w$
on $L^2(\Z^d;\lambda)$:
\begin{equation}\label{e2-1}
\sL^\w f(x):=\sum_{y\in \Z^d}\left(f(y)-f(x)\right)C_{x,y}(\w),\quad f\in L^2(\Z^d;\lambda).
\end{equation} It is well known (see \cite{B}) that for a.s. any fixed $\w\in \Omega$, $\sL^\w$ is the infinitesimal generator for the variable speed random walk (VSRW) $(X^\w_t)_{t\ge 0}$
corresponding to the conductance $\{C_{x,y}(\w): x,y\in \Z^d\}$; moreover, the Dirichlet form
$(\sD^\w, \mathscr{D}(\sD^\w))$ on $L^2(\Z^d;\lambda)$ associated with the process  $(X^\w_t)_{t\ge 0}$ is given by
\begin{equation}\label{e2-2}
 \sD^\w(f,f)=
 \frac 1 2 \sum_{x,y\in \Z^d}\left(f(x)-f(y)\right)^2 C_{x,y}(\w),\quad \mathscr{D}(\sD^\w)=\{f\in L^2(\Z^d;\lambda): \sD^\w(f,f)<\infty\}.
\end{equation}
In particular, when $C_{x,y}(\w)=0$ for a.s. $\w\in \Omega$ and every $x,y\in \Z^d$ with
$|x-y|>1$,
the process $(X^\w_t)_{t\ge 0}$ above becomes
the NNRW
in random conductance models.

Different from the setting for the NNRW,  for the LRRW we will consider the following moment condition
\begin{equation}\label{a2-1-1}
 \tilde \mu(\w):=\sum_{z\in \Z^d}|z|^2 C_{0,z}(\w) \in L^p(\Omega;\Pp),\quad
 \tilde \nu(\w):=\sum_{z\in \Z^d:|z|=1}\frac{1}{C_{0,z}(\w)}\in L^q(\Omega;\Pp)
\end{equation}
instead of \eqref{e1-0}. The following assumption is made in \cite{BCKW}.
\begin{assumption}\label{a2-1}\it
The condition \eqref{a2-1-1} holds with some
$p,q\in (1,+\infty]$
such that $$\frac{1}{p}+\frac{1}{q}<\frac{2}{d}.$$
\end{assumption}

According to \cite[Theorem 2.1]{BCKW}, under Assumption \ref{a2-1} the QIP holds for the LRRW  $(X_t^\w)_{t\ge0}$ such that the limit process is
Brownian motion with (constant) diffusion matrix $M=(M_{ij})_{1\le i,j\le d}$ that is defined by
\begin{equation}\label{e1-1}
M_{ij}:=\Ee\left[\sum_{z\in \Z^d}C_{0,z}\left(z_i+\chi_i(z)\right)\left(z_j+\chi_j(z)\right)\right],\quad 1\le i,j\le d,
\end{equation}
where $\chi=(\chi_1,\cdots,\chi_d):\Z^d \times \Omega\to \R^d$ is the corrector constructed in Section 3 below.
See also \cite[Remark 2.4]{BCKW} for the details. The proof of \cite[Theorem 2.1]{BCKW} is a combination of corrector
methods, functional inequalities and heat-kernel technology. The novelty is that it avoids proving everywhere-sublinearity of the corrector. This is relevant because a counterexample in \cite{BCKW} indicates that, for long-range percolation with exponents between $d+2$ and $2d$ (which is a special example for the LRRW in ergodic random conductance model), the corrector exists but fails to be sublinear everywhere.

\ \

In order to study the quenched local limit theorem, we need the following stronger assumption.

\begin{assumption}\label{Assm2}\it
Suppose that  \eqref{a2-1-1} holds with some $p,q\in (1,+\infty]$ that satisfy
\begin{equation}\label{a2-2-1}
\frac{1}{p}+\frac{1}{q}\le
\left(1+\frac{1}{p}\right)
\frac{1}{d},\quad \frac{1}{p-1}+\frac{1}{q}<\frac{2}{d},
\end{equation}
 and that either the following condition holds
\begin{equation}\label{a2-2-2}
\tilde \mu_{d+2}(\w)
:=\sum_{z\in \Z^d}|z|^{d+2}C_{0,z}(\w)\in L^p(\Omega;\Pp)
\end{equation}
or $q=+\infty$ which is equivalent to $$\inf_{x,y\in \Z^d:|x-y|=1}C_{x,y}>0.$$
\end{assumption}

Let $p^\w(t,x,y)$ be the heat kernel of the process $(X^\w_t)_{t\ge 0}$.
Write
$k_M(t,x)$
for the Gaussian heat kernel with the diffusion matrix $M$ defined by \eqref{e1-1}.
The following theorem is our main result of this paper.

\begin{theorem}\label{CLT} Under Assumption $\ref{Assm2}$, for any $T_2>T_1>0$ and $R>1$,
$$\lim_{n\to\infty}\sup_{|x|\le R}\sup_{t\in [T_1,T_2]}|n^dp^\w(n^2t,0,[nx])-k_{M}(t,x)|=0,$$ where $[x]=([x_1],[x_2],\cdots, [x_d])$ for any $x=(x_1,x_2,\cdots,x_d)\in \R^d$.
\end{theorem}

We will give some comments
(including difficulty and novelty)
on Theorem \ref{CLT}.
\begin{itemize}

\item[(1)] The quenched local limit
theorem for the LRRW in ergodic random conductance models is proved in Theorem \ref{CLT} under
some stronger moment condition \eqref{a2-2-1} and suitable control of long range jumps
\eqref{a2-2-2} (which yields the higher integrability of the conductances). In particular, Theorem \ref{CLT} partly
solves
\cite[Open Problem 2.7]{BCKW}.  Indeed, according to
a counterexample given in
\cite[Example 4.9]{XCKW}, even the elliptic Harnack inequality may not
hold for the LRRW in random conductance models
due to possible large conductances in long range jumps.
However, in the present setting we can show the weak parabolic Harnack inequality (WPHI) (see \eqref{t4-1-2} below) instead, which
in turn implies the H\"older continuity of the associated caloric functions. The WPHI for non-local symmetric $\alpha$-stable like processes
on $\R^d$
has been introduced by Felsinger and Kassmann in \cite{FK}, where uniform and non-degenerate bounds of jump kernels are
required. Different from \cite{FK}, here we only use
moments of the conductances
(which corresponds to
moments of the jump kernels)
to establish the
WPHI for LRRW with $L^2$ integrable jumping kernel,
which
is adapted to general ergodic environments.

\item [(2)] Compared with the condition $1/p+1/q<2/d$ in \cite{BCKW} that suffices
to ensure the QIP for the LRRW in ergodic random conductance models, here we need a stronger
moment condition \eqref{a2-2-1}. \eqref{a2-2-1} is used to establish the Sobolev inequality \eqref{l2-6-1} which
is a key step for the Moser iteration procedure
adapted to LRRW in Section 4.
The main difference between \eqref{l2-6-1} and
\eqref{l2-2-1} is that \eqref{l2-6-1} holds for any function on $\Z^d$ without the restriction of having finite support. It is crucial
in obtaining the maximal inequality \eqref{c3-1-1} and the negative moment estimate \eqref{l4-3-1} for caloric functions in large scales, which is suitable for handling the effects of long range jumps.
See also
Remark \ref{r4-1} below for more details.
Meanwhile, as illustrated by a counterexample given in
\cite[Proposition 1.5 (ii)]{DF},
suitable moment conditions
on the conductances are necessary for local limit theorem.

\item [(3)] If $q<+\infty$, the condition \eqref{a2-2-2} is used to prove on-diagonal upper bounds of heat kernel
(see Proposition \ref{p3-2} below), which is a direct consequence of the maximal inequality \eqref{c3-1-1} for caloric functions.
Different from the NNRW, there is an extra tail term in \eqref{c3-1-1} which reflects
the behavior of long range jumps for the LRRW. We believe that such tail term is necessary since a caloric function on
a subset $A\subset \Z^d$ associated with the LRRW is determined by its value on the whole $\Z^d$, not only its value on
the set $A$. Roughly speaking, the condition
\eqref{a2-2-2} will be adapted
 to control this extra tail term.
\end{itemize}

As a byproduct of our approach to establish  maximal inequalities for the caloric functions, we can obtain the following statement
the everywhere sublinear property for the corrector will hold true under some additional assumption on the conductance.

\begin{theorem}\label{p3-1}
Suppose that Assumption $\ref{a2-1}$ holds
with some $p,q\in (1,+\infty]$
and that
\begin{equation}\label{p3-1-1a}
\tilde \mu_*(\w)
:=\sum_{z\in \Z^d}|z|^{(d(p-1)+4p)/(1+p)}C^{{2p}/({p+1})}_{0,z}(\w)\in L^{({p+1})/{2}}(\Omega;\Pp).
\end{equation}
Then the following everywhere sublinear property holds for the corrector $\chi$
involved in the expression \eqref{e1-1} of the diffusion matrix $M$:
\begin{equation}\label{p3-1-1}
\lim_{n \to \infty}\frac{1}{n}\sup_{x\in B_n}|\chi(x)|=0,
\end{equation}
where $B_n:=\{x\in \Z^d: |x|\le n\}$.
\end{theorem}

To verify the everywhere sublinear property for the corrector $\chi$ as usual we will establish the maximal inequality for it.
Partly inspired by \cite{FK,S}, for the LRRW we will prove the maximal inequality with an
extra tail term which reflects the effects
of long range jumps; see \eqref{p3-1-6} and \eqref{t3-2-4} below.
Essentially different
from the framework of \cite{FK,S}, here we
can only
apply the moments rather than the uniform bounds of the conductances. Combining the moment condition on
$\tilde \mu_*$
defined by \eqref{p3-1-1a} with
the $L^{{2p}/({p-1})}$ sublinearty  for the corrector (see \eqref{l2-1-1}), we can finally control the tail term
in the maximal inequality for the corrector $\chi$ to show its everywhere sublinearty \eqref{p3-1-1}. As indicated
in \cite{BFO,N}, the $L^{{2p}/({p-1})}$ order seems the optimal
exponent for the sublinearty of
associated corrector
in general ergodic environments. We believe that for some special models with some stronger conditions on
the ergodic random conductances (such as the
long range percolation cluster), the sublinearty with better exponents can be obtained.

\ \

Finally, we give two examples
for
the applications of Theorems \ref{CLT} and \ref{p3-1}.
\begin{example}\label{exmT}\it Consider a sequence of independent random variables $\{C_{x,y}(\w)=C_{y,x}(\w):x,y\in \Z^d\}$.
\begin{itemize}
\item[{\rm(i)}] If
for all $x,y\in \Z^d$ with $x\neq y$
\begin{equation}\label{ex1-1}
\Pp(C_{x,y}=1)=\frac{1}{|x-y|^{d+s}},\quad
\Pp(C_{x,y}=0)=1-\frac{1}{|x-y|^{d+s}}
\end{equation}
with $s\in [2d-2,\infty)\cap (d+2,\infty)$,
then the quenched local limit theorem holds.
On the other hand, if $\{C_{x,y}:x,y\in \Z^d\}$ satisfies \eqref{ex1-1} with $s>{(d^2+2d)}/{4}$, then the
everywhere sublinear property \eqref{p3-1-1} is true for the associated corrector.

\item[{\rm(ii)}] Suppose that for all $x,y \in \Z^d$ with $x\neq y$
\begin{equation}\label{ex1-2}
C_{x,y}=\frac{\xi_{x,y}}{|x-y|^{d+s}},
\end{equation}
where $s>2$,
$\xi_{x,y}\equiv1$ for $|x-y|=1$, and $\{\xi_{x,y}=\xi_{y,x}:
x,y\in \Z^d \text{ with } |x-y|>1\}$
are i.i.d. random variables so that $\xi_{x,y}\in L^r(\Omega;\Pp)$ with some $r>d/2$.
Then the quenched local limit theorem holds; moreover,  the associated corrector also enjoys the everywhere sublinear property
\eqref{p3-1-1}.
\end{itemize} \end{example}

Note that, Example \ref{exmT}(i) corresponds to the long-range percolation, which is a special example for the LRRW
in the
ergodic random conductance model. As shown by a counterexample in
\cite[Theorem 2.5]{BCKW}, when
 $d\ge 3$ and
 $s\in
  (2,d)$ the corrector fails to be sublinear everywhere.
 Then the assertion in Example \ref{exmT}(i) further shows the everywhere sublinear property of the corrector when $s>{(d^2+2d)}/{4}$, which is not claimed in \cite{BCKW}.

\ \

The remainder of this paper is arranged as follows. In the next section, we collect some known tools that are used in the proofs of Theorems \ref{CLT} and \ref{p3-1}, and prove a new Sobolev-type inequality (see \eqref{l2-6-1} below). Section \ref{section3} is devoted to maximal inequalities for the corrector and caloric  functions associated with the operator $\sL^\w$. In particular, we prove Theorem \ref{p3-1}
and obtain the on-diagonal heat kernel bounds as well. In Section \ref{section4} we establish the parabolic weak Harnack inequalities and the H\"{o}lder regularity for caloric   functions by following the approach due to Felsinger and Kassmann \cite{FK} for local regularity for symmetric $\alpha$-stable-like operators. Then,
the proofs of Theorem \ref{CLT} and Example \ref{exmT} are given in the last section.

\section{Preliminaries}
Let $\R_+:=[0,\infty)$.
Given any $A\subset \Z^d$, let
$
\partial A:=\{x\in A: {\rm there\ exists}\ y\notin A\ {\rm such\ that}\ |x-y|=1\}.
$
For $A\subset \Z^d$,  $|A|=\lambda(A)$ denotes the counting measure of $A$. For $I\subset \R$, $|I|$ denotes the Lebesgue measure of $I$.  For $A\subset \R\times \Z^d$, let $|A|:=(dt\times \lambda)(A)$ with $dt$ being
the Lebesgue measure on $\R$.
For any function $f: \Z^d\to \R$, set
$$\|\nabla f\|_\infty:=\sup_{x,y\in \Z^d: x\neq y}\frac{|f(x)-f(y)|}{|x-y|}.$$
For any finite subset $A\subset \Z^d$ and $f: A\to \R$,
define
$$
\|f\|_{p,A}:=\begin{cases}\left(\frac{1}{|A|}\displaystyle \sum_{x\in A}|f(x)|^p\right)^{{1}/{p}}=:\left(\displaystyle \oint_A |f|^p\right)^{1/p},&\quad  p\in [1,\infty),\\
\sup_{x\in A}|f(x)|,&\quad p=\infty.\end{cases}
$$
For every bounded subinterval $I \subset \R$, finite subset $A\subset \Z^d$,
$p\in [1,+\infty]$ and function $f:I \times A \to \R$, define
$$
\|f\|_{p,p',I\times A}:=\begin{cases}\left(\frac{1}{|I|}\displaystyle \int_{I}\|f(t,\cdot)\|_{p,A}^{p'}\,dt\right)^{{1}/{p'}},&\quad p'\in [1,\infty),\\
{\rm esssup}_{t\in I}\|f(t,\cdot)\|_{p,A},&\quad p'=\infty.\end{cases}
$$

For every $s, t>0$, $R\ge 1$ and $x\in \Z^d$, define
$$
B_R(x):=\{z\in \Z^d: |z-x|\le R\},\quad Q_{t,x,s,R}:=[t-s,t]\times B_R(x).
$$
For simplicity, we write $B_R(0)$ and $Q_{0,0,s,R}$ as $B_R$ and $Q_{s,R}$ respectively.  Throughout the paper, for $p, q,\alpha\in [1,+\infty]$,
we use
notations $p_*$, $q_*$ and $\alpha_*$ to denote
the conjugate numbers
of the constants $p$, $q$ and $\alpha$ respectively, i.e., $p_*:={p}/({p-1})$.
In particular, here we use the convention that $p_*=\infty$ when $p=1$, and $p_*=1$ when $p=\infty$.

In the following, define $\mu:\Z^d \times \Omega \to \R$ and $\nu:\Z^d \times \Omega \to \R$
as follows
\begin{equation}\label{e2-0}
\mu(x,\w):=
 \tilde \mu(\tau_x \w),\quad \nu(x,\w):=
  \tilde \nu(\tau_x \w),\quad (x,\w)\in \Z^d \times \Omega,
\end{equation} where $\tilde \mu$ and $\tilde \nu$ are given by \eqref{a2-1-1}.
For simplicity, in the following without confusion we omit $\w$ in the notations.

\begin{lemma}$($\cite[Proposition 3.5]{ADS}$)$\label{l2-2}
Suppose that Assumption $\ref{a2-1}$ holds. Then, for any $q\in [1,\infty)$, there is a constant
$C_1>0$
so that for every finite subset $A\subset \Z^d$, cut-off function
$\eta:\Z^d \to [0,1]$ satisfying that ${\rm supp}[\eta]\subset A$ and $\eta=0$ on $\partial A$, and $ u:\Z^d \to \R$,
\begin{equation}\label{l2-2-1}
\|(\eta u)^2\|_{\rho, A}\le
C_1
|A|^{{2}/{d}}\|\nu\|_{q,A}\left(
\frac{\tilde \sD_{\Z^d;\eta^2}(u)}{|A|}
+\|\nabla \eta\|_\infty^2 \|u^2\mu\|_{1,A}\right),
\end{equation}
where $\rho=\frac{d}{ d-2+d/q}$ and
$$
\tilde \sD_{\Z^d;\eta^2}(u):=\sum_{x,y\in \Z^d:|x-y|=1}\left(u(x)-u(y)\right)^2 \max\left\{\eta^2(x),\eta^2(y)\right\}
C_{x,y}.
$$
 When $d\ge3$, the inequality \eqref{l2-2-1} holds for $q=\infty$ as well.
\end{lemma}

\begin{lemma}\label{l2-6}
Suppose that the condition \eqref{a2-1-1} holds with some
$p,q\in (1,+\infty]$
satisfying that
\begin{equation}\label{l2-6-0}
\left(1-\frac{1}{d}\right)\frac{1}{p}+\frac{1}{q}\le \frac{1}{d}.
\end{equation}
Then there is a constant
$C_2>0$
so that for every $R\ge 1$ and $u:\Z^d \to \R$ it holds that
\begin{equation}\label{l2-6-1}
\begin{split}
\|u^2\|_{\rho, B_R}\le C_2|B_R|^{{2}/{d}}\|\nu\|_{q,B_R} \left(\frac{1}{|B_R|} \sum_{x,y\in B_R:|x-y|=1}(u(x)-u(y))^2 C_{x,y}\right)+
 \|u^2\|_{p_*,B_R},
\end{split}
\end{equation}
where $\rho=\frac{d}{d-2+d/q}$.
\end{lemma}
\begin{proof}
The classical local $L^1$-Poincar\'e inequality can be proved by the discrete version of the co-area formula, see \cite[Lemma 3.3.5, p.\ 380]{Sa},
which in turn implies that
the following local $L^\beta$-Poincar\'e inequality holds
(see the explanation in \cite[p.\ 91]{C}),
i.e.,  for all $\beta\in [1,d)$, there is a constant $c_1>0$ such that for all $x\in\Z^d$, $R\ge1$ and $u:\Z^d\to \R$,
$$
\sum_{z\in B_R(x)}\left|u(z)-\oint_{ B_R(x)} u\right|^\beta\le c_1R^\beta\sum_{y,z\in B_R(x):|y-z|=1}|u(y)-u(z)|^\beta.
$$
This along with \cite[Theorem 4.1]{C} gives us the following $L^\beta$-Sobolev inequality (see also
\cite[(2.5)]
{ADS1})
\begin{equation}\label{l2-6-2}
\sum_{x\in B_R}\left|u(x)-\oint_{B_R}u\right|^{\frac{d\beta}{d-\beta}}\le c_2\left(\sum_{x,y\in B_R:|x=y|=1}|u(x)-u(y)|^\beta\right)^{\frac{d}{d-\beta}}.
\end{equation}
Recall that $\rho=\frac{d}{d-2+d/q}$, and define $\alpha=\frac{(p-1)qd}{pq(d-1)+pd-q}$. It holds that
\begin{equation}\label{l2-6-3}
\frac{2d\alpha p_* }{d-\alpha}=2\rho=\frac{2\left(2p_*-1\right)q\alpha}{2q-(q+1)\alpha}.
\end{equation}
In particular, by  \eqref{l2-6-0} and $d\ge 2$, one can see that $\alpha\in
[1,2q/(q+1))$.
Applying the Sobolev inequality \eqref{l2-6-2} to $|u|^{2p_*}$ with $\beta=\alpha$ (noting that it is clear that $\alpha<d$ thanks to $d\ge 2$) yields that
\begin{align*}
\sum_{x\in B_R}|u(x)|^{2\rho}&=
\sum_{x\in B_R}\left(|u(x)|^{2p_*}\right)^{\frac{d\alpha}{d-\alpha}}
\le c_3\sum_{x\in B_R}\left||u(x)|^{2p_*}-\oint_{B_R}|u|^{2p_*}\right|^{\frac{d\alpha}{d-\alpha}}+c_3|B_R|
\left(\oint_{B_R}|u|^{2p_*}\right)^{\frac{d\alpha}{d-\alpha}}\\
&\le c_4\left(\sum_{x,y\in B_R:|x-y|=1}\left||u(x)|^{2p_*}-|u(y)|^{2p_*}\right|^\alpha\right)^{\frac{d}{d-\alpha}}+
c_4|B_R|
\left(\oint_{B_R}|u|^{2p_*}\right)^{\frac{d\alpha}{d-\alpha}}.
\end{align*}
On the other hand, we have
\begin{align*}
&\left(\sum_{x,y\in B_R:|x-y|=1}\left||u(x)|^{2p_*}-|u(y)|^{2p_*}\right|^\alpha\right)^{\frac{d}{d-\alpha}}\\
&\le c_5\left(\sum_{x,y\in B_R:|x-y|=1}|u(x)-u(y)|^\alpha\cdot\left(|u(x)|^{(2p_*-1)\alpha}+|u(y)|^{(2p_*-1)\alpha}\right)\right)^{\frac{d}{d-\alpha}}\\
&\le c_6\left(\sum_{x,y\in B_R:|x-y|=1}|u(x)-u(y)|^2C_{x,y}\right)^{\frac{d\alpha}{2(d-\alpha)}}\cdot
\left(\sum_{x,y\in B_R:|x-y|=1}|u(x)|^{\frac{2(2p_*-1)\alpha}{2-\alpha}}C_{x,y}^{-\frac{\alpha}{2-\alpha}}\right)^{\frac{d(2-\alpha)}{2(d-\alpha)}}\\
&\le c_7\left(\sum_{x,y\in B_R:|x-y|=1}|u(x)-u(y)|^2C_{x,y}\right)^{\frac{d\alpha}{2(d-\alpha)}}\cdot
\left(\sum_{x\in B_R}|u(x)|^{\frac{2(2p_*-1)\alpha}{2-\alpha}}\nu(x)^{\frac{\alpha}{2-\alpha}}\right)^{\frac{d(2-\alpha)}{2(d-\alpha)}}\\
&\le c_7\left(\sum_{x\in B_R}\nu(x)^q\right)^{\frac{d\alpha}{2q(d-\alpha)}}\left(\sum_{x,y\in B_R:|x-y|=1}|u(x)-u(y)|^2C_{x,y}\right)^{\frac{d\alpha}{2(d-\alpha)}}
\left(\sum_{x\in B_R}|u(x)|^{2\rho}\right)^{\frac{d}{d-\alpha}\left(1-\frac{(q+1)\alpha}{2q}\right)},
\end{align*}
where in the second inequality we used the H\"older inequality with exponent ${2}/{\alpha}$, in the third inequality we used the fact that $\sum_{y:|x-y|=1} C_{x,y}^{-\alpha/(2-\alpha)}\le c_8 \nu(x)^{\alpha/(2-\alpha)}$, and the last inequality
follows from the H\"older inequality with exponent ${q(2-\alpha)}/{\alpha}$
(noting that the fact $\alpha\in \left[1,{2q}/({q+1})\right)$ ensures that ${q(2-\alpha)}/{\alpha}>1$) as well
as the second equality in \eqref{l2-6-3}.

According to both estimates above, we find that
\begin{align*}
&\sum_{x\in B_R}|u(x)|^{2\rho}\\
&\le
c_9\left(\sum_{x\in B_R}\nu(x)^q\right)^{\frac{d\alpha}{2q(d-\alpha)}}\left(\sum_{x,y\in B_R:|x-y|=1}|u(x)-u(y)|^2C_{x,y}\right)^{\frac{d\alpha}{2(d-\alpha)}}
\left(\sum_{x\in B_R}|u(x)|^{2\rho}\right)^{\frac{d}{d-\alpha}\left(1-\frac{(q+1)\alpha}{2q}\right)}\\
&\quad +c_9|B_R|\left(\oint_{B_R}|u|^{2p_*}\right)^{\frac{d\alpha}{d-\alpha}}\\
&\le \frac{1}{2}\sum_{x\in B_R}|u(x)|^{2\rho}+c_{10}\left(\sum_{x\in B_R}\nu(x)^q\right)^{\frac{\rho}{q}}\cdot\left(\sum_{x,y\in B_R:|x-y|=1}|u(x)-u(y)|^2 C_{x,y}\right)^{\rho}
\\
&\quad +c_{10}|B_R|\left(\oint_{B_R}|u|^{2p_*}\right)^{\frac{\rho}{p_*}},
\end{align*}
where we have used Young's inequality with exponent $l=\frac{d-\alpha}{d\left(1-{(q+1)\alpha}/({2q})\right)}$ and the fact
$\frac{d\alpha}{2(d-\alpha)}\frac{l}{l-1}=\rho$ that can be verified by some direct computations.
Hence,
$$\sum_{x\in B_R}|u(x)|^{2\rho}\le c_{11}|B_R|^{{\rho}/{q}}\|\nu\|_{q,B_R}^\rho\left(\sum_{x,y\in B_R:|x-y|=1}(u(x)-u(y))^2 C_{x,y}\right)^\rho
+c_{11}|B_R|\left(\oint_{B_R}|u|^{2p_*}\right)^{\frac{\rho}{p_*}},$$ and so
the desired conclusion \eqref{l2-6-1} follows.
\end{proof}

\begin{lemma}$($\cite[Proposition 2.3]{ADS1}$)$\label{l2-7}
Suppose that Assumption $\ref{a2-1}$ holds.
 There is a constant
$C_3>0$
so that for any $R\ge1$,
$\eta(x)=\Phi\left(|x|\right)$ with some non-increasing continuous function
$\Phi:\R_+\to \R_+$ such that ${\rm supp} [\eta]\subset B_R$, and any $u:\Z^d \to \R$, it holds that
\begin{equation}\label{l2-7-1}
\sum_{x\in B_R}\eta(x)^2\left|u(x)-\oint_{B_R,\eta^2}u\right|^2\le C_3
|B_R|^{2/d}\|\nu\|_{{d}/{2},B_R}
\frac{\tilde \sD_{B_R;\eta^2}(u)}{|B_R|},
\end{equation}
where $$\oint_{B_R,\eta^2}u:=\frac{\sum_{x\in B_R}u(x)\eta(x)^2}{\sum_{x\in B_R}\eta(x)^2}$$ and $$
\tilde \sD_{B_R;\eta^2}(u):=\sum_{x,y\in B_R:|x-y|=1}\left(u(x)-u(y)\right)^2 \max\left\{\eta^2(x),\eta^2(y)\right\}
C_{x,y}.
$$
\end{lemma}

The following inequality is essentially taken from \cite{DKP}.
\begin{lemma} \label{l2-3}
There is a constant
$C_4>0$
so that for any $u:\Z^d \to \R$, any $\eta:\Z^d \to \R$ with finite supports, any $R>0$ and $k>0$,
\begin{equation}\label{l2-3-1}
\begin{split}
&\sum_{x,y\in B_R}(u(x)-u(y))\left(u_{-k,+}(x)\eta^2(x)-u_{-k,+}(y)\eta^2(y)\right)C_{x,y}\\
&\ge \sum_{x,y\in B_R}\left(u_{-k,+}(x)-u_{-k,+}(y)\right)\left(u_{-k,+}(x)\eta^2(x)-u_{-k,+}(y)\eta^2(y)\right)C_{x,y}\\
&\ge \frac{1}{2}\sum_{x,y\in B_R}\left(u_{-k,+}(x)-u_{-k,+}(y)\right)^2\max\{\eta^2(x),\eta^2(y)\}C_{x,y}\\
&\quad -C_4
\sum_{x,y\in B_R}\max\left\{u_{-k,+}^2(x),u_{-k,+}^2(y)\right\}|\eta(x)-\eta(y)|^2C_{x,y},
\end{split}
\end{equation}
where $u_{-k,+}(x):=(u(x)-k)_+$.
\end{lemma}
\begin{proof}
The first inequality in \eqref{l2-3-1} has been proved in \cite[p.\ 1285]{DKP}, and
the second one in \eqref{l2-3-1}
is
shown in \cite[p.\ 1286]{DKP}.
\end{proof}

We need the following elementary iteration lemma, see, e.g., \cite[Lemma 7.1]{Giu}.
\begin{lemma}\label{add-lemma} Let $\beta>0$ and $k\ge 1$, and let $\{A_i\}_{0\le i\le k+1}$ be a sequence of real positive numbers such
that
$$A_{i+1}\le c_0
b^i
A_i^{1+\beta},\quad 0\le i\le k$$ with $c_0>0$ and $b>1$. If
$$A_0\le c_0^{-1/\beta} b^{-1/\beta^2},$$ then $$A_i\le b^{-i/\beta} A_0,\quad 0\le i\le k+1.$$   \end{lemma}

We also need the following elementary inequality for the Moser's iteration procedure of
non-local operators.
\begin{lemma}\label{l2-4}$($\cite[Lemma 3.3]{FK}$)$
The following statements hold.
\begin{itemize}
\item [(1)] Suppose that $q>1$, $a,b>0$ and $\alpha,\beta\ge 0$. Then it holds with
$c_{0,q}:=\max\left\{4,\frac{6q-5}{2}\right\}$
that
\begin{equation}\label{l2-4-1}
\begin{split}
\left(b-a\right)\left(\alpha^{q+1}a^{-q}-\beta^{q+1}a^{-q}\right)
&\ge \frac{\alpha \beta}{q-1}\left[\left(\frac{b}{\beta}\right)^{\frac{1-q}{2}}-
\left(\frac{a}{\alpha}\right)^{\frac{1-q}{2}}\right]\\
&\quad -c_{0,q}(\beta-\alpha)^2\left[\left(\frac{b}{\beta}\right)^{1-q}+\left(\frac{a}{\alpha}\right)^{1-q}\right],
\end{split}
\end{equation}
where we use the convention $\left(\frac{b}{\beta}\right)^{1-q}=0$ when $\beta=0$ since $1-q<0$.

\item [(2)] Suppose that $q\in (0,1)$, $a,b>0$ and $\alpha,\beta\ge 0$. Then with
$c_{1,q}:=\frac{2q}{3(1-q)}$ and $c_{2,q}:=\frac{4q}{1-q}+\frac{9}{q}$ it holds that
\begin{equation}\label{l2-4-2}
\begin{split}
(b-a)
\left(\alpha^{2}a^{-q}-\beta^{2}b^{-q}\right)
&\ge
c_{1,q}\left(\beta b^{\frac{1-q}{2}}-\alpha a^{\frac{1-q}{2}}\right)^2 -c_{2,q}(\beta-\alpha)^2\left(b^{1-q}+a^{1-q}\right).
\end{split}
\end{equation}
\end{itemize}
\end{lemma}

\begin{lemma}\label{l2-5}$($\cite[Lemma 4.1]{S}$)$\,\,
Let $f:[1/2,1]\to \R_+$  be such that there exist constants
$C_5,C_6,\theta>0$
and $0<\beta<1$ so that
$$
f(s_1)\le C_5(s_2-s_1)^{-\theta}+C_6+\beta f(s_2),\quad 1/2\le s_1<s_2\le 1.
$$
Then there exists a constant $C_7>0$ $($that only depends on $\theta$ and $\beta$$)$ such that
$$
f(s_1)\le C_7 (C_5(s_2-s_1)^{-\theta}+C_6),\quad 1/2\le s_1<s_2\le 1.
$$
\end{lemma}

The following lemma was firstly proved in \cite{BG}, see also \cite[Lemma 2.5]{ADS1}
(for discrete case)
or \cite[Lemma 3.5]{S},
which was needed to prove the weak parabolic Harnack inequality.

\begin{lemma}\label{l2-8}
Assume  that $\{U(\theta r)\}_{\theta\in [1/2,1]}$ is a family of domains in $\R\times \Z^d$ so that
$U(\theta_1 r)\subset U(\theta_2 r)$ for every $1/2\le \theta_1\le \theta_2\le 1$.
Suppose that $w:\R\times \Z^d\to \R_+$ satisfies that
\begin{equation}\label{l2-8-1}
\left|U(r) \cap \{(t,x)\in \R\times \Z^d: \log w(t,x)>\gamma\}\right|\le \frac{C_8}{\gamma}|U(r)|,\quad r,\gamma>0,
\end{equation}
and
\begin{equation}\label{l2-8-2}
\|w\|_{p_0,p_0,U(\theta r)}\le \left(\frac{C_8}{(1-\theta)^\delta}\right)^{\frac{1}{p}-\frac{1}{p_0}}\|w\|_{p,p,U(r)},\quad
\quad p\in (0,\min\{1,\sigma p_0\})
,\ 1/2\le \theta\le 1
\end{equation}
for some constants $\sigma\in (0,1)$, $p_0\in (0,+\infty]$ and $\delta,C_8>0$. Then there exists a positive
constant $C_9:=C_9(\sigma,\theta,\delta,C_8,p_0)$ such that
$$
\|w\|_{p_0,p_0,U(\theta r)}\le C_9.
$$
\end{lemma}

\section{Maximal inequalities and on-diagonal heat kernel estimates}\label{section3}
In this section, we will establish maximal inequalities for the corrector and the caloric  functions associated with the operator $\sL^\w$ given by \eqref{e2-1}. They respectively yield the everywhere sublinear property of the corrector and the on-diagonal heat kernel estimates.
Throughout this part, for every $u:\Z^d\to \R$ and $R>0$, we define the tail of $u$ as follows
\begin{equation}\label{e3-1}
{\rm Tail}(u,R)(x):=R^2\sum_{y\in B_R(x)^c}u(y)C_{x,y},\quad x\in \Z^d.
\end{equation} For simplicity, in the following without confusion we omit $\w$ in the notations.

\subsection{Maximal inequalities for the corrector}

We first briefly introduce  the construction of the corrector
$\chi:\Z^d\times \Omega \to \R^d$ corresponding to the operator $\sL$.
The readers are referred to \cite[Section 3.2]{B} or \cite[Section 5.2]{FHS}
for more details.

We say
a function $\psi: \Z^d\times \Omega \to \R$
is shift invariant, if
\begin{equation}\label{e2-3}
\psi(x+z,\w)-\psi(x,\w)=\psi(z,\tau_x \w),\quad x,z\in \Z^d,\ \omega\in \Omega.
\end{equation}
Define the space $L^2_{{\rm cov}}$ as follows
$$
L^2_{{\rm cov}}:=\left\{\psi:\Z^d\times \Omega \to \R: \eqref{e2-3}\ {\rm holds\ and}\ \|\psi\|_{L^2_{{\rm cov}}}<\infty\right\},
$$
where
$$
\|\psi\|_{L^2_{{\rm cov}}}^2:=\Ee\left[\sum_{x\in \Z^d}|\psi(x,\omega)|^2 C_{0,x}(\omega)\right].
$$
It is easy to see that $L^2_{{\rm cov}}$ is a Hilbert space endowed with the inner product
$$
\left\langle \psi_1,\psi_2\right\rangle_{L^2_{{\rm cov}}}=\Ee\left[\sum_{x\in \Z^d}\psi_1(x,\omega)\psi_2(x,\omega) C_{0,x}(\omega)\right],
\quad \psi_1,\psi_2\in L^2_{{\rm cov}}.
$$
For every $\varphi:\Omega \to \R$, define the gradient $D\varphi:\Z^d\times \Omega \to \R$ by
$D\varphi(z,\w):=\varphi(\tau_z \w)-\varphi(\w)$. Hence, for every $\varphi \in L^\infty(\Omega;\Pp)$,  $D\varphi$ satisfies \eqref{e2-3} and so $D\varphi\in L^2_{{\rm cov}}$. Furthermore,  define
$$
L_{{\rm pot}}^2=\overline{\{D\varphi: \varphi\in L^\infty(\Omega;\Pp)\}}^{L^2_{{\rm cov}}},\quad
L^2_{{\rm cov}}=L_{{\rm pot}}^2\oplus L_{{\rm sol}}^2;
$$
that is, $L_{{\rm sol}}^2$ is the orthogonal complement of $L_{{\rm pot}}^2$ on $L^2_{{\rm cov}}$.

Define the corrector $\chi:\Z^d\times \Omega \to \R^d$ by
$$
\chi:={\rm argmin}\left\{\Ee\left[\sum_{z\in \Z^d}\left|z+\tilde\chi(z,\w)\right|^2C_{0,z}(\w)\right]: \tilde \chi\in
(L^2_{{\rm pot}})^{\otimes d}
\right\}.
$$
By the definition above,
$\chi(z,\w)$ is the orthogonal projection of $z$
from $(L^2_{{\rm cov}})^{\otimes d}$ into  $(L^2_{{\rm pot}})^{\otimes d}$. Therefore, $\chi\in (L_{{\rm pot}}^2)^{\otimes d}$ and satisfies
the following equation
\begin{equation}\label{e2-4}
\sL \chi(\cdot,\w)(x)=-\sum_{z\in \Z^d}zC_{x,x+z}(\w),\quad (x,\w)\in \Z^d\times \Omega;
\end{equation} see \cite[Proposition 3.7]{B} for more details.
Moreover, we have the following average sublinear property for the corrector $\chi$.
\begin{lemma}\label{l2-1}
Suppose that Assumption $\ref{a2-1}$ holds. Then, it holds that
\begin{equation}\label{l2-1-1}
\lim_{R \to \infty}\frac{1}{R}\|\chi(\cdot,\w)\|_{\frac{2p}{p-1},B_R}=0,\ \ {\rm a.s.}\ \w\in \Omega.
\end{equation}
\end{lemma}
\begin{proof}  Write $\chi=\left(\chi_1,\cdots,\chi_d\right)$.
For any bounded function $\varphi:\Omega \to \R$, set $\tilde D \varphi(\w)=
 (\tilde D_1 \varphi(\w),\cdots \tilde D_d \varphi(\w))$ by
$\tilde D_i\varphi(\w):=\varphi(\tau_{-e_i}\w)-\varphi(\w)$ for
every $1\le i \le d$, where $\{e_i\}_{1\le i\le d}$ is
a
standard orthonormal basis
of $\R^d$. Similarly, define $\tilde \nabla \chi(x,\w):= (\tilde \nabla_1 \chi(x,\w),\cdots, \tilde \nabla_d \chi(x,\w))$
with $\tilde \nabla_i \chi(x,\w):=\chi_i(x-e_i,\w)-\chi_i(x,\w)$ for every $1\le i \le d$.

Since $\chi\in (L_{{\rm pot}}^2)^{\otimes d}$,  there exists a sequence of functions $\{\xi_n\}_{n\ge 1} \subset L^\infty(\Omega;\Pp)$ such that
$$
\lim_{n \to \infty}\Ee\left[\sum_{z\in \Z^d}\left|D\xi_n(z,\omega)-\chi(z,\omega)\right|^2C_{0,z}(\w)\right]=0.
$$
Hence, according to Assumption $\ref{a2-1}$ and the fact that $\chi(0,\w)=0$,
\begin{align*}
&\lim_{n \to \infty}\Ee\left[\left|\tilde D \xi_n(\w)-\tilde \nabla \chi(0,\w)\right|^{\frac{2q}{q+1}}\right]\\
&\le \lim_{n \to \infty}\Ee\left[\sum_{z\in \Z^d:|z|=1}C_{0,z}^{-q}\right]^{\frac{1}{q+1}}
\cdot\Ee\left[\sum_{z\in \Z^d: |z|=1}\left|\tilde D \xi_n(\w)-\tilde \nabla \chi(0,\w)\right|^{2}C_{0,z}(\w)\right]^{\frac{q}{q+1}}\\
& \le c_1\lim_{n \to \infty}\Ee\left[\sum_{z\in \Z^d:|z|=1}\left|D\xi_n(z,\omega)-\chi(z,\omega)\right|^2C_{0,z}(\w)\right]\\
&\le c_1\lim_{n \to \infty}\Ee\left[\sum_{z\in \Z^d}\left|D\xi_n(z,\omega)-\chi(z,\omega)\right|^2C_{0,z}(\w)\right]=0,
\end{align*}
which implies that $\tilde \nabla \chi(\cdot,\w)(0) \in \overline{\{\tilde D \varphi:\ \varphi\in L^\infty(\Omega;\Pp)\}}^{L^{{2q}/({q+1})}}$.

Combining this with the fact $\frac{1}{p}+\frac{1}{q}<\frac{2}{d}$ as well as \cite[Theorems 3.6 and 3.8]{N}, we
obtain
\eqref{l2-1-1}.
\end{proof}

\begin{lemma}\label{l3-1}
Let $\eta:\Z^d \to [0,1]$ be
a cut-off function such that ${\rm supp}[\eta]\subset B_R$ for some $R>0$.
Suppose that $u:\Z^d \to \R$ satisfies the following equation
\begin{equation}\label{l3-1-0}
\sL u(x)=\sL f(x),\quad x\in B_R,
\end{equation}
where $f\in \Z^d \to \R$ so that $\|\nabla f\|_\infty<\infty$. Then, there is  a constant
$C_1>0$ $($which is independent of $f$, $\eta$ and $u$$)$
so that for all $k>0$,
\begin{equation}\label{l3-1-1}
\begin{split}
&\sum_{x,y\in B_R}\left(u_{-k,+}(x)-u_{-k,+}(y)\right)^2\max\{\eta^2(x),\eta^2(y)\}C_{x,y}\\
&\le
C_1
\|\nabla \eta\|_\infty^2\sum_{x\in B_R} u_{-k,+}^2(x)\mu(x)+C_1\|\nabla f\|_\infty^2\sum_{x\in B_R}\I_{\{u(x)>k\}}\mu(x)\\
&\quad
+C_1\|\nabla f\|_\infty\|\nabla \eta\|_\infty\sum_{x\in B_R}u_{-k,+}(x)\mu(dx)\\
&\quad +C_1\sum_{x\in B_R}u_{-k,+}(x)\eta^2(x)\left(\sum_{y\in B_R^c}\left(u_{-k,+}(y)+\|\nabla f\|_\infty|x-y|\right)C_{x,y}\right),
\end{split}
\end{equation}
where $u_{-k,+}(x)=(u(x)-k)_+$.
\end{lemma}
\begin{proof}
By changing the positions of the variables $x$ and $y$ and using the fact that $\eta(x)=0$ for all
$x\in B_R^c$,
we have
\begin{align*}
&-2\sum_{x\in \Z^d}\sL u(x)u_{-k,+}(x)\eta^2(x)\\
&=\sum_{x,y\in \Z^d}\left(u(x)-u(y)\right)\left(u_{-k,+}(x)\eta^2(x)-u_{-k,+}(y)\eta^2(y)\right)C_{x,y}\\
&=\left(\sum_{x,y\in B_R}+2\sum_{x\in B_R,y \in B_R^c}\right)\left(u(x)-u(y)\right)\left(u_{-k,+}(x)\eta^2(x)-u_{-k,+}(y)\eta^2(y)\right)C_{x,y}\\
&=:I_1+I_2,
\end{align*} and
\begin{align*}
&-2\sum_{x\in \Z^d}\sL f(x)u_{-k,+}(x)\eta^2(x)\\
&=\sum_{x,y\in \Z^d}\left(f(x)-f(y)\right)\left(u_{-k,+}(x)\eta^2(x)-u_{-k,+}(y)\eta^2(y)\right)C_{x,y}\\
&=\left(\sum_{x,y\in B_R}+2\sum_{x\in B_R,y \in B_R^c}\right)\left(f(x)-f(y)\right)\left(u_{-k,+}(x)\eta^2(x)-u_{-k,+}(y)\eta^2(y)\right)C_{x,y}\\
&=:I_3+I_4.
\end{align*}
Thus, testing  \eqref{l3-1-0} with $v=u_{-k,+}\eta^2$ yields
$$
I_1+I_2=I_3+I_4.
$$

By \eqref{l2-3-1}, we have
\begin{align*}
I_1\ge &\frac{1}{2}\sum_{x,y\in B_R}\left(u_{-k,+}(x)-u_{-k,+}(y)\right)^2\max\{\eta^2(x),\eta^2(y)\}C_{x,y}\\&  -c_1
\sum_{x,y\in B_R}\max\left\{u_{-k,+}^2(x),u_{-k,+}^2(y)\right\}(\eta(x)-\eta(y))^2C_{x,y}\\
\ge &\frac{1}{2}\sum_{x,y\in B_R}\left(u_{-k,+}(x)-u_{-k,+}(y)\right)^2\max\{\eta^2(x),\eta^2(y)\}C_{x,y}\\
&-2c_1\|\nabla \eta\|_\infty^2 \sum_{x\in B_R}u_{-k,+}^2(x) \sum_{y\in B_R}|x-y|^2 C_{x,y} \\
\ge &\frac{1}{2}\sum_{x,y\in B_R}\left(u_{-k,+}(x)-u_{-k,+}(y)\right)^2\max\{\eta^2(x),\eta^2(y)\}C_{x,y}-2c_1\|\nabla \eta\|_\infty^2\sum_{x\in B_R}u_{-k,+}^2(x)\mu(x).
\end{align*}
On the other hand, note that for every $x,y\in \Z^d$,
\begin{align*}
\left(u(x)-u(y)\right)u_{-k,+}(x)&\ge -\left(u(y)-u(x)\right)_+(u(x)-k)_+\\
&\ge -(u(y)-k)_+(u(x)-k)_+=-u_{-k,+}(y)u_{-k,+}(x).
\end{align*}
We can deduce that
\begin{align*}
I_2&\ge -2\sum_{x\in B_R,y\in B_R^c}u_{-k,+}(x)u_{-k,+}(y)\eta^2(x)C_{x,y}\\
&=-2\sum_{x\in B_R}u_{-k,+}(x)\eta^2(x) \sum_{y\in B_R^c}u_{-k,+}(y)C_{x,y}
\end{align*}
For the term $I_3$, we set
\begin{align*}
I_3&=\sum_{x,y\in B_R}\left(u_{-k,+}(x)-u_{-k,+}(y)\right)\left(f(x)-f(y)\right)\eta^2(x)C_{x,y}\\
&\quad+\sum_{x,y\in B_R}\left(\eta^2(x)-\eta^2(y)\right)\left(f(x)-f(y)\right)u_{-k,+}(x)C_{x,y}\\
&=:I_{31}+I_{32}.
\end{align*}
Note that $u_{-k,+}(x)-u_{-k,+}(y)\neq 0$ only if $u(x)>k$ or $u(y)>k$. Applying Young's inequality, we get
\begin{align*}
I_{31}&\le \frac{1}{2}\sum_{x,y\in B_R}\left(u_{-k,+}(y)-u_{-k,+}(x)\right)^2\eta^4(x)C_{x,y}+ \frac{1}{2}\|\nabla f\|_\infty^2
\sum_{x,y\in B_R}\left(\I_{\{u(x)>k\}}+\I_{\{u(y)>k\}}\right)|y-x|^2C_{x,y}\\
&\le \frac{1}{2}\sum_{x,y\in B_R}\left(u_{-k,+}(y)-u_{-k,+}(x)\right)^2
\eta^2(x)
C_{x,y}+\|\nabla f\|_\infty^2 \sum_{x\in B_R}\I_{\{u(x)>k\}}\mu(x).
\end{align*}
Here in the last inequality we have used the fact $\eta^4(x)\le \eta^2(x)$ since $\eta^2(x)\le 1$.
By some direct computations, it holds that
\begin{align*}
I_{32}&\le 2\|\nabla f\|_\infty\|\nabla \eta\|_\infty\sum_{x,y\in B_R}u_{-k,+}(x)|x-y|^2 C_{x,y}\\
&\le 2\|\nabla f\|_\infty\|\nabla \eta\|_\infty\sum_{x\in B_R}u_{-k,+}(x)\mu(x).
\end{align*}
Note that $u_{-k,+}(y)\eta^2(y)=0$ for every $y\in B_R^c$, so
\begin{align*}
I_4&\le 2\|\nabla f\|_\infty \sum_{x\in B_R}u_{-k,+}(x)\eta^2(x) \sum_{y\in B_R^c}|y-x|C_{x,y}.
\end{align*}

Putting all the estimates above for
$I_1, \cdots, I_4$
together yields the desired conclusion \eqref{l3-1-1}.
\end{proof}

\begin{proof}[Proof of Theorem $\ref{p3-1}$] The proof is split into
two steps.

{\bf Step 1.}\,\,
Let $1\le \sigma_1<\sigma_2\le {3}/{2}$ and $n\in \N_+$ be fixed constants such that $(\sigma_2-\sigma_1)n\ge 4$, which in particular indicates that $B_{\sigma_1 n}\neq B_{\sigma' n}$ with $\sigma'=\frac{\sigma_1+\sigma_2}{2}$.
Furthermore, we take a cut-off function $\eta_n:\Z^d \to [0,1]$ so that
\begin{equation}\label{p3-1-2a}
\begin{split}
{\rm supp}[\eta_n]\subset B_{\sigma' n},\,\, \|\nabla \eta_n\|_\infty\le \frac{4}{(\sigma_2-\sigma_1)n},\,\, \eta_n(x)=1\ {\rm for\ all}\ x\in B_{\sigma_1 n}.
\end{split}
\end{equation}
Recall that $p_*:=\frac{p}{p-1}$ is the conjugate number of $p$ and $\rho:=\frac{d}{(d-2)+d/q}$. By $\frac{1}{p}+\frac{1}{q}<\frac{2}{d}$, we can find
a constant $\alpha>1$ such that $\alpha p_*=\rho$.

Let $\tilde \chi_n:=n^{-1}\chi=\left(n^{-1}\chi_1,\cdots,n^{-1}\chi_d\right)$. For simplicity of the notation, we use $\tilde \chi_n$ to denote any item
$n^{-1}\chi_{i}$ with $1\le i \le d$.
Applying the H\"older inequality,  for all $l>0$,
\begin{equation}\label{p3-1-2}
\begin{split}
\|(\tilde \chi_n-l)_+^2\|_{p_*,B_{\sigma_1 n}}&\le \|(\tilde \chi_n-l)_+^2\|_{\alpha p_*,B_{\sigma_1 n}}\|\I_{\{\tilde \chi_n> l\}}\|_{\alpha_* p_*,B_{\sigma_1 n}}\\
&=\|(\tilde \chi_n-l)_+^2\|_{\rho,B_{\sigma_1 n}}\|\I_{\{\tilde \chi_n> l\}}\|_{p_*,B_{\sigma_1 n}}^{{1}/{\alpha_*}}.
\end{split}
\end{equation}
According to the Sobolev inequality \eqref{l2-2-1} and  \eqref{p3-1-2a}, we obtain
\begin{equation}\label{p3-1-3}
\begin{split}
&\|(\tilde \chi_n-l)_+^2\|_{\rho,B_{\sigma_1 n}}\le \|((\tilde \chi_n-l)_+\eta_n)^2\|_{\rho,B_{\sigma_2 n}}\\
&\le c_1|B_{\sigma_2 n}|^{{2}/{d}}\|\nu\|_{q, B_{\sigma_2 n}}
\cdot\Bigg[\frac{1}{|B_{\sigma_2 n}|}\sum_{x,y\in B_{\sigma_2 n}}\big ((\tilde \chi_n(x)-l)_+-(\tilde \chi_n(y)-l)_+\big )^2\max\left\{\eta_n^2(x),\eta_n^2(y)\right\}C_{x,y}\\
&\quad +\frac{\|\nabla \eta_n\|_\infty^2}{|B_{\sigma_2 n}|} \cdot \bigg(\sum_{x\in B_{\sigma_2 n}}(\tilde \chi_n(x)-l)_+^2\mu(x)\bigg)\Bigg]\\
&\le c_1|B_{\sigma_2 n}|^{{2}/{d}}\|\nu\|_{q, B_{\sigma_2 n}}
\cdot\Bigg(\frac{1}{|B_{\sigma_2 n}|}\sum_{x,y\in B_{\sigma_2 n}}\big((\tilde \chi_n(x)-l)_+-(\tilde \chi_n(y)-l)_+\big)^2\max\left\{\eta_n^2(x),\eta_n^2(y)\right\}C_{x,y}\\
&\quad +\frac{c_2}{(\sigma_2-\sigma_1)^2n^2} \|\mu\|_{p,B_{\sigma_2 n}} \|(\tilde \chi_n-l)_+^2\|_{p_*,B_{\sigma_2 n}}\Bigg).
\end{split}
\end{equation}
Here in the second inequality we have used the property that for every $u:\Z^d \to \R$,
\begin{align*}
\tilde \sD_{\Z^d;\eta_n^2}(u)
\le \sum_{x,y\in B_{\sigma_2 n}}(u(x)-u(y))^2\max\left\{\eta_n^2(x),\eta_n^2(y)\right\}C_{x,y},
\end{align*}
which can be verified directly by the fact ${\rm supp}[\eta_n]\subset B_{\sigma' n}$
and $B_{\sigma'n}\neq B_{\sigma_2 n}$, and the last inequality follows from \eqref{p3-1-2a} and the H\"older inequality.

Note that $\tilde \chi_n$ satisfies \eqref{l3-1-0} with $f(x)={x}/{n}$, which means that
$\|\nabla f\|_\infty \le {1}/{n}$. So, by \eqref{l3-1-1}, it holds that
\begin{equation}\label{p3-1-4}
\begin{split}
&\sum_{x,y\in B_{\sigma_2 n}}\Big((\tilde \chi_n(x)-l)_+-(\tilde \chi_n(y)-l)_+\Big)^2\max\left\{\eta_n^2(x),\eta_n^2(y)\right\}C_{x,y}\\
&\le \frac{c_3}{(\sigma_2-\sigma_1)^2n^2}\sum_{x\in B_{\sigma_2 n}}(\tilde \chi_n(x)-l)_+^2\mu(x)+\frac{c_3}{n^2}\sum_{x\in B_{\sigma_2 n}}\I_{\{\tilde \chi_n>l\}}(x)\mu(x)
\\
&\quad +\frac{c_3}{(\sigma_2-\sigma_1)n^2}\sum_{x\in B_{\sigma_2 n}}(\tilde \chi_n(x)-l)_+\mu(x)\\
&\quad +c_3\sum_{x\in B_{\sigma_2 n}}(\tilde \chi_n(x)-l)_+\eta_n^2(x)\left(\sum_{y\in B_{\sigma_2 n}^c}\left((\tilde \chi_n(y)-l)_++n^{-1}|x-y|\right)C_{x,y}\right).
\end{split}
\end{equation}
Since ${\rm supp}[\eta_n] \subset B_{\sigma' n}$, for every $x\in B_{\sigma_2 n}$ we have
\begin{align*}
\eta_n^2(x)\sum_{y\in B_{\sigma_2 n}^c}|x-y|C_{x,y}&\le \sum_{y\in \Z^d: |y-x|\ge \frac{(\sigma_2-\sigma_1)n}{2}}|y-x|C_{x,y}\\
&\le \frac{2}{(\sigma_2-\sigma_1)n}\sum_{y\in \Z^d}|y-x|^2 C_{x,y}\le \frac{2\mu(x)}{(\sigma_2-\sigma_1)n}.
\end{align*}
Meanwhile, for every $0<k<l$,
\begin{equation}\label{p3-1-5}
\begin{split}
&\|\I_{\{\tilde \chi_n>l\}}\|_{p_*,B_{\sigma_2 n}}=
\|\I_{\{\tilde \chi_n-k>l-k\}}\|_{p_*,B_{\sigma_2 n}}\le \frac{\|(\tilde \chi_n-k)_+^2\|_{p_*,B_{\sigma_2 n}}}{(l-k)^2},\\
&\|(\tilde \chi_n-l)_+\|_{p_*,B_{\sigma_2 n}}=\|(\tilde \chi_n-l)_+\I_{\{\tilde \chi_n-k>l-k\}}\|_{p_*,B_{\sigma_2 n}}\le \frac{\|(\tilde \chi_n-k)_+^2\|_{p_*,B_{\sigma_2 n}}}{l-k}.
\end{split}
\end{equation}
Putting all the estimates above into \eqref{p3-1-4} and using the H\"older inequality, we obtain that
\begin{align*}
&\sum_{x,y\in B_{\sigma_2 n}}\big((\tilde \chi_n(x)-l)_+-(\tilde \chi_n(y)-l)_+\big)^2\max\left\{\eta_n^2(x),\eta_n^2(y)\right\}C_{x,y}\\
&\le \frac{c_4|B_{\sigma_2 n}|}{(\sigma_2-\sigma_1)^2n^2}\left(\|\mu\|_{p,B_{\sigma_2 n}}\left(1+\frac{1}{l-k}+\frac{1}{(l-k)^2}\right)+\frac{\|{\rm Tail}(|\tilde \chi_n|,\sigma_2 n)\|_{p,B_{\sigma' n}}}{l-k}\right)\|(\tilde \chi_n-k)_+^2\|_{p_*,B_{\sigma_2 n}},
\end{align*}
where we have used the facts $\sigma_2-\sigma_1\le 1$, ${\rm supp}[\eta_n] \subset B_{\sigma' n}$ and
\begin{align*}
{\rm Tail}(|\tilde \chi_n|,\sigma_2n)(x):=\sigma_2^2n^2\sum_{y\in B^c_{\sigma_2 n}}|\tilde \chi_n(y)|C_{x,y},\quad x\in \Z^d.
\end{align*}
Combining this with \eqref{p3-1-2}, \eqref{p3-1-3} and \eqref{p3-1-5} yields that
\begin{equation}\label{p3-1-6}
\begin{split}
 \|(\tilde \chi_n&-l)_+^2\|_{p_*,B_{\sigma_1 n}}\le  \frac{c_5\|\max\{\mu,1\}\|_{p,B_{\sigma_2 n}}\|\max\{\nu,1\}\|_{q,B_{\sigma_2 n}}}{(\sigma_2-\sigma_1)^2}\\
&\times\left(\left(1+\frac{1}{l-k}+\frac{1}{(l-k)^2}\right)+\frac{\|{\rm Tail}(|\tilde \chi_n|,\sigma_2n)\|_{p,B_{\sigma' n}}}{l-k}\right)\frac{\|(\tilde \chi_n-k)_+^2\|_{p_*,B_{\sigma_2 n}}^{1+{1}/{\alpha_*}}}{(l-k)^{{2}/{\alpha_*}}}.
\end{split}
\end{equation}

{\bf Step 2.}\quad
Now we set $\sigma_j:=1+2^{-j}$, $\sigma_j':=\frac{\sigma_j+\sigma_{j+1}}{2}$, $k_j:=K(1-2^{-j})$ for every $j\ge 0$, where
$K$ is a positive constant to be determined later.
Define
$$\zeta_n(k,\sigma):=\|(\tilde \chi_n-k)_+^2\|_{p_*,B_{\sigma n}},\quad k>0,\ \sigma>0.$$
Let $J:=\left[\frac{d \log n}{(2+(4+\gamma_0)\alpha_*) \log 2}\right]$ with $\gamma_0=\frac{4p+dp-d}{2p}.$
Then, for $n$ large enough, $J\ge 1$ and
$$
(\sigma_{j-1}-\sigma_j)n=2^{-j}n\ge
2^{-1}n^{1-\frac{d}{2(1+2\alpha_*)}}
\ge 4,\quad j=1,\cdots,
J+1,
$$ where we used the fact that $\alpha=\rho/p_*\le \rho=d/(d-2-d/q)\le d/(d-2)$ and so
$\alpha_*\ge \left(\frac{d}{d-2}\right)_*=d/2$.
(Recall that $\alpha_*$ is the conjugate number of $\alpha$.)
Hence, for $n$ large enough, we can apply \eqref{p3-1-6} with $k=k_{j+1}$, $l=k_j$, $\sigma_1=\sigma_{j+1}$, $\sigma_2=\sigma_j$ to derive that
for every $0\le j\le J$,
\begin{equation}\label{p3-1-7}
\begin{split}
 \zeta_n\left(k_{j+1},\sigma_{j+1}\right)  \le & c_5M_n 2^{2(j+1)}\left(\frac{2^{j+1}}{K}\right)^{{2}/{\alpha_*}}\\
  &\times
\left(1+\frac{2^{j+1}}{K}\left(1+\|{\rm Tail}(|\tilde \chi_n|, \sigma_jn)\|_{p,B_{\sigma_j'n}}\right)+\left(\frac{2^{j+1}}{K}\right)^2\right)
\zeta_n\left(k_j,\sigma_j\right)^{1+{1}/{\alpha_*}},
\end{split}
\end{equation}
where $M_n:=\|\max\{\mu,1\}\|_{p,B_{2 n}}\cdot\|\max\{\nu,1\}\|_{q,B_{2 n}}$.

According to \eqref{l2-1-1}, there exists
$c_6>0$
(which may be random, i.e.,
may
depend on $\w$)
such that
\begin{equation}\label{p3-1-7a}
\sum_{x\in B_R}|\tilde \chi_n(x)|^{\frac{2p}{p-1}}\le \frac{c_6R^{d+\frac{2p}{p-1}}}{n^{\frac{2p}{p-1}}},\quad n\ge 1,\ R\ge 1.
\end{equation}
(Note that from here to the end the constants may depend on $\w$.)
For any $0\le j\le J$ and $n\ge 1$, we can find $m_0\in \N_+$ such that
$2^{m_0-1}<\sigma_j n\le 2^{m_0}$.
For every $x\in B_{\sigma'_j n}$ it holds that for every $\gamma:=\frac{4p+dp-d}{p+1}$,
\begin{align*}
&{\rm Tail}(|\tilde \chi_n|, \sigma_jn)(x)\\
&=\sigma_j^2n^2
\sum_{y\in B_{\sigma_j n}^c}|\tilde \chi_n(y)|C_{x,y}\\
&\le 4n^2
\left(\sum_{y\in B_{2^{m_0}}\backslash B_{\sigma_j n}}|\tilde \chi_n(y)|C_{x,y}+\sum_{m=0}^\infty\sum_{y\in B_{2^{m_0+m+1}}\backslash B_{2^{m_0+m}}}|\tilde \chi_n(y)|C_{x,y}\right)\\
&\le 4n^2
\Bigg[\Big(\sum_{y\in B_{2^{m_0}}\backslash B_{\sigma_j n}}|\tilde \chi_n(y)|^{\frac{2p}{p-1}}\Big)^{\frac{p-1}{2p}}
\Big(\sum_{y\in B_{2^{m_0}}\backslash B_{\sigma_j n}}C_{x,y}^{\frac{2p}{p+1}}\Big)^{\frac{p+1}{2p}}\\
&\qquad\quad +\sum_{m=0}^\infty \Big(\sum_{y\in B_{2^{m_0+m+1}}\backslash B_{2^{m_0+m}}}|\tilde \chi_n(y)|^{\frac{2p}{p-1}}\Big)^{\frac{p-1}{2p}}
\Big(\sum_{y\in B_{2^{m_0+m+1}}\backslash B_{2^{m_0+m}}}C_{x,y}^{\frac{2p}{p+1}}\Big)^{\frac{p+1}{2p}}\Bigg]\\
&\le c_72^{\frac{j\gamma(p+1)}{2p}}n^{1-\frac{\gamma(p+1)}{2p}}\Big(2^{m_0\left(1+\frac{d(p-1)}{2p}\right)}\sum_{m=0}^\infty 2^{m\left(1+\frac{d(p-1)}{2p}-\frac{\gamma(p+1)}{2p}\right)}\Big)\cdot
\Big(\sum_{y\in \Z^d}|x-y|^\gamma C_{x,y}^{\frac{2p}{p+1}}\Big)^{\frac{p+1}{2p}}\\
&\le c_82^{\frac{j\gamma(p+1)}{2p}}n^{2+\frac{d(p-1)}{2p}-\frac{\gamma(p+1)}{2p}}\mu_*(x)^{\frac{p+1}{2p}}\le c_92^{j\gamma_0}\mu_*(x)^{\frac{p+1}{2p}},
\end{align*}
where $\gamma_0:=\frac{\gamma(p+1)}{2p}$
and
$\mu_*(x,\w):= \tilde \mu_*(\tau_x \w)
=\sum_{z\in \Z^d}|z|^{\frac{4p+dp-d}{p+1}}C_{x,x+z}^{\frac{2p}{p+1}}(\w)$.
In the third inequality we have used \eqref{p3-1-7a} and the facts that
\begin{align*}
|y-x|&\ge |y|-|x| \ge 2^{m_0+m}-\sigma_j'n \ge \sigma_j n 2^m-\sigma_j'n\\
&\ge 2^m(\sigma_j n-\sigma_j' n)\ge c_{10}2^{m-j}n,\quad y\in B_{2^{m_0+m+1}}\backslash B_{2^{m_0+m}},\ x\in B_{\sigma_j'n},\\
|y-x|&\ge |y|-|x| \ge 2^{m_0}-\sigma_j'n \ge \sigma_j n -\sigma_j'n\\
&\ge (\sigma_j-\sigma_j') n\ge 2^{-j-2}n,\quad y\in B_{2^{m_0}}\backslash B_{\sigma_j n},\ x\in B_{\sigma_j'n},
\end{align*}
and,
for every $x\in B_{\sigma_j'n}$,
\begin{align*}
\sum_{y\in B_{2^{m_0+m+1}}\backslash B_{2^{m_0+m}}}C_{x,y}^{\frac{2p}{p+1}}&\le
c_{10}^{-\gamma}2^{\gamma(j-m)}n^{-\gamma}\sum_{y\in B_{2^{m_0+m+1}}\backslash B_{2^{m_0+m}}}|x-y|^\gamma C_{x,y}^{\frac{2p}{p+1}}\\
&\le c_{10}^{-\gamma}2^{\gamma(j-m)}n^{-\gamma}\sum_{y\in \Z^d}|x-y|^\gamma C_{x,y}^{\frac{2p}{p+1}},
\end{align*}
\begin{align*}
\sum_{y\in B_{2^{m_0}}\backslash B_{\sigma_j n}}C_{x,y}^{\frac{2p}{p+1}}&\le
4^{\gamma}2^{\gamma j}n^{-\gamma}\sum_{y\in B_{2^{m_0+}}\backslash B_{\sigma_j n}}|x-y|^\gamma C_{x,y}^{\frac{2p}{p+1}}\\
&\le 4^{\gamma}2^{\gamma j}n^{-\gamma}\sum_{y\in \Z^d}|x-y|^\gamma C_{x,y}^{\frac{2p}{p+1}};
\end{align*}
the fourth and fifth inequalities follow from the fact that $\frac{\gamma(p+1)}{2p}>1+\frac{d(p-1)}{2p}$ and $2+\frac{d(p-1)}{2p}=\frac{\gamma(p+1)}{2p}$ respectively.
In particular, we have
\begin{align*}
\|{\rm Tail}(|\tilde \chi_n|, \sigma_jn)\|_{p,B_{\sigma_j'n}}\le c_{11}2^{j\gamma_0}\|
\mu_*
\|_{\frac{p+1}{2},B_{2n}}^{\frac{p+1}{2p}}.
\end{align*}

Putting the estimate into \eqref{p3-1-7} again yields that
\begin{align*}
&\zeta_n\left(k_{j+1},\sigma_{j+1}\right)\\
&\le c_{12}M_n 2^{(2+\gamma_0)(j+1)}\left(\frac{2^{j+1}}{K}\right)^{{2}/{\alpha_*}}
\left(1+\frac{2^{j+1}}{K}\left(1+\|\mu_*\|_{\frac{p+1}{2},B_{2 n}}^{\frac{p+1}{2p}}\right)+\left(\frac{2^{j+1}}{K}\right)^2\right)
\zeta_n\left(k_j,\sigma_j\right)^{1+{1}/{\alpha_*}}.
\end{align*}
Note that, by \eqref{a2-1-1} and $  \mu_*\in L^{\frac{p+1}{2}}(\Omega;\Pp)$,  $\sup_{n\ge 1}(M_n+\|\mu_*\|_{\frac{p+1}{2},B_{2 n}})<\infty$.
Hence we have
\begin{equation}\label{p3-1-8}
\quad\zeta_n\left(k_{j+1},\sigma_{j+1}\right)\le c_{13} 2^{(2+\gamma_0)(j+1)}\left(\frac{2^{j+1}}{K}\right)
^{2/\alpha_*}
\left(1+\frac{2^{j+1}}{K}+\left(\frac{2^{j+1}}{K}\right)^2\right)
\zeta_n\left(k_j,\sigma_j\right)^{1+1/\alpha_*}
\end{equation} for all $0\le j \le J.$
For any fixed
$\delta\in (0,1]$,
choosing $K=\delta+c_{14} {\delta} ^{-{\alpha_*} }\zeta_n(0,2)^{{1}/{2}}$ for large constant $c_{14}>0$ in
\eqref{p3-1-8}, we get
$$
\zeta_n\left(k_{j+1},\sigma_{j+1}\right)\le c_{13}2^{2+\gamma_0}\left(\frac{2}{K}\right)^{2/\alpha_*}
\left(\frac{
4}{\delta}\right)^22^{(4+\gamma_0+2/\alpha_*)j}\zeta_n(k_j,\sigma_j)^{1+1/\alpha_*},\quad 0\le j \le J.
$$
Hence, we can apply Lemma \ref{add-lemma} with $\beta=1/\alpha_*$, $b=2^{4+\gamma_0+2/\alpha_*}$ and $c_0=
c_{13}2^{2+\gamma_0}\left(\frac{2}{K}\right)^{2/\alpha_*}
\left(\frac{4}{\delta}\right)^2$ to obtain
\begin{align}\label{p3-1-9}
\zeta_n\left(k_{j},\sigma_{j}\right)
\le \frac{\zeta_n(0,2)}{2^{(2+(4+\gamma_0)\alpha_*)j}},\quad 0\le j \le J.
\end{align}

By the definition of $J$, it holds that  $\frac{n^d}{2^{(2+(4+\gamma_0)\alpha_*)J}}\le 1$, and so
\begin{align*}
\max_{x\in B_{n}}(\tilde \chi_n-k_J)_+^2&\le |B_{n}|\|(\tilde \chi_n-k_J)_+^2\|_{1,B_{n}}\le c_{15}n^d
\|(\tilde \chi_n-k_J)_+^2\|_{p_*,B_{\sigma_J n}}\\
&\le c_{16}\frac{n^d\zeta_n(0,2)}{2^{(2+(4+\gamma_0)\alpha_*)J}}\le c_{17}\|(\tilde \chi_n)_+^2\|_{p_*,B_{2 n}}
\le c_{17}\|\tilde \chi_n\|_{2p_*,B_{2 n}}^2.
\end{align*}
Applying the same argument to the function $-\tilde \chi_n$, we can obtain
\begin{align*}
\max_{x\in B_{ n}}
(\tilde \chi_n-k_J)_-^2\le c_{18}\|\tilde \chi_n\|_{2p_*,B_{2 n}}^2.
\end{align*}
Therefore,
\begin{align*}
\frac{1}{n}\sup_{x\in B_n}|\chi(x)|&\le \sup_{x\in B_{n}}|\tilde \chi_n(x)|\le
\max_{x\in B_{ n}}(\tilde \chi_n-k_J)_++\max_{x\in B_{ n}}(\tilde \chi_n-k_J)_-+k_J\\
&\le \max_{x\in B_{ n}}(\tilde \chi_n-k_J)_++\max_{x\in B_{ n}}(\tilde \chi_n-k_J)_-+K\\
&\le \delta+c_{19}\left(\frac{1}{\delta}\right)^{{\alpha_*}/{2}}\|\tilde \chi_n\|_{2p_*,B_{2 n}}.
\end{align*}
According to \eqref{l2-1-1},
by letting $n \to \infty$ and then $\delta\downarrow 0$, we obtain the desired result.
\end{proof}

\subsection{Maximal inequalities for caloric functions}
The aim of this part is to establish maximal inequalities for caloric functions associated with the operator $\sL$.

\begin{lemma}\label{l3-2}
Given any $R\in \mathbb{N}_+$ and $s_1<s_2$, let
$\eta:\Z^d \to [0,1]$ be such that ${\rm supp}[\eta]\subset B_R$, and $\vartheta:\R \to [0,1]$ be such that
$\vartheta(s)=0$ for all $s\in (-\infty,s_1]$.
Suppose that $u:[s_1,s_2]\times \Z^d \to \R$ satisfies the following equation
\begin{equation}\label{l3-2-0}
\partial_t u(t,x)-\sL u(t,\cdot)(x)=0,\quad (t,x)\in [s_1,s_2]\times B_R.
\end{equation}
Then, there is a constant
$C_2>0$
so that for all $k>0$,
\begin{equation}\label{l3-2-1}
\begin{split}
&\sup_{t\in [s_1,s_2]}\sum_{x\in B_R}\vartheta(t)\eta^2(x)u_{-k,+}(t,x)^2\\
&+\int_{s_1}^{s_2}\vartheta(t)\sum_{x,y\in B_R}\left(u_{-k,+}(t,x)-u_{-k,+}(t,y)\right)^2\max\{\eta^2(x),\eta^2(y)\}C_{x,y}\,dt\\
&\le C_2 \|\nabla \eta\|_\infty^2 \int_{s_1}^{s_2} \sum_{x\in B_R} u_{-k,+}^2(t,x)\mu(x) \,dt+C_2\|\vartheta'\|_\infty \int_{s_1}^{s_2} \sum_{x\in B_R} u_{-k,+}^2(t,x)\, dt\\
&\quad + C_2\int_{s_1}^{s_2}\vartheta(t)\sum_{x\in B_R}u_{-k,+}(t,x)\eta^2(x)\sum_{y\in B_R^c}u_{-k,+}(t,y)C_{x,y}\,dt,
\end{split}
\end{equation}
where  $u_{-k,+}(t,x)=(u(t,x)-k)_+$.
\end{lemma}
\begin{proof}
Multiplying both sides of \eqref{l3-2-0} with $u_{-k,+}(t,x)\eta^2(x)\vartheta(t)$, we have
\begin{equation}\label{l3-2-2}
\int_{s_1}^t \sum_{x\in B_R}\partial_s u(s,x) u_{-k,+}(s,x)\eta^2(x)\vartheta(s) \,ds=\int_{s_1}^t  \sum_{x\in B_R}
\sL u(s,\cdot)(x)u_{-k,+}(s,x)\eta^2(x) \vartheta(s)\, ds.
\end{equation}
According to the arguments for the estimates of $I_1$ and $I_2$ in the proof of Lemma \ref{l3-1}, we have immediately that
for every $s\in [s_1,s_2]$,
\begin{align*}
\sum_{x\in B_R}
\sL u(s,\cdot)(x)u_{-k,+}(s,x)\eta^2(x)&\le -\frac{1}{2}\sum_{x,y\in B_R}\left(u_{-k,+}(t,x)-u_{-k,+}(t,y)\right)^2\max\{\eta^2(x),\eta^2(y)\}C_{x,y}\\
&\quad +2\sum_{x\in B_R}u_{-k,+}(s,x)\eta^2(x)\left(\sum_{y\in B_R^c}u_{-k,+}(s,y)C_{x,y}\right)\\
&\quad+2c_1\|\nabla \eta\|^2_\infty\sum_{x\in B_R}u_{-k,+}(x)^2\mu(x).
\end{align*}
Noting that $\frac{\partial}{\partial s}\left(u_{-k,+}^2(s,x)\right)=2u_{-k,+}(s,x)\partial_s u(s,x) $,
it holds that
\begin{align*}
2\partial_s u(s,x) u_{-k,+}(s,x)\vartheta(s)&=
\frac{\partial}{\partial s}\left(u_{-k,+}^2(s,x)\vartheta(s)\right)-u_{-k,+}^2(s,x)\vartheta'(s),
\end{align*}
which yields
\begin{align*}
& 2\int_{s_1}^t\left(\sum_{x\in B_R}\partial_s u(s,x) u_{-k,+}(s,x)\vartheta(s)\eta^2(x)\right)ds\\
&=\int_{s_1}^t \frac{\partial}{\partial s}\left(\sum_{x\in B_R}u_{-k,+}^2(s,x)\vartheta(s)\eta^2(x)\right)ds-
\int_{s_1}^t \vartheta'(s)\left(\sum_{x\in B_R}u_{-k,+}^2(s,x)\eta^2(x)\right)ds\\
&=\sum_{x\in B_R}u_{-k,+}^2(t,x)\vartheta(t)\eta^2(x)-\sum_{x\in B_R}u_{-k,+}^2(s_1,x)\vartheta(s_1)\eta^2(x)-\int_{s_1}^t \vartheta'(s)
\left(\sum_{x\in B_R}u_{-k,+}^2(s,x)\eta^2(x)\right)ds\\
&\ge \vartheta(t)\sum_{x\in B_R}u_{-k,+}^2(t,x)\eta^2(x)-\|\vartheta'\|_\infty
\int_{s_1}^t
\left(\sum_{x\in B_R}u_{-k,+}^2(s,x)\eta^2(x)\right)ds.
\end{align*}
Here in the last inequality we used the fact $\vartheta(s_1)=0$.

Putting all the estimates above together yields the desired conclusion \eqref{l3-2-1}.
\end{proof}

Now, we are going to prove the maximal inequality for
caloric
functions. For any
$m>0$,  define
$$
\tilde \mu_m(\w)
:=\sum_{z\in \Z^d}|z|^mC_{0,z}(\w),\quad
\mu_m(x,\w):=\tilde \mu_m(\tau_x \w).$$
Obviously, $\tilde \mu_m=\tilde \mu$ when $m=2$. Recall that $Q_{s,R}:=Q_{0,0,s,R}=[-s,0]\times B_R(0)$  for all $s>0$ and $R\ge1$.

\begin{theorem}\label{t3-2}
Suppose that the condition \eqref{a2-1-1} holds with some
$p,q\in (1,+\infty]$
satisfying \eqref{l2-6-0},
and that there exists $m\ge 2$ such that
$\tilde \mu_m \in L^p(\Omega;\Pp)$.
Then, there exists a random variable
$N_0:=N_0(\w)>0$
such that for every $n>N_0$, if $u:(-4n^2,0]\times \Z^d \to \R$ is bounded and satisfies the following equation
\begin{equation}\label{t3-2-0}
 \partial_t u(t,x) -\sL u(t,\cdot)(x)=0,\quad (t,x)\in Q_{4n^2,2n},
\end{equation}
then
we can find constants $C_3, \kappa>0$ independent of $n$
so that for every $\delta\in (0,2)$ and $\frac{1}{2}\le \theta'< \theta< 1$,
\begin{equation}\label{t3-2-1}
\begin{split}
\max_{(t,x)\in Q_{\theta'n^2,\theta' n}}|u(t,x)|\le \delta n^{-(m-2)}
\|u\|_{\infty,\infty,[-n^2,0]\times\Z^d}
+C_3
\left(\frac{M_n}{\delta(\theta-\theta')^{m+3}}\right)^\kappa
\|u\|_{2p_*,2,Q_{\theta n^2,\theta n}},
\end{split}
\end{equation}
where
$\|u\|_{\infty,\infty,[-n^2,0]\times\Z^d}
:=\sup_{(s,x)\in (-n^2,0]\times \Z^d}|u(s,x)|$ and $$M_n:=\|\max\{\mu,1\}\|_{p,B_{\theta n}}\cdot
\|\max\{
\mu_m
,1\}\|_{p,B_{\theta n}}\cdot\|\max\{\nu,1\}\|_{q,B_{\theta n}}.$$
\end{theorem}
\begin{proof} (i)
As in the proof of Theorem \ref{p3-1}, let ${1}/{2}\le \sigma_1<\sigma_2\le 1$ and $n\in \N_+$ be fixed constants, and
set $\sigma':=\frac{\sigma_1+\sigma_2}{2}$.
Take $\eta_n:\Z^d \to [0,1]$
such that \eqref{p3-1-2a} holds
(note that for this choice the condition $(\sigma_2-\sigma_1)n\ge 4$ is not needed),
and $\vartheta_n:\R \to [0,1]$ such that
\begin{equation}\label{t3-2-1a}
\vartheta_n(s)=1 \text{ for all } s\in [-\sigma_1n^2, 0],\quad \vartheta_n(s)=0\text{ for all }  s\in (-\infty, -\sigma_2 n^2],\ \|\vartheta_n'\|\le \frac{2}{(\sigma_2-\sigma_1)n^2}.
\end{equation}

Recall that $\rho=\frac{d}{(d-2)+d/q}$ and it holds that $\rho>p_*$ due to the condition $1/p+1/q<2/d$.
Then, we can
find a constant $\alpha>1$ such that
$\alpha=1+\frac{1}{p_*}-\frac{1}{\rho}$.
Let $I_{\sigma n^2}:=[-\sigma n^2,0]$ and
$Q_{\sigma n}:=Q_{\sigma n^2,\sigma n}=I_{\sigma n^2}\times B_{\sigma n}$ for every
$\sigma>0$.
In the proof, we write $\|u\|_{\infty,\infty,[-n^2,0]\times\Z^d}$ as $\|u\|_{\infty,\infty}$ for simplicity of the notation.
Applying the H\"older inequality,
\begin{equation}\label{t3-2-2}
\begin{split}
\|(u-l)_+^2\|_{p_*,1,Q_{\sigma_1 n}}&\le \|(u-l)_+^2\|_{\alpha p_*,\alpha, Q_{\sigma_1 n}}\|\I_{\{u> l\}}\|_{\alpha_* p_*,\alpha_*,Q_{\sigma_1 n}}\\
&\le\left(\|(u-l)_+^2\|_{\rho,1, Q_{\sigma_1 n}}+\|(u-l)_+^2\|_{1,\infty,Q_{\sigma_1 n}}\right)\|\I_{\{u > l\}}\|_{p_*,1,Q_{\sigma_1 n}}^{\frac{1}{\alpha_*}},
\end{split}
\end{equation}
where in the last inequality we have used the interpolation inequality (due to the conditions that $\frac{1}{\alpha p_*}+\frac{1}{\alpha}(1-\frac{1}{\rho})=1$
and $1< \alpha p_*\le \rho$)
\begin{align*}
\|(u-l)_+^2\|_{\alpha p_*,\alpha, Q_{\sigma_1 n}}\le \|(u-l)_+^2\|_{\rho,1, Q_{\sigma_1 n}}+\|(u-l)_+^2\|_{1,\infty,Q_{\sigma_1 n}};
\end{align*}
see \cite[Lemma 2.8]{ACS}.
By \eqref{l2-6-1} we obtain
\begin{equation}\label{t3-2-3}
\begin{split}
&\|(u-l)_+^2\|_{\rho,1, Q_{\sigma_1 n}}\\
&\le \frac{c_1|B_{\sigma_1 n}|^{\frac{2}{d}}\|\nu\|_{q, B_{\sigma_1 n}}}{|[-\sigma_1 n^2,0 ]|}
\cdot\Bigg(\frac{1}{|B_{\sigma_1 n}|}\int^0_{-\sigma_1 n^2}\sum_{x,y\in B_{\sigma_1 n}}\Big((u(s,x)-l)_+-(u(s,y)-l)_+\Big)^2C_{x,y}ds\Bigg)\\
&\quad\quad +\|(u-l)_+^2\|_{p_*,1,Q_{\sigma_1 n}}\\
&\le \frac{c_1|B_{\sigma_1 n}|^{\frac{2}{d}}\|\nu\|_{q, B_{\sigma_1 n}}}{|[-\sigma_1 n^2,0 ]|}
\cdot\Bigg(\frac{1}{|B_{\sigma_1 n}|}
\int^0_{-\sigma_2 n^2}\vartheta_n(s)
\sum_{x,y\in B_{\sigma_2n}}\Big((u(s,x)-l)_+-(u(s,y)-l)_+\Big)^2\\
&\qquad\qquad\qquad\qquad\qquad\qquad\qquad\qquad\times \max\left\{\eta_n^2(x),\eta_n^2(y)\right\}C_{x,y}\,ds\Bigg)
+\|(u-l)_+^2\|_{p_*,1, Q_{\sigma_1 n}}.
\end{split}
\end{equation}
According to \eqref{l3-2-1},
it holds that for every $0<k<l$,
\begin{align*}
&\int_{-\sigma_2 n^2}^0\vartheta_n(s)\sum_{x,y\in B_{\sigma_2 n}}\Big((u(s,x)-l)_+-(u(s,y)-l)_+\Big)^2\max\left\{\eta_n^2(x),\eta_n^2(y)\right\}C_{x,y}\,ds\\
&\le \frac{c_2}{(\sigma_2-\sigma_1)^2n^2}\int_{-\sigma_2 n^2}^0\sum_{x\in B_{\sigma_2 n}}(u(s,x)-l)_+^2\max\{\mu(x),1\}\,ds\\
&+c_2\int_{-\sigma_2 n^2}^0\vartheta_n(s)\sum_{x\in B_{\sigma_2 n}}(u(s,x)-l)_+\eta_n^2(x)\left(\sum_{y\in B_{\sigma_2 n}^c}(u(s,y)-l)_+C_{x,y}\right)\,ds\\
&\le
\frac{c_3|B_{\sigma_2 n}|}
{(\sigma_2-\sigma_1)^2}\left(\|\max\{\mu,1\}\|_{p,\sigma_2 n}+
\frac{\sup_{s\in I_{\sigma_2 n^2}}\|{\rm Tail}(|u(s,\cdot)|,\sigma_2 n)\|_{p,\sigma' n}}{l-k}\right)\|(u(s,\cdot)-k)_+^2\|_{p_*,1,Q_{\sigma_2 n}},
\end{align*}
where in the last step we have used again the H\"older inequality and \eqref{p3-1-5}. Combining the inequality above with \eqref{t3-2-3}, we find that for every $0<k<l$,
\begin{align*}\|(u-l)_+^2\|_{\rho,1, Q_{\sigma_1 n}}\le &\frac{c_4\|\max\{\mu,1\}\|_{p,B_{\sigma_2 n}}\|\max\{\nu,1\}\|_{q,B_{\sigma_2 n}}}{(\sigma_2-\sigma_1)^2}\\
&\times \left(1+\frac{\sup_{s\in I_{\sigma_2 n^2}}\|{\rm Tail}(|u(s,\cdot)|,\sigma_2n)\|_{p,B_{\sigma' n}}}{l-k}\right)\|(u-k)_+^2\|_{p_*,1,Q_{\sigma_2 n}}.\end{align*}

On the other hand, applying \eqref{l3-2-1} and the H\"older inequality, we can deduce that
\begin{align*}
&\|(u-l)_+^2\|_{1,\infty,Q_{\sigma_1 n}}\\
&=\sup_{s\in I_{\sigma_1 n^2}}\|(u(s,\cdot)-l)_+^2\|_{1,B_{\sigma_1 n}}\\
&\le
\frac{c_5}
{(\sigma_2-\sigma_1)^2}\left(\|\max\{\mu,1\}\|_{p,\sigma_2 n}+
\frac{\sup_{s\in I_{\sigma_2 n^2}}\|{\rm Tail}(|u(s,\cdot)|,\sigma_2 n)\|_{p,\sigma' n}}{l-k}\right)\|(u(s,\cdot)-k)_+^2\|_{p_*,1,Q_{\sigma_2 n}}.
\end{align*}

Putting both estimates above together into \eqref{t3-2-2} and \eqref{t3-2-3} yields that for every $0<k<l$,
\begin{equation}\label{t3-2-4}
\begin{split}
\|(u-l)_+^2\|_{p_*,1,Q_{\sigma_1 n}}\le& \frac{c_4\|\max\{\mu,1\}\|_{p,B_{\sigma_2 n}}\|\max\{\nu,1\}\|_{q,B_{\sigma_2 n}}}{(\sigma_2-\sigma_1)^2}\\
&\times\left(1+\frac{\sup_{s\in I_{\sigma_2 n^2}}\|{\rm Tail}(|u(s,\cdot)|,\sigma_2n)\|_{p,B_{\sigma' n}}}{l-k}\right)\frac{\|(u-k)_+^2\|_{p_*,1,Q_{\sigma_2 n}}^{1+{1}/{\alpha_*}}}{(l-k)^{{2}/{\alpha_*}}},
\end{split}
\end{equation} where in the inequality we also used \eqref{p3-1-5}.

(ii) Given ${1}/{2}\le \theta'<\theta<1$,  define $\sigma_j:=\theta'+(\theta-\theta')2^{-j}$, $\sigma_j':=\frac{\sigma_j+\sigma_{j+1}}{2}$ and $k_j:=K(1-2^{-j})$ for every $j\ge 0$, where
$K$ is a positive constant to be determined later.
Let
$$\zeta_n(k,\sigma):=\|(u-k)_+^2\|_{p_*,1,Q_{\sigma n}},\quad k>0,\ \sigma>0.$$
Thus, applying \eqref{t3-2-4} with $k=k_{j+1}$, $l=k_j$, $\sigma_1=\sigma_{j+1}$ and $\sigma_2=\sigma_j$, we derive that
for every $j\ge 0$,
\begin{equation}\label{t3-2-5}
\begin{split}
&\zeta_n\left(k_{j+1},\sigma_{j+1}\right)\\
&\le c_5M_n \left(\frac{2^{j+1}}{\theta-\theta'}\right)^2\left(\frac{2^{j+1}}{K}\right)^{{2}/{\alpha_*}}
\left(1+\frac{2^{j+1}}{K}\sup_{s\in I_{\sigma_j n^2}}\|{\rm Tail}(|u(s,\cdot)|, \sigma_jn)\|_{p,B_{\sigma_j'n}}\right)
\zeta_n\left(k_j,\sigma_j\right)^{1+1/\alpha_*}.
 \end{split}
\end{equation}

Note that, for every $x\in B_{\sigma_j' n}$ and
$m\ge 2$,
it holds that
\begin{align*}
\sup_{s\in I_{\sigma_j n^2}}{\rm Tail}(|u(s,\cdot)|, \sigma_jn)(x)&\le n^2
\sum_{y\in B_{\sigma_j n}^c} C_{x,y}\\
&\le n^2
\|u\|_{\infty,\infty}
\sum_{y\in \Z^d: |y-x|\ge 2^{-(j+1)}n(\theta-\theta')}C_{x,y}\\
&\le \frac{2^{m(j+1)}}{(\theta-\theta')^mn^{m-2}}
\|u\|_{\infty,\infty}
\sum_{y\in \Z^d}|y-x|^m C_{x,y}\\
&=\frac{2^{m(j+1)}}{(\theta-\theta')^mn^{m-2}}
\|u\|_{\infty,\infty}
\tilde\mu^m(x).
\end{align*}
where in the second inequality we used the fact that $|y-x|\ge 2^{-(j+1)}n(\theta-\theta')$ for every $x\in B_{\sigma_j' n}$ and
$y \in B_{\sigma_j n}^c$
(note that this property still holds when $(\sigma_j-\sigma_j')n\le 1$).
The inequality above along with \eqref{t3-2-5} yields that
\begin{equation}\label{t3-2-6}
\begin{split}
\zeta_n\left(k_{j+1},\sigma_{j+1}\right)&\le c_{6}M_n \left(\frac{2^{j+1}}{\theta-\theta'}\right)^2\left(\frac{2^{j+1}}{K}\right)^{{2}/{\alpha_*}}
\left(1+\frac{2^{(m+1)(j+1)}
\|u\|_{\infty,\infty}}
{(\theta-\theta')^mn^{m-2}K}\right)
\zeta_n\left(k_j,\sigma_j\right)^{1+{1}/{\alpha_*}}\\
&\le c_{6}M_n \left(\frac{2^{j+1}}{\theta-\theta'}\right)^{m+3}\left(\frac{2^{j+1}}{K}\right)^{{2}/{\alpha_*}}
\left(1+\frac{\|u\|_{\infty,\infty}}{n^{m-2}K}\right)
\zeta_n\left(k_j,\sigma_j\right)^{1+{1}/{\alpha_*}},\quad j\ge 0.
\end{split}
\end{equation}

For any $\delta>0$, choosing
$K=\delta\|u\|_{\infty,\infty} n^{-(m-2)}+c_72^{m\alpha_*^2/2}(\delta(\theta-\theta')^{m+3})^{-\alpha_*/2} M_n^{\alpha_*/2}\zeta_n(0,\theta)^{{1}/{2}}$
with $c_7$ large enough
in
\eqref{t3-2-6} and using Lemma \ref{add-lemma}
(e.g.\ following the same arguments as these for \eqref{p3-1-9}), we can get that
\begin{align*}
\zeta_n\left(k_{j+1},\sigma_{j+1}\right)\le
\frac{\zeta_n(0,\theta)}{2^{(2+(m+3)\alpha_*)j}}
\le \frac{\zeta_n(0,\theta)}{2^{(2+5\alpha_*)j}},\quad j\ge 0.
\end{align*}
Therefore, letting $j \to \infty$, we obtain
\begin{align*}
\|(u-K)_+^2\|_{p_*,1,Q_{\theta' n}}&=\lim_{j \to \infty}\zeta_n\left(k_{j},\sigma_{j}\right)=0.
\end{align*}
This implies immediately that
\begin{align*}
\max_{(s,x)\in Q_{\theta' n}}(u(s,x)-K)_+= 0.
\end{align*}

Applying the same argument above to the function
$-(u(s,x)-k_j)$, we can also obtain
\begin{align*}
\max_{(s,x)\in Q_{\theta' n}}(u(s,x)-K)_-
= 0.
\end{align*}
Therefore, it holds that
\begin{align*}
 \sup_{(s,x)\in Q_{\theta' n}}|u(s,x)|&\le
\max_{(s,x)\in Q_{\theta' n}}(u(s,x)-K)_++\max_{(s,x)\in Q_{\theta' n}}(u(s,x)-K)_-+K\\
&\le \delta n^{-(m-2)}\|u\|_{\infty,\infty}+c_{8}2^{m\alpha_*^2/2}\left(\frac{M_n}{\delta(\theta-\theta')^{m+3} }\right)^{{\alpha_*}/{2}}\|u\|_{2p_*,2,Q_{\theta n}}.
\end{align*}
The proof is finished. \end{proof}

\begin{remark}\label{r4-1}
We note
that in  the proof for \eqref{t3-2-3},
the condition like
$(\sigma_2-\sigma_1)n>2$ is not required. This is different from the proof for \eqref{p3-1-3}, where
the condition that $(\sigma_2-\sigma_1)n\ge 4$ involved in is necessary. The main reason is that in \eqref{t3-2-3} we directly apply
the Sobolev inequality \eqref{l2-6-1} which holds for all $u:\Z^d\to \R$; however, for \eqref{p3-1-3}
we use the Sobolev inequality \eqref{l2-2-1} which only holds for any cut-off function
$\eta$ on $\Z^d$ such that $\eta=0$ on $\partial A$
(such condition holds for $\eta_n$ and $A=B_{\sigma_2 n}$ only when $(\sigma_2-\sigma_1)n>2$). We also stress that,
for the proof of \eqref{l3-2-1},
the condition like $(\sigma_2-\sigma_1)n>2$ is not needed either.
\end{remark}

As a consequence of  Theorem \ref{t3-2}, we have the following $L^1$ mean value inequality.
\begin{corollary}\label{c3-1}
Suppose that the conditions in Theorem $\ref{t3-2}$ hold. Then, there exist constants
$C_4,\kappa_1>0$
so that  for every $\frac{1}{2}\le \theta'< \theta< 1$ and
$n>N_0$ $($here
$N_0:=N_0(\w)>0$ is the same random variable in Theorem $\ref{t3-2}$$)$,
\begin{equation}\label{c3-1-1}
\begin{split}
\max_{(t,x)\in Q_{\theta'n^2,\theta' n}}|u(t,x)|\le C_4 \left(n^{-(m-2)}
\|u\|_{\infty,\infty,[-n^2,0]\times \Z^d}
+\left(\frac{M_n}{(\theta-\theta')^{m+3}}\right)^{\kappa_1}\|u\|_{1,1,Q_{\theta n^2,\theta n}}\right),
\end{split}
\end{equation}
where $\|u\|_{\infty,\infty,[-n^2,0]\times \Z^d}$ and $M_n$ are defined by the same way as those in Theorem $\ref{t3-2}$.
\end{corollary}
\begin{proof}
For every $\beta\in (0,1)$
\begin{align*}
\|u\|_{2p_*,2,Q_{\theta n^2,\theta n}}&\le \left(\max_{(t,x)\in Q_{\theta n^2,\theta n}}|u(t,x)|\right)^{1-\frac{1}{2p_*}}
\cdot \left(\frac{1}{\theta n^2}\int_{-\theta n^2}^0\left(\frac{1}{|B_{\theta n}|}\sum_{x\in B_{\theta n}}|u(t,x)|\right)dt\right)^{\frac{1}{2p_*}}\\
&\le \beta \left(\max_{(t,x)\in Q_{\theta n^2,\theta n}}|u(t,x)|\right)+
c_1\beta^{-(2p_*-1)}
\|u\|_{1,1,Q_{\theta n^2,\theta n}},
\end{align*}
where the first inequality is due to the H\"older inequality and the second inequality follows from Young's inequality.

Combining this with \eqref{t3-2-1} (by taking $\delta=1$) yields that for every $\beta\in (0,1)$ small enough and
$\frac{1}{2}\le \theta'<\theta\le 1$,
\begin{align*}
\max_{(t,x)\in Q_{\theta'n^2,\theta' n}}|u(t,x)|\le &
c_2\beta^{-(2p_*-1)}\left(\frac{M_n}{(\theta-\theta')^{m+3}}\right)^{2p_*\kappa}
\|u\|_{1,1,Q_{\theta n^2,\theta n}}
+ n^{-(m-2)}\|u\|_{\infty,\infty,[-n^2,0]\times \Z^d}\\
&+\beta \max_{(t,x)\in Q_{\theta n^2,\theta n}}|u(t,x)|
\end{align*}

Take $f(\theta)=\max_{(t,x)\in Q_{\theta n^2,\theta n}}|u(t,x)|$. Then, the desired conclusion follows from Lemma \ref{l2-5}.
\end{proof}

Finally, we discuss
on-diagonal upper estimates for $p(t,x,y)$.
\begin{proposition}\label{p3-2}
Suppose that the condition \eqref{a2-1-1} holds with some
$p,q\in (1,+\infty]$
satisfying \eqref{l2-6-0},
and that
there exists $m\ge d+2$ such that
$$
\tilde \mu_m(\w)
:=\sum_{z\in \Z^d}|z|^mC_{0,z}(\w)\in L^p(\Omega;\Pp).
$$
Then, there exist a
constant $C_5>0$ and a random variable
$T_0(\w)>0$ such that for every $t>T_0(\w)$,
\begin{equation}\label{p3-2-1}
\sup_{x\in B_{{\sqrt{t}}/{4}}}p(t,0,x)\le C_5t^{-d/2},\quad x\in B_{\sqrt{t}}.
\end{equation}
\end{proposition}
\begin{proof}
Fix $t>0$ and define $u_t:(-t,0)\times \Z^d \to \R_+$  by
$u_t(s,x):=p(t+s,0,x)$. Choosing $n=[\frac{\sqrt{t}}{2}]+1$ and applying
\eqref{c3-1-1}, we have
\begin{equation}\label{p3-2-2}
\begin{split}
\sup_{x\in B_{ {\sqrt{t}}/{4}}}p(t,0,x)&\le \max_{(s,x)\in Q_{\frac{n^2}{2},\frac{n}{2}}}|u_t(s,x)|\\
&\le C_1n^{-(m-2)}
\|u_t\|_{\infty,\infty,[-n^2,0]\times \Z^d}
+C_1M_n^{\kappa_1} \|u_t\|_{1,1,Q_{n^2, n}}.
\end{split}
\end{equation}

Using the property
$$\sum_{x\in B_n}u_t(s,x)\le \sum_{x\in \Z^d}p(t+s,0,x)\le 1, $$
we have immediately that $\|u_t\|_{1,1,Q_{n^2, n}}\le n^{-d}$.

On the other hand, by the ergodic theorem, there exists $N_0(\w)$ such that
\begin{align*}
\sup_{n\ge N_0(\w)}M_n(\w)\le c_2,
\end{align*}
where $c_2$ is a non-random positive constant.

Hence, putting these estimates into \eqref{p3-2-2} and using the fact that $u_t(s,x)\le 1$ for all $s\in (-t,0)$ and $x\in \Z^d$ yield that there is a random variable
$T_0(\w)>0$ such that for every $t\ge T_0(\w)$,
\begin{align*}
\sup_{x\in B_{ {\sqrt{t}}/{4}}} p(t,0,x)&\le c_4(n^{-(m-2)}+n^{-d})\le c_5n^{-d}\le c_6t^{-d/2},
\end{align*}
where we have used the fact $m\ge d+2$. The proof is complete.\end{proof}

\section{Parabolic weak Harnack inequalities}\label{section4}
In this section, we always suppose that
$u:\R\times \Z^d \to \R$ is non-negative
on $[-2R^2,2R^2]\times B_R$ for some $R\ge 2$ and satisfies that
\begin{equation}\label{l4-1-1}
\partial_t u(t,x)-\sL u(t,\cdot)(x)=0,\quad (t,x)\in
[-2R^2,2R^2] \times B_R.
\end{equation} Recall that ${\rm Tail}(u,R)$ is defined by \eqref{e3-1}.
Throughout this section, for every $t\in \R$ and $r>0$,
define
\begin{equation}\label{e4-1}
Q_r^-(t):=[t-r^2,t]\times B_r,\quad Q_r^+(t):=[t,t+r^2]\times B_r.
\end{equation}
In particular, we write $Q_r^-:=Q_r^-(0)$ and $Q_r^+:=Q_r^+(0)$ for simplicity of
notations.
\begin{lemma}\label{l4-1}
Given any $2\le r_0<r<R/2$ and $-R^2\le t_0<t_1<t_2\le R^2$, let
$\eta:\Z^d \to \R_+$
and $\vartheta\in C^\infty([t_0,t_2] ; \R_+)$
satisfy  that $\eta(x)=0$ for every
$x\in B_{r_0}^c$,
$\vartheta(t_0)=0$ and $\vartheta(t)=1$ for all $t\in [t_1,t_2]$. Then, there is a constant $C_1>0$ such that for every $l>0$ and $a>0$,
\begin{equation}\label{l4-1-2}
\begin{split}
&\int_{t_0}^{t_2}\vartheta(t)\sum_{x,y\in B_{r}}\eta(x)\eta(y)
\left[\left(\frac{u_{+a}(t,x)}{\eta(x)}\right)^{-l/2}-
\left(\frac{u_{+a}(t,y)}{\eta(y)}\right)^{-l/2}\right]^2C_{x,y}\,dt\\
&\quad +\sup_{t_1\le t\le t_2}\sum_{x\in B_{r}}\eta(x)^{l+2}|u_{+a}(t,x)|^{-l}\\
&\le C_1l\int_{t_0}^{t_2}\vartheta(t)\sum_{x\in B_{r}}\left(\frac{{\rm Tail}(u_{-}(t,\cdot),R)(x)}{aR^2}+\frac{\mu(x)}{(r-r_0)^2}\right)
u_{+a}(t,x)^{-l}\eta(x)^{l+2}\,dt\\
&\quad +C_1\int_{t_0}^{t_2}|\vartheta'(t)| \sum_{x\in B_{r}}u_{+a}(t,x)^{-l}\eta(x)^{l+2}\,dt\\
&\quad +C_1l(l+1)\int_{t_0}^{t_2}\vartheta(t)\sum_{x,y\in B_{r}}\left(\eta(x)-\eta(y)\right)^2
\left[\left(\frac{u_{+a}(t,x)}{\eta(x)}\right)^{-l}+
\left(\frac{u_{+a}(t,y)}{\eta(y)}\right)^{-l}\right]C_{x,y}\,dt,
\end{split}
\end{equation}
where $u_{+a}:=u+a$ and $u_{-}:=\max\{-u, 0\}$.
\end{lemma}
\begin{proof}
For any $q>1$, define $v(t,x):=u_{+a}(t,x)^{({1-q})/{2}}$ and $\vv(t,x):=u_{+a}(t,x)^{-q}\eta(x)^{q+1}\vartheta(t)$.
Testing  \eqref{l4-1-1} with $\vv$
and noting that $\sL u=\sL u_{+a}$, we obtain that for every $t\in [t_0,t_2]$
\begin{align*}
0&=\sum_{x\in B_{r}}\left(\partial_t
u_{+a}(t,x)
\vv(t,x)-
\sL u_{+a}(t,\cdot)(x)
\vv(t,x)\right)\\
&=\sum_{x\in B_{r}}
\partial_t u_{+a}(t,x)
\vv(t,x)+\frac{1}{2}\sum_{x,y\in B_{r}}
\left(u_{+a}(t,x)-u_{+a}(t,y)\right)
\left(\vv(t,x)-\vv(t,y)\right)C_{x,y}\\
&\quad +\sum_{x\in B_{r},y\in B_{r}^c}
\left(u_{+a}(t,x)-u_{+a}(t,y)\right)\left(\vv(t,x)-\vv(t,y)\right)C_{x,y}\\
&=:I_1(t)+I_2(t)+I_3(t).
\end{align*}

For any $t\in [t_1,t_2]$, it holds that, thanks to $\vartheta(t_0)=0$ and $\vartheta(t)=1$ for all $t\in [t_1,t_2]$,
\begin{align*}
\int_{t_0}^{t}I_1(s)\,ds&=\int_{t_0}^{t}\sum_{x\in B_{r}}
\partial_s u_{+a}(s,x)
u_{+a}(s,x)^{-q}\eta(x)^{q+1}\vartheta(s) \,ds\\
&=-\frac{1}{q-1}\int_{t_0}^{t}\sum_{x\in B_{r}}\partial_s (v(s,x)^2)\eta(x)^{q+1}\vartheta(s)\,ds\\
&=-\frac{1}{q-1}\sum_{x\in B_{r}}v(t,x)^2\eta(x)^{q+1}+\frac{1}{q-1}\int_{t_0}^{t}\sum_{x\in B_{r}}v^2(s,x)\eta(x)^{q+1}\vartheta'(s)\,ds.
\end{align*}

Since ${\rm supp}[\eta]\subset B_{r_0}$, we can write
\begin{align*}
 \int_{t_0}^{t}I_3(s)\,ds&=\int_{t_0}^{t}\sum_{x\in B_{r},y\in B_{r}^c}u_{+a}(s,x)\vv(s,x)C_{x,y}\,ds-
 \int_{t_0}^{t}\sum_{x\in B_{r},y\in B_{r}^c}u_{+a}(s,y)\vv(s,x)C_{x,y}\,ds\\
&=:I_{31}(t)+I_{32}(t).
\end{align*}
We have
\begin{align*}
I_{31}(t)&=\int_{t_0}^{t}\sum_{x\in B_{r}}v(s,x)^2\eta(x)^{q+1}\cdot\sum_{y\in B_{r}^c}C_{x,y}\vartheta(s)\,ds\\
&\le \frac{1}{(r-r_0)^2}\int_{t_0}^{t}\sum_{x\in B_{r}}v(s,x)^2\eta(x)^{q+1}\mu(x)\vartheta(s)\,ds,
\end{align*}
where in the last inequality we have used the fact that
\begin{align}\label{l4-1-3}
\sum_{y\in B_{r}^c}C_{x,y}\le \frac{1}{(r-r_0)^2}\sum_{y\in B_{r}^c}|x-y|^2 C_{x,y}\le \frac{\mu(x)}{(r-r_0)^2},\quad  x\in {\rm supp}[\eta] \subset B_{r_0}.
\end{align}
On the other hand, noting that $u\ge 0$ and $u_{+a}\ge a$ on $[-4r^2,4r^2]\times B_R$, we have
\begin{align*}
I_{32}(t)&\le \int_{t_0}^{t} \sum_{x\in B_{r}}\frac{v(s,x)^2}{u_{+a}(s,x)}\eta(x)^{q+1}\vartheta(s)
\cdot \sum_{y \in B_{R}^c}u_{-}(s,y)C_{x,y}\,ds\\
&\le \frac{1}{a R^2}\int_t^{t_2}\sum_{x\in B_{r}}v(s,x)^2{\rm Tail}(u_{-}(s,\cdot),R)(x)\eta(x)^{q+1}\vartheta(s)\,ds.
\end{align*}

Furthermore, for $I_2(t)$ we apply \eqref{l2-4-1} (with $\alpha=\eta(x)$, $\beta=\eta(y)$, $a=u_{+a}(x)$ and $b=u_{+a}(y)$) to obtain that
\begin{align*}
-\int_{t_0}^{t}I_2(s)\,ds&\ge \frac{1}{q-1}\int_{t_0}^t \sum_{x,y\in B_{r}}\eta(x)\eta(y)
\left[\left(\frac{u_{+a}(s,x)}{\eta(x)}\right)^{\frac{1-q}{2}}-\left(\frac{u_{+a}(s,y)}{\eta(y)}\right)^{\frac{1-q}{2}}\right]^2C_{x,y}\vartheta(s)\,ds\\
&\quad -c_{0,q}\int_{t_0}^t \sum_{x,y\in B_{r}}\left(\eta(x)-\eta(y)\right)^2\left[
\left(\frac{u_{+a}(s,x)}{\eta(x)}\right)^{1-q}+\left(\frac{u_{+a}(s,y)}{\eta(y)}\right)^{1-q}\right]C_{x,y}\vartheta(s)\,ds,
\end{align*}
where
$c_{0,q}$
is given in \eqref{l2-4-1}.

Putting all the estimates above together, we obtain that for every $t\in [t_1,t_2]$,
\begin{align*}
&\frac{1}{q-1}\left(\int_{t_0}^t \sum_{x,y\in B_{r}}\eta(x)\eta(y)
\left[\left(\frac{u_{+a}(s,x)}{\eta(x)}\right)^{\frac{1-q}{2}}-\left(\frac{u_{+a}(s,y)}{\eta(y)}\right)^{\frac{1-q}{2}}\right]^2C_{x,y}\vartheta(s)\,ds
+\sum_{x\in B_{r}}v(t,x)^2\eta(x)^{q+1}\right)\\
&\le  \int_{t_0}^{t}\sum_{x\in B_{r}}v(s,x)^2\eta(x)^{q+1}
\left(\frac{{\rm Tail}(u_{-}(s,\cdot),R)(x)}{aR^2}+\frac{\mu(x)}{(r-r_0)^2}\right)\vartheta(s)\,ds\\
&\quad +\frac{1}{q-1}\int_{t_0}^{t}\sum_{x\in B_{r}}v^2(s,x)\eta(x)^{q+1}\vartheta'(s)\,ds\\
&\quad +c_{0,q}\int_{t_0}^t \sum_{x,y\in B_{r}}\left(\eta(x)-\eta(y)\right)^2\left[
\left(\frac{u_{+a}(s,x)}{\eta(x)}\right)^{1-q}+\left(\frac{u_{+a}(s,y)}{\eta(y)}\right)^{1-q}\right]\vartheta(s)\,ds.
\end{align*}

Therefore, setting $l=q-1$, and setting $t=t_2$ and $t=\tilde t_0$ such that
 $$\sup_{t_1\le t \le t_2}\sum_{x\in B_{r}}v(t,x)^2\eta(x)^{q+1}=
 \sum_{x\in B_{r}}v(\tilde t_0,x)^2\eta(x)^{q+1}$$
 respectively in the inequality above, we
prove the desired assertion \eqref{l4-1-2}.
\end{proof}

\begin{lemma}\label{l4-3}
Suppose that
\eqref{a2-1-1} holds for some
$p,q\in (1,+\infty]$ with
\begin{equation}\label{eq:newp-q}
\left(1-\frac{1}{d}\right)\frac{1}{p}+\frac{1}{q}\le \frac{1}{d},\quad \frac{1}{p-1}+\frac{1}{q}<\frac{2}{d}.
\end{equation}
Then there exists a constant
$C_2>0$
such that for every $2\le r< R/4$, $t_0\in \R$,
$\gamma\in (0,1]$  and $l\in (0,1]$ satisfying that
$[t_0-4r^2,t_0]\subset [-R^2,R^2]$,
we have
\begin{equation}\label{l4-3-1}
\sup_{(t,x)\in Q^{-}_{r}(t_0)}
u_{+a}^{-1}(t,x)
\le \left(\frac{C_2M_0}{\gamma^2}\right)^{\frac{p_*(1+\kappa)}{(1+\kappa-p_*)l}}
\|u_{+a}^{-1}\|_{l,l,Q_{(1+\gamma)r}^{-}(t_0)},
\end{equation}
where $Q^{-}_{r}(t_0)$
is defined by \eqref{e4-1},
$u_{+a}:=u+a$ with $a\ge \left(\frac{2r}{R}\right)^2\sup_{t\in [t_0-4r^2,t_0]}\|{\rm Tail}\left(u_{-}(t,\cdot),R\right)\|_{p,B_{2r}}$, $M_0:=\sup_{k\ge 1}\|\max\{\mu,1\}\|_{p,B_{k}}\cdot
\|\max\{\nu,1\}\|_{q,B_{k}}$, and $\kappa:=\frac{\rho-1}{\rho}$ with
$\rho=\frac{d}{d-2+d/q}$.
\end{lemma}
\begin{proof}
Without loss of generality we only prove the conclusion for the case that $t_0=0$, since
the case $t_0\ne 0$
can be shown in the same way.

{\bf Step 1.}\quad For any $2\le  r_1<r_2\le
r<R/4$, choose $\eta:\Z^d \to \R_+$ and $\vartheta\in C^\infty([-r_2^2,0];\R_+)$
such that
\begin{equation}\label{l4-3-2}
\eta(x)=1 \hbox{ for }
x\in B_{r_1},
\,\, \eta(x)=0 \hbox{ for }
x\in B_{\frac{r_1+r_2}{2}}^c,\,\,\|\eta\|_\infty\le 1,\,\, \|\nabla \eta\|_\infty\le \frac{c_1}{r_2-r_1};
\end{equation}
\begin{align*}
\vartheta(t)=1 \hbox{ for } t\in [-r_1^2,0],\,\, \vartheta(-r_2^2)=0,\ \ \|\vartheta\|_\infty\le 1,\ \|\vartheta'\|_\infty\le \frac{c_1}{r_2^2-r_1^2}.
\end{align*}
Here, we emphasize that, similar to the proof of
Theorem \ref{t3-2}, we can find $\eta:\Z^d \to \R_+$ satisfying \eqref{l4-3-2} without the restriction
that $r_2-r_1>2$.
Applying \eqref{l4-1-2} with $\eta$ and $\vartheta$ above, we know that for every $l>0$,
\begin{align*}
& \int_{-r_1^2}^{0}\sum_{x,y\in B_{r_1}} (u_{+a}(t,x)^{-l/2}-
u_{+a}(t,y)^{-l/2})^2C_{x,y}\,dt
+\sup_{t\in [-r_1^2,0]}\sum_{x\in B_{r_1}}u_{+a}(t,x)^{-l}\\
&\le c_2l\int_{-r_2^2}^{0}\sum_{x\in B_{r_2}}\left(\frac{{\rm Tail}(u_{-}(t,\cdot),R)(x)}{aR^2}+\frac{\mu(x)}{(r_2-r_1)^2}\right)
u_{+a}(t,x)^{-l}\eta(x)^{l+2}\,dt\\
&\quad +\frac{c_2}{r_2^2-r_1^2}\int_{-r_2^2}^{0}\sum_{x\in B_{r_2}}u_{+a}(t,x)^{-l}\eta(x)^{l+2}\,dt\\
&\quad +c_2l(l+1)\int_{-r_2^2}^{0}\sum_{x,y\in B_{r_2}}\left(\eta(x)-\eta(y)\right)^2
\left[\left(\frac{u_{+a}(t,x)}{\eta(x)}\right)^{-l}+
\left(\frac{u_{+a}(t,y)}{\eta(y)}\right)^{-l}\right]C_{x,y}\,dt\\
&=:I_1+I_2+I_3.
\end{align*}
Below we set $a\ge \left(\frac{r_2}{R}\right)^2\sup_{t\in [-r_2^2,0]}\|{\rm Tail}\left(u_{-}(t,\cdot),R\right)\|_{p,B_{r_2}}$. By the H\"older inequality and the
fact ${\rm supp}[\eta] \subset B_{r_2}$, it holds that
\begin{align*}
I_1+I_2&\le c_3|B_{r_2}|\left(\frac{l}{(r_2-r_1)^2}+\frac{1}{r_2^2-r_1^2}\right)\int_{-r_2^2}^0 \left(1+\|\mu\|_{p,B_{r_2}}\right)\|u_{+a}(t,\cdot)^{-l}\|_{p_*,B_{r_2}}dt.
\end{align*}
On the other hand, according to \eqref{l4-3-2},
\begin{align*}
I_3&\le \frac{c_4l(l+1)}{(r_2-r_1)^2}\int_{-r_2^2}^{0}\sum_{x\in B_{r_2}}u_{+a}(t,x)^{-l}
\cdot \sum_{y\in B_{r_2}}|x-y|^2C_{x,y} \,dt\\
&\le \frac{c_4l(l+1)}{(r_2-r_1)^2}\int_{-r_2^2}^{0}\sum_{x\in B_{r_2}}u_{+a}(t,x)^{-l}\mu(x)\,dt\\
&\le \frac{c_4|B_{r_2}|l(l+1)}{(r_2-r_1)^2}\int_{-r_2^2}^{0}\|\mu\|_{p,B_{r_2}}\|u_{+a}(t,\cdot)^{-l}\|_{p_*,B_{r_2}}\,dt.
\end{align*}
Hence, combining
all the inequalities above yields
\begin{equation}\label{l4-3-3}
\begin{split}
&\int_{-r_1^2}^{0}\sum_{x,y\in B_{r_1}} (u_{+a}(t,x)^{-l/2}-u_{+a}(t,y)^{-l/2} )^2C_{x,y}\,dt
+\sup_{t\in [-r_1^2,0]}\sum_{x\in B_{r_1}}u_{+a}(t,x)^{-l}\\
&\le c_5|B_{r_2}|\left(\frac{1}{r_2^2-r_1^2}+\frac{
l(l+1)}{(r_2-r_1)^2}\right)\int_{-r_2^2}^0 \left(1+\|\mu\|_{p,B_{r_2}}\right)\|u_{+a}(t,\cdot)^{-l}\|_{p_*,B_{r_2}}\,dt.
\end{split}
\end{equation}

Furthermore, note that $\kappa=\frac{1}{\rho_*}=\frac{2}{d}-\frac{1}{q}$. Using the H\"older inequality and applying \eqref{l2-6-1} to $u_{+a}(t,\cdot)^{-l/2}$, we find that
\begin{align*}
&\int_{-r_1^2}^0\sum_{x\in B_{r_1}}u_{+a}(t,x)^{-l(1+\kappa)}\,dt\\
&\le \int_{-r_1^2}^0 \left(\sum_{x\in B_{r_1}}u_{+a}(t,x)^{-l\kappa\rho_*}\right)^{\frac{1}{\rho_*}}\cdot\left(\sum_{x\in B_{r_1}}u_{+a}(t,x)^{-l\rho}\right)^{\frac{1}{\rho}}\,dt\\
&\le \left(\sup_{t\in [-r_1^2,0]}\sum_{x\in B_{r_1}}|u_{+a}(t,x)|^{-l}\right)^{\frac{1}{\rho_*}}\\
&\,\,
\times \int_{-r_1^2}^0 \left(|B_{r_1}|^{{1}/{q}}\|\nu\|_{q,B_{r_1}}\sum_{x,y\in B_{r_1}}\left(u_{+a}(t,x)^{-l/2}-u_{+a}(t,y)^{-l/2}\right)^2C_{x,y}+
|B_{r_1}|^{1/\rho}\|u_{+a}(t,\cdot)^{-l}\|_{p_*,B_{r_1}}\right)\,dt\\
&\le c_6M_{r_2}^{1+\kappa}|B_{r_2}|K(l,r_1,r_2)^{\kappa}
\left(|B_{r_2}|^{{2}/{d}}K(l,r_1,r_2)+1\right)
\left(\int_{-r_2^2}^0 \|u_{+a}(t,\cdot)^{-l}\|_{p_*,B_{r_2}}dt\right)^{1+\kappa},
\end{align*}
where $M_{r_2}:=\|\max\{\mu,1\}\|_{p,B_{r_2}}\cdot
\|\max\{\nu,1\}\|_{q,B_{r_2}}$, $K(l,r_1,r_2):=\frac{l+1}{r_2^2-r_1^2}+\frac{
l(l+1)}{(r_2-r_1)^2}$,
and in the last inequality we used \eqref{l4-3-3}.
In particular, it holds that
\begin{equation}\label{l4-3-4}
\|u_{+a}^{-l}\|_{1+\kappa,1+\kappa,Q^{-}_{r_1}}
 \le c_7M_{r_2}\left(\frac{|B_{r_2}|}{|B_{r_1}|}
\left(|B_{r_2}|^{2/d}K(l,r_1,r_2)+1\right)\right)^{\frac{1}{1+\kappa}}\frac{r_2^2K(l,r_1,r_2)^{\frac{\kappa}{1+\kappa}}}{r_1^{\frac{2}{1+\kappa}}}
\|u_{+a}^{-l}\|_{p_*,p_*,Q^{-}_{r_2}},
\end{equation}
where $u_{+a}:=u+a$ with $a\ge \left(\frac{r_2}{R}\right)^2\sup_{t\in [-r_2^2,0]}\|{\rm Tail}\left(u_{-}(t,\cdot),R\right)\|_{p,B_{r_2}}$.

{\bf Step 2.}\,\, For any $l>0$ and $r\ge 1$,
set $\M(l,r):=\|u_{+a}^{-1}\|_{l,l,Q^{-}_{r}}$,
where $u_{+a}:=u+a$ with $a\ge \left(\frac{2r}{R}\right)^2\sup_{t\in [-4r^2,0]}\|{\rm Tail}\left(u_{-}(t,\cdot),R\right)\|_{p,B_{2r}}$.
Setting $\hat l:=l/p_*$,
it holds that $\theta l=(1+\kappa)\hat l$ with $\theta:=\frac{1+\kappa}{p_*}>1$ (thanks to $\frac{1}{p-1}+\frac{1}{q}<\frac{2}{d}$);
moreover, for all
$2\le r_1<r_2\le r$,
$$
\| u_{+a}^{-\hat l}\|_{p_*,p_*,Q^{-}_{r_2}}
=\|u_{+a}^{-1}\|_{\hat lp_*,\hat lp_*,Q^{-}_{r_2}}^{\hat l}=\| u_{+a}^{-1}\|_{l,l,Q_{r_2}^{-}}^{\hat l}$$ and
$$\|u_{+a}^{-\hat l}\|_{1+\kappa,1+\kappa,Q^{-}_{r_1}}=\| u_{+a}^{-1}\|_{\hat l(1+\kappa),\hat l(1+\kappa),Q^{-}_{r_1}}^{\hat l}=\| u_{+a}^{-1}\|_{\theta l,\theta l,Q^{-}_{r_1}}^{\hat l}.$$
It then follows from   \eqref{l4-3-4}
that for every $l>0$ and $1\le \delta_1<\delta_2\le 2$,
\begin{equation}\label{l4-3-5}
\M(\theta l, \delta_1 r)\le \left(\frac{c_8 M_0(l+1)^2}{(\delta_2-\delta_1)^2}\right)^{\frac{p_*}{l}}\M(l,\delta_2r),
\end{equation}
where $M_0:=\sup_{k\ge 1}M_k$
that is finite by the ergodic theorem and \eqref{a2-1-1}, and we used the fact that $\hat l\le l$.

Suppose that $r\ge 2$,
$\gamma\in (0,1]$ and $0<l \le 1$. Set $l_m:=l\theta^m$ and $\tilde r_m:=r+\frac{\gamma r}{2^m}$ for every
$m\in \mathbb{N}_+$. Thanks to \eqref{l4-3-5},
\begin{align*}
\M(l_{m+1},\tilde r_{m+1})\le \left(\frac{c_94^{m}(1+l_m)^2}{\gamma^2}\right)^{\frac{p_*}{l_m}}\M(l_m,\tilde r_m),\ m\ge 0,
\end{align*}
Therefore,
\begin{align*}
\sup_{(t,x)\in Q^{-1}_{r}}u_{+a}^{-1}(t,x)
&\le \lim_{m \to \infty} \M(l_{m},\tilde r_{m})\\
&\le \left(\prod_{m=0}^\infty\left(\frac{c_9M_04^m(1+l_m)^2}{\gamma^2}\right)^{\frac{p_*}{l_m}}\right)\M(l,(1+\gamma)r)\\
&\le  \left(\frac{C_0M_0}{\gamma^2}\right)^{\frac{p_*(1+\kappa)}{(1+\kappa-p_*)l}}\|
u_{+a}^{-1}\|_{l,l,Q^{-}_{(1+\gamma)r}},
\end{align*}
which completes the proof.
\end{proof}

\begin{lemma}\label{l4-2}
Given any $2\le r_0<r<R/2$ and $-R^2\le t_0<t_1<t_2\le R^2$,
let $\eta: \Z^d\to \R_+$ satisfy the same condition as that in Lemma $\ref{l4-1}$. Suppose that $\vartheta\in C^\infty([t_0,t_2]; \R_+)$ is decreasing and
satisfies
$\vartheta(t)=1$ for all $t\in [t_0,t_1]$ and $\vartheta(t_2)=0$.
Then, there is a constant
$C_3>0$
so that  for every
$l\in (0,1)$ and $a>0$,
\begin{equation}\label{l4-2-1}
\begin{split}
&(1-l)\int_{t_0}^{t_2}\vartheta(t)\sum_{x,y\in B_{r}}(u_{+a}(t,x)^{l/2}\eta(x)-
u_{+a}(t,y)^{l/2}\eta(y))^2C_{x,y}\, dt+\sup_{
t_0\le t\le t_1}\sum_{x\in B_{r}}u_{+a}(t,x)^l \eta(x)^2\\
&\le \frac{C_3(1+l^2-l)}{1-l}\int_{t_0}^{t_2}\sum_{x,y\in B_{r}}(\eta(x)-\eta(y))^2
(u_{+a}(t,x)^l+u_{+a}(t,y)^l)C_{x,y}\vartheta(t)\,dt\\
&\quad + C_3l\int_{t_0}^{t_2}\sum_{x\in B_{r}}\left(\frac{{\rm Tail}(u_{-}(t,\cdot),R)(x)}{a R^2}+\frac{\mu(x)}{(r-r_0)^2}\right)
u_{+a}(t,x)^l\eta(x)^2\vartheta(t)\,dt\\
&\quad-C_3\int_{t_0}^{t_2}\sum_{x\in B_{r}}u_{+a}(t,x)^l\eta(x)^2 \vartheta'(t)\,dt,
\end{split}
\end{equation}
where $u_{+a}=u+a$.
\end{lemma}
\begin{proof}
For any $q\in (0,1)$, define $v(t,x):=u_{+a}(t,x)^{({1-q})/{2}}$ and $\vv(t,x):=u_{+a}(t,x)^{-q}\eta(x)^2\vartheta(t)$.
As before, by testing \eqref{l4-1-1} with $\vv(t,x)$, we find that for every $t\in [t_0,t_2]$,
\begin{align*}
0&=\sum_{x\in B_{r}}\partial_t u_{+a}(t,x)\vv(t,x)-\sL u_{+a}(t,\cdot)(x)\vv(t,x)\\
&=\sum_{x\in B_{r}}\partial_t u_{+a}(t,x)\vv(t,x)\,dt+\frac{1}{2}\sum_{x,y\in B_{r}}
\left(u_{+a}(t,x)-u_{+a}(t,y)\right)\left(\vv(t,x)-\vv(t,y)\right)C_{x,y}\\
&\quad +\sum_{x\in B_{r},y\in B_{r}^c}
\left(u_{+a}(t,x)-u_{+a}(t,y)\right)\left(\vv(t,x)-\vv(t,y)\right)C_{x,y}\\
&=:I_1(t)+I_2(t)+I_3(t).
\end{align*}

By the same argument as that for the estimate of $I_1(t)$ in the proof of Lemma \ref{l4-1} and the fact that $\vartheta(t_2)=0$, we have
\begin{align*}
\int_{t}^{t_2} I_1(s)\,ds&=-\frac{1}{1-q}\sum_{x\in B_{r}}v(t,x)^2\eta(x)^{2}
\vartheta(t)
-\frac{1}{1-q}\int_{t}^{t_2}\sum_{x\in B_{r}}v(s,x)^2\eta(x)^{2}\vartheta'(s)\,ds.
\end{align*}

Set
\begin{align*}
  \int_{t}^{t_2}I_3(s)\,ds&=\int_{t}^{t_2}\sum_{x\in B_{r},y\in B_{r}^c}u_{+a}(s,x)\vv(s,x)C_{x,y}\,ds-
 \int_{t}^{t_2}\sum_{x\in B_{r},y\in B_{r}^c}u_{+a}(s,y)\vv(s,x)C_{x,y}\,ds\\
&=:I_{31}(t)+I_{32}(t).
\end{align*}
Following the same argument
as in
the proof of Lemma \ref{l4-1}, we get
\begin{align*}
 I_{31}(t)\le &\frac{1}{(r-r_0)^2}\int_{t}^{t_2}\sum_{x\in B_{r}}v(s,x)^2\eta(x)^{2}\mu(x)\vartheta(s)\,ds,\\
I_{32}(t)\le&  \frac{1}{a R^2}\int_t^{t_2}\sum_{x\in B_{r}}v(s,x)^2{\rm Tail}(u_{-}(s,\cdot),R)(x)\eta(x)^{2}\vartheta(s)\,ds.
\end{align*}

According to \eqref{l2-4-2}, we obtain
\begin{align*}
-\int_{t}^{t_2}I_2(s)\,ds\ge &c_{1,q}\int_{t_0}^t \sum_{x,y\in B_{r}}
\left(\eta(x)v(s,x)-\eta(y)v(s,y)\right)^2C_{x,y}\vartheta(s)\,ds\\
&-c_{2,q}\int_{t}^{t_2} \sum_{x,y\in B_{r}}\left(\eta(x)-\eta(y)\right)^2\left(
v(s,x)^2+v(s,y)^2\right)
C_{x,y}\vartheta(s)\,ds.
\end{align*}
Hence, combining
all the estimates above together yields, for every $t\in
[t_0,t_2]$,
\begin{align*}
& c_{1,q}\int_{t}^{t_2} \sum_{x,y\in B_{r}}
\left(\eta(x)v(s,x)-\eta(y)v(s,y)\right)^2C_{x,y}\vartheta(s)\,ds
+\frac{1}{1-q}\sum_{x\in B_{r}}v(t,x)^2\eta(x)^{2}
\vartheta(t)\\
&\le  \int_{t}^{t_2}\sum_{x\in B_{r}}v(s,x)^2\eta(x)^{2}
\left(\frac{{\rm Tail}(u_{-}(s,\cdot),R)(x)}{aR^2}+\frac{\mu(x)}{(r-r_0)^2}\right)\vartheta(s)\,ds\\
&\quad-\frac{1}{1-q}\int_{t}^{t_2}\sum_{x\in B_{r}}v(s,x)^2\eta(x)^{2}\vartheta'(s)\,ds\\
&\quad +c_{2,q}\int_{t}^{t_2} \sum_{x,y\in B_{r}}\left(\eta(x)-\eta(y)\right)^2\left(
v(s,x)^2+v(s,y)^2\right)C_{x,y}\vartheta(s)\,ds.
\end{align*}
Letting $l=1-q$ and following the same arguments in the proof of Lemma \ref{l4-1} again, we
obtain
\eqref{l4-2-1}.
\end{proof}

\begin{lemma}\label{l4-4}
Suppose that
\eqref{a2-1-1} holds for
some $p,q\in (1,+\infty]$
that satisfy \eqref{eq:newp-q}.
Let $\kappa:=\frac{\rho-1}{\rho}$
and $\theta:=\frac{1+\kappa}{p_*}$ with $\rho=\frac{d}{d-2+d/q}$. Then there exists a constant
$C_4>0$
such that for every $\gamma\in (0,1]$, $2\le r<R/4$,
$t_0\in \R$ and $l\in (0,\theta^{-1}]$
satisfying that $[t_0,t_0+4r^2]\subset [-R^2,R^2]$,
\begin{equation}\label{l4-4-1}
\|u_{+a}\|_{1,1,Q_{r}^+(t_0)}
\le \left(\frac{C_4 M_0}{\gamma^2}\right)^{\frac{p_*\theta(1+\theta)}{\theta-1}\left(\frac{1}{l}-1\right)}\|u_{+a}\|_{l,l,Q_{(1+\gamma)r}^+(t_0)},
\end{equation}
where
$Q_{r}^+(t_0)$ is defined by \eqref{e4-1},
$u_{+a}:=u+a$ with  $a\ge \left(\frac{2r}{R}\right)^2\sup_{t\in [t_0,t_0+4r^2]}\|{\rm Tail}\left(u_{-}(t,\cdot),R\right)\|_{p,B_{2r}}$, and $M_0:=\sup_{k\ge 1}\|\max\{\mu,1\}\|_{p,B_{k}}\cdot
\|\max\{\nu,1\}\|_{q,B_{k}}$.
\end{lemma}
\begin{proof} (1) As explained before, without loss of generality
we may assume
that $t_0=0$.
Given $2\le r_1<r_2<r\le R/4$, we can choose
$\eta:\Z^d \to \R_+$ such that
\eqref{l4-3-2} holds, and
$\vartheta\in C^\infty([0,r_2^2];\R_+)$ such that it is non-negative and decreasing, and satisfies
\begin{align*}
\vartheta(t)=1 \hbox{ for all }t\in
[0, r_1^2],\quad \vartheta(r_2^2)=0,\quad \|\vartheta\|_\infty\le 1,\quad \vartheta'(t)\ge -\frac{c_2}{r_2^2-r_1^2}\hbox{ for all } t\in(r_1^2,r_2^2).
\end{align*}
For every
$l\in (0,\theta^{-1}]$,
it holds that
$1-l\ge 1-\theta^{-1}$. Here we also used the fact $\theta:=\frac{1+\kappa}{p_*}>1$ thanks to $\frac{1}{p-1}+\frac{1}{q}<\frac{2}{d}$.

Then, we can apply \eqref{l4-2-1} with $\eta$ and $\vartheta$ above and obtain
\begin{align*}
&\int_{0}^{r_1^2}\sum_{x,y\in B_{r_1}}(u_{+a}(t,x)^{l/2}-
u_{+a}(t,y)^{l/2})^2C_{x,y} \,dt+\sup_{
0\le t \le r^2_1}\sum_{x\in B_{r_1}}u_{+a}(t,x)^l \\
&\le c_3\int_{0}^{r_2^2}\sum_{x,y\in B_{r_2}}(\eta(x)-\eta(y))^2
(u_{+a}(t,x)^l+u_{+a}(t,y)^l)C_{x,y}\,dt\\
&\quad + c_3\int_{0}^{r_2^2}\sum_{x\in B_{r_2}}\left(\frac{{\rm Tail}(u_{-}(t,\cdot),R)(x)}{aR^2}+\frac{\mu(x)}{(r_2-r_1)^2}\right)
u_{+a}(t,x)^l\,dt+\frac{c_3}{r_2^2-r_1^2}\int_{0}^{r_2^2}\sum_{x\in B_{r_2}}u_{+a}(t,x)^l\, dt=:I.
\end{align*}

Below, choosing $a\ge \left(\frac{r_2}{R}\right)^2\sup_{t\in [0,r_2^2]}\|{\rm Tail}\left(u_{-}(t,\cdot),R\right)\|_{p,B_{r_2}}$ and following the proof of Lemma \ref{l4-3}, we can obtain that
\begin{align*}
I&\le c_4|B_{r_2}|\left(\frac{1}{(r_2-r_1)^2}+\frac{1}{r_2^2-r_1^2}\right)\int^{r_2^2}_0 \left(1+\|\mu\|_{p,B_{r_2}}\right)\|u_{+a}(t,\cdot)^l\|_{p_*,B_{r_2}}\,dt
\end{align*}
which yields that
\begin{equation}\label{l4-4-2}
\begin{split}
&\int_{0}^{r_1^2}\sum_{x,y\in B_{r_1}}(u_{+a}(t,x)^{l/2}-
u_{+a}(t,y)^{l/2})^2C_{x,y} dt+\sup_{
0\le t\le r_1^2}\sum_{x\in B_{r_1}}u_{+a}(t,x)^l \\
&\le c_5|B_{r_2}|\left(\frac{1}{(r_2-r_1)^2}+\frac{1}{r_2^2-r_1^2}\right)\int^{r_2^2}_0 \left(1+\|\mu\|_{p,B_{r_2}}\right)\|u_{+a}(t,\cdot)^l\|_{p_*,B_{r_2}}\,dt.
\end{split}
\end{equation}

Recalling that $\kappa=\frac{1}{\rho_*}=\frac{2}{d}-\frac{1}{q}$, one can further get that
\begin{align*}
&\int^{r_1^2}_0\sum_{x\in B_{r_1}}u_{+a}(t,x)^{l(1+\kappa)}\,dt\\
&\le \int^{r_1^2}_0 \left(\sum_{x\in B_{r_1}}u_{+a}(t,x)^{l\rho}\right)^{\frac{1}{\rho}}\cdot\left(\sum_{x\in B_{r_1}}u_{+a}(t,x)^{l\kappa\rho_*}\right)^{\frac{1}{\rho_*}}\,dt\\
&\le \left(\sup_{t\in [0,r_1^2]}\sum_{x\in B_{r_1}}|u_{+a}(t,x)|^{l}\right)^{\frac{1}{\rho_*}}\\
&\quad\times  \int^{r_1^2}_0 \left(|B_{r_1}|^{{1}/{q}}\|\nu\|_{q,B_{r_1}}\sum_{x,y\in B_{r_1}}\left(u_{+a}(t,x)^{l/2}-u_{+a}(t,y)^{l/2}\right)^2C_{x,y}+
|B_{r_1}|^{1/\rho}\|u_{+a}(t,\cdot)^{l}\|_{p_*,B_{r_1}}\right)\,dt\\
&\le c_6M_{r_2}^{1+\kappa}|B_{r_2}|K(r_1,r_2)^{\kappa}
\left(|B_{r_2}|^{{2}/{d}}K(r_1,r_2)+1\right)
\left(\int^{r_2^2}_0 \|u_{+a}(t,\cdot)^{l}\|_{p_*,B_{r_2}}dt\right)^{1+\kappa},
\end{align*}
where $M_{r_2}:=\|\max\{\mu,1\}\|_{p,B_{r_2}}\cdot
\|\max\{\nu,1\}\|_{q,B_{r_2}}$, $K(r_1,r_2):=\frac{1}{r_2^2-r_1^2}+\frac{1}{(r_2-r_1)^2}$, in the first inequality we have used the H\"{o}lder inequality,
the second inequality follows from \eqref{l2-6-1}, and the last inequality is due to \eqref{l4-4-2}.
Hence, for every
$l\in (0,\theta^{-1}]$,
\begin{equation}\label{l4-4-3}
\begin{split}
&\|u_{+a}^{l}\|_{1+\kappa,1+\kappa,Q_{r_1}^+}\le c_7M_{r_2}\left(\frac{|B_{r_2}|}{|B_{r_1}|}
\left(|B_{r_2}|^{\frac{2}{d}}K(r_1,r_2)+1\right)\right)^{\frac{1}{1+\kappa}}\frac{r_2^2K(r_1,r_2)^{\frac{\kappa}{1+\kappa}}}{r_1^{\frac{2}{1+\kappa}}}
\|u_{+a}^{l}\|_{p_*,p_*,Q_{r_2}^+},
\end{split}
\end{equation}
where $u_{+a}:=u+a$ with $a\ge \left(\frac{r_2}{R}\right)^2\sup_{t\in [0,r_2^2]}\|{\rm Tail}\left(u_-(t,\cdot),R\right)\|_{p,B_{r_2}}$.

(2) For any $l>0$ and $r\ge 2$, define $\M(l,r):=\|u_{+a}\|_{l,l,Q_{r}^+}$, where $u_{+a}:=u+a$, with
$$a\ge \left(\frac{2r}{R}\right)^2\sup_{t\in [0,4r^2]}\|{\rm Tail}\left(u_-(t,\cdot),R\right)\|_{p,B_{2r}}.$$
Applying the same arguments as these for \eqref{l4-3-5} and using \eqref{l4-4-3}, we find that for every
$l\in (0,\theta^{-1}]$
and $1\le \delta_1<\delta_2\le 2$,
\begin{equation}\label{l4-4-4}
\M(\theta l, \delta_1 r)\le \left(\frac{c_8 M_0}{(\delta_2-\delta_1)^2}\right)^{\frac{p_*}{l}}\M(l,\delta_2 r).
\end{equation}

Given $l\in (0,\theta^{-1}]$, $2\le r<R/4$
and $\gamma\in (0,1]$, define $l_m:=\theta^{-m}$ and $\tilde r_j:=r+\frac{\gamma r}{2^m}$
for every $m\ge 0$. Then, for every $k\ge 1$, we use \eqref{l4-4-4} inductively to obtain
\begin{align*}
\M(1,r)&\le \M(1,\tilde r_k)=\M(\theta l_1, \tilde r_k)\le \left(\frac{c_84^k M_0}{\gamma^2}\right)^{\frac{p_*}{l_1}}\M(l_1,\tilde r_{k-1})\\
&\le \prod_{i=1}^k \left(\frac{c_84^{k-i+1} M_0}{\gamma^2}\right)^{\frac{p_*}{l_i}}\M(l_k,(1+\gamma)r).
\end{align*}
According to the facts that
\begin{align}\label{l4-4-5}
\sum_{i=1}^k \frac{1}{l_i}=\frac{\theta}{\theta-1} \left(\frac{1}{l_k}-1\right),\quad \sum_{i=1}^k \frac{k-i+1}{l_i}\le \frac{\theta^3}{(\theta-1)^3}\left(\frac{1}{l_k}-1\right),
\end{align}
we get
\begin{align*}
\M(1,r)\le \left(\frac{c_9 M_0}{\gamma^2}\right)^{\frac{p_*\theta}{\theta-1}\left(\frac{1}{l_k}-1\right)}
 \M(l_k,(1+\gamma)r).
\end{align*}
For every $l\in (0,\theta^{-1})$, we can find $k\ge 1$ such that $l_k\le l <l_{k-1}$. This along with \cite[(3.25),  p.\ 1557]{FK} gives us \begin{align*}
\frac{1}{l_k}-1\le (1+\theta)\left(\frac{1}{l}-1\right).
\end{align*}
Combining this with \eqref{l4-4-5} yields that for every $l\in (0,\theta^{-1})$,
\begin{align*}
\M(1,r)\le \left(\frac{c_9 M_0}{\gamma^2}\right)^{\frac{p_*\theta(1+\theta)}{\theta-1}\left(\frac{1}{l}-1\right)}
 \M(l_k,(1+\gamma)r)\le \left(\frac{c_9 M_0}{\gamma^2}\right)^{\frac{p_*\theta(1+\theta)}{\theta-1}\left(\frac{1}{l}-1\right)}
 \M(l,(1+\gamma)r),
\end{align*}
which completes the proof.
\end{proof}

According to the proof of \cite[Lemma 4.1]{FK}, we have the following logarithmic estimates.
\begin{lemma}\label{l4-5}
Suppose that $2\le r<R/4$.
Given $a> 0$, $\eta:\Z^d \to \R_+$ satisfying that ${\rm supp}[\eta]\subset B_r$ and
$\|\nabla \eta\|_\infty \le C_5r^{-2}$ for some
$C_5>0$,
there is a constant $C_6>0$ so that for every
$-R^2\le t_1<t_2\le R^2$,
\begin{equation}\label{l4-5-1}
\begin{split}
 &\int_{t_1}^{t_2}\sD\left(u_{+a}(t,\cdot),-\eta^2 u_{+a}(t,\cdot)^{-1}\right)\,dt\\
&\ge  \int_{t_1}^{t_2}\sum_{x,y\in B_r}\eta(x)\eta(y)\left|\log \frac{u_{+a}(t,x)}{\eta(x)}-\log \frac{u_{+a}(t,y)}{\eta(y)}\right|^2 C_{x,y}\,dt\\
&\quad -C_6r^{-2}(t_2-t_1)\sum_{x\in B_r}\mu(x)-\frac{C_6}{aR^2}\int_{t_1}^{t_2}\sum_{x\in B_r} \eta(x)^2{\rm Tail}\left(u_{-}(t,\cdot),R\right)(x)\,dt,
\end{split}
\end{equation}
where $u_{+a}=u+a$.
\end{lemma}
\begin{proof}
Noting that ${\rm supp}[\eta] \subset B_r$ and following the arguments for \cite[(4.1) in Lemma 4.1]{FK}, we derive
\begin{align*}
\sD&  \left(u_{+a}(t,\cdot), -\eta^2 u_{+a}(t,\cdot)^{-1}\right)\\
\ge& \sum_{x,y \in B_r}\eta(x)\eta(y)\left(\frac{\eta(x)u_{+a}(t,y)}{\eta(y)u_{+a}(t,x)}+
\frac{\eta(y)u_{+a}(t,x)}{\eta(x)u_{+a}(t,y)}-\frac{\eta(y)}{\eta(x)}-\frac{\eta(x)}{\eta(y)}\right)C_{x,y}\\
&+2\sum_{x\in B_r, y\in B_r^c} \left(u_{+a}(t,y)-u_{+a}(t,x)\right)\left(\eta(x)^2u_{+a}(t,x)^{-1}-\eta(y)^2u_{+a}(t,y)^{-1}\right)C_{x,y}\\
=: &I_1(t)+I_2(t).
\end{align*}

Following the proof of \cite[(4.2) in Lemma 4.1]{FK}, we can obtain
\begin{align*}
I_1(t)&\ge \sum_{x,y\in B_r}\eta(x)\eta(y)\left|\log \frac{u_{+a}(t,x)}{\eta(x)}-\log \frac{u_{+a}(t,y)}{\eta(y)}\right|^2 C_{x,y}-\sD(\eta,\eta)\\
&\ge \sum_{x,y\in B_r}\eta(x)\eta(y)\left|\log \frac{u_{+a}(t,x)}{\eta(x)}-\log \frac{u_{+a}(t,y)}{\eta(y)}\right|^2 C_{x,y}-C_0r^{-2}\sum_{x\in B_r}\mu(x),
\end{align*}
where in the last inequality we have used the fact that
\begin{align*}
-\sD(\eta,\eta)&\ge -2\sum_{x\in B_r,y\in \Z^d}\left(\eta(x)-\eta(y)\right)^2 C_{x,y}\\
&\ge -2\|\nabla \eta\|_\infty^2\sum_{x\in B_r} \sum_{y\in \Z^d}|x-y|^2C_{x,y} \ge -c_1r^{-2}\sum_{x\in B_r}\mu(x).
\end{align*}

According to the facts that $u(t,x)\ge 0$ for any $x\in B_R$ and ${\rm supp}[\eta]\subset B_r$,
\begin{align*}
I_2(t)&\ge  -2\sum_{x\in B_r, y \in B_r^c}\eta(x)^2 C_{x,y}-2\sum_{x\in B_r, y \in B_R^c}\eta(x)^2\frac{(u_{+a})_{-}(t,y)}{u_{+a}(t,x)}C_{x,y}=:I_{21}(t)+I_{22}(t).
\end{align*}
Note that
\begin{align*}
I_{21}(t)&=-2\sum_{x\in B_r, y \in B_r^c}\left(\eta(x)-\eta(y)\right)^2 C_{x,y}\\
&\ge -2\|\nabla \eta\|_\infty\sum_{x\in B_r}\sum_{y\in \Z^d}|x-y|^2C_{x,y}\ge -c_2 r^{-2}\sum_{x\in B_r}\mu(x)
\end{align*} and
\begin{align*}
I_{22}(t)&\ge -\frac{2}{aR^2}\sum_{x\in B_r}\eta(x)^2{\rm Tail}\left(u_{-}(t,\cdot),R\right)(x).
\end{align*}
Hence, putting
all the estimates together above yields  the desired conclusion \eqref{l4-5-1}.
\end{proof}

\begin{lemma}\label{l4-6}
Suppose that Assumption $\ref{a2-1}$ holds
for some $p,q\in (1,+\infty]$.
Then, for every $\gamma>0$, $2\le r<R/4$,
and $t_0\in \R$ satisfying that
$[t_0-4r^2,t_0+4r^2]\subset [-R^2,R^2]$,
there exists a constant $A$ $($which may depend on $u(0,\cdot)$ and $r$, but is independent of $t_0$ and $\gamma$$)$
such that the following estimates hold:
\begin{equation}\label{l4-6-1}
\left|Q_{r}^+(t_0)\cap \{(t,x)\in \R\times \Z^d: \log u_{+a}(s,x)<-\gamma-A\}\right|\le \frac{C_7(M_0)|Q_{r}^+(t_0)|}{\gamma},
\end{equation}
\begin{equation}\label{l4-6-2}
\left|Q^{-}_{r}(t_0)\cap \{(t,x)\in \R\times \Z^d: \log u_{+a}(s,x)>\gamma-A\}\right|\le \frac{C_7(M_0)|Q^{-}_{r}(t_0)|}{\gamma},
\end{equation}
where $u_{+a}:=u+a$ with $a\ge \left(\frac{2r}{R}\right)^2\sup_{t\in [t_0-4r^2,t_0+4r^2]}\|{\rm Tail}(u_{-}(t,\cdot),R)\|_{p,B_{2r}}$,
$Q^+_{r}(t_0)$ and $Q^-_{r}(t_0)$ are defined by \eqref{e4-1},
and $C_7(M_0)$ is a positive constant which only depend on $M_0:=\sup_{k\ge 1}\|\max\{\mu,1\}\|_{p,B_{k}}\cdot
\|\max\{\nu,1\}\|_{q,B_{k}}$.
\end{lemma}
\begin{proof}
For simplicity we  only give the proof of  \eqref{l4-6-1}, since the proof of \eqref{l4-6-2}
is similar.
As explained before, without loss of generality
we may assume that $t_0=0$.

Given $2\le r<R/4$, let $\eta:\Z^d \to \R_+$
be of the form $\eta(x)=\Phi(|x|)$ that satisfies $\eta(x)=1$ for all $x\in B_{r}$ and $\eta(x)=0$ for every $x\in B_{3r/2}^c$, where $\Phi:\R_+ \to \R_+$ is continuous
and non-increasing. Define $\phi(t,x):=\frac{\eta(x)^2}{u_{+a}(t,x)}$ and
$v(t,x):=-\log \frac{u_{+a}(t,x)}{\eta(x)}$.
Testing
$\phi$ in \eqref{l4-1-1}, we obtain that for every $0\le t_1<t_2\le r^2$,
\begin{equation}\label{l4-6-3}
\int_{t_1}^{t_2}\sum_{x\in B_{3r/2}}\eta(x)^2 \partial_s v(s,x)\, ds+\frac{1}{2}\int_{t_1}^{t_2}
\sD\left(u_{+a}(s,\cdot), \eta^2u_{+a}(s,\cdot)^{-1}\right)\,ds=0.
\end{equation}
Hence, applying \eqref{l4-5-1} to \eqref{l4-6-3},
for every $a\ge \left(\frac{2r}{R}\right)^2\sup_{t\in [-4r^2,4r^2]}\|{\rm Tail}(u_{-}(t,\cdot),R)\|_{p,B_{2r}}$ it holds that
\begin{equation}\label{l4-6-4}
\begin{split}
&\int_{t_1}^{t_2}\partial_s\left(\sum_{x\in B_{3r/2}}\eta(x)^2  v(s,x)\right) ds+\frac{1}{2}\int_{t_1}^{t_2}
\sum_{x,y\in B_{3r/2}}\eta(x)\eta(y)\left|v(s,x)-v(s,y)\right|^2 C_{x,y}\,ds\\
&\le c_1r^{-2}(t_2-t_1)\sum_{x\in B_{3r/2}}\mu(x)+\frac{c_1}{a R^2}\int_{t_1}^{t_2}\sum_{x\in B_{3r/2}}\eta(x)^2{\rm Tail}(u_{-}(s,\cdot),R)\,ds\\
&\le \frac{c_1(t_2-t_1)|B_{2r}|}{r^2}\left(\|\mu\|_{p,B_{3r/2}}+
\frac 1a\left(\frac{r}{R}\right)^2\sup_{s\in [-4r^2,4r^2]}\|{\rm Tail}(u_{-}(s,\cdot),R)\|_{p,B_{2r}}
\right)\\
&\le \frac{c_1(t_2-t_1)|B_{2r}|}{ r^2}\left(1+\|\mu\|_{p,B_{3r/2}}\right).
\end{split}
\end{equation}

Next, we set $V(t):=\frac{\sum_{x\in B_{3r/2}}v(t,x)\eta(x)^2}{\sum_{B_{3r/2}}\eta(x)^2}$ (for simplicity we omit the parameter $r$ in the definition).
Combining
\eqref{l4-6-4} with the weighted Poincar\'e inequality \eqref{l2-7-1} yields that
for every $0\le t_1<t_2\le r^2$,
\begin{equation*}
\begin{split}
&  \sum_{x\in B_{3r/2}}v(t_2,x)\eta(x)^2-\sum_{x\in B_{3r/2}}v(t_1,x)\eta(x)^2+
\frac{c_2}{r^2\|\nu\|_{d/2,B_{3r/2}}}\int_{t_1}^{t_2}\sum_{x\in B_{3r/2}}\left(v(s,x)-V(s)\right)^2\eta(x)^2 \,ds\\
&\le \int_{t_1}^{t_2}\partial_s\left(\sum_{x\in B_{3r/2}}\eta(x)^2  v(s,x)\right)\, ds+\frac{1}{2}\int_{t_1}^{t_2}
\sum_{x,y\in B_{3r/2}}\eta(x)\eta(y)\left|v(s,x)-v(s,y)\right|^2 C_{x,y}\,ds\\
&\le \frac{c_3(t_2-t_1)|B_{r}|}{r^2}\left(1+\|\mu\|_{p,B_{3r/2}}\right).
\end{split}
\end{equation*}
Dividing both sides of the inequality above by $\sum_{x\in B_{3r/2}}\eta(x)^2$,
and noting that $c_4|B_{r}|\le \sum_{x\in B_{3r/2}}\eta(x)^2$ $\le c_5|B_{r}|$,  we
obtain that for every $0\le t_1<t_2\le r^2$,
\begin{equation}\label{l4-6-5}
\begin{split}
V(t_2)-V(t_1)+\frac{c_6}{|Q_{r}^+|M_0}\int_{t_1}^{t_2}\sum_{x\in B_{r}}\left(v(s,x)-V(s)\right)^2\, ds
\le c_7M_0r^{-2}(t_2-t_1),
\end{split}
\end{equation}
where we have used the fact that $\eta(x)=1$ for every $x\in B_{r}$.
As pointed out in \cite[Remark 2.1]{S},
$V:[0,r^2]\to \R$ is continuous. So we can find
$\delta\in (0,1)$ such that
$$
|V(s_1)-V(s_2)|\le 1,\quad  s_1,s_2\in [0,r^2]\ {\rm with}\ |s_1-s_2|<\delta.
$$
Therefore, for every $0\le t_1<t_2\le r^2$ with $t_2-t_1<\delta$ and $s\in [t_1,t_2]$, it holds that
\begin{align*}
|v(s,x)-V(s)|^2 &\ge \left(|v(s,x)-V(t_2)|-|V(s)-V(t_2)|\right)^2\\
&\ge \frac{1}{2}|v(s,x)-V(t_2)|^2-8|V(s)-V(t_2)|^2\ge
\frac{1}{2}|v(s,x)-V(t_2)|^2-8.
\end{align*}
Putting this into \eqref{l4-6-5}, we have for every $0\le t_1<t_2\le r^2$ with $t_2-t_1<\delta$,
\begin{align*}
V(t_2)-V(t_1)+\frac{c_6}{2|Q_{r}^+|M_0}\int_{t_1}^{t_2}\sum_{x\in B_{r}}\left(v(s,x)-V(t_2)\right)^2 \,ds
\le r^{-2}(t_2-t_1)\left(\frac{8c_6}{M_0}+c_7M_0\right).
\end{align*}
Define
$$
w(s,x):=v(s,x)-\left(\frac{8c_6}{M_0}+c_7M_0\right)r^{-2}s,\quad W(s):=V(s)-\left(\frac{8c_6}{M_0}+c_7M_0\right)r^{-2}s.
$$
Then, for every $0\le t_1<t_2\le r^2$ with $t_2-t_1<\delta$, we obtain
\begin{equation}\label{l4-6-6}
W(t_2)-W(t_1)+\frac{c_6}{2|Q_{r}^+|M_0}\int_{t_1}^{t_2}\sum_{x\in B_{r}}\left(w(s,x)-W(t_2)+\left(\frac{8c_6}{M_0}+c_7M_0\right)r^{-2}(t_2-s)\right)^2\,ds\le 0,
\end{equation}
which obviously implies that $W:[0,r^2]\to \R$ is non-increasing.

Below we choose $A:=W(0)$, and, for every $s\in [0,r^2]$ and $\gamma>0$, we define
$$U^+(s,\gamma):=B_{r}\cap \{x\in \Z^d: w(s,x)>\gamma+A\}.$$
Then,
$$
w(s,x)-W(t_2)\ge \gamma+A-W(t_2)
=\gamma+W(0)-W(t_2)
\ge \gamma,\quad x\in U^+(s,\gamma),
$$
 and so
\begin{align*}
&  \left(w(s,x)-W(t_2)+\left(\frac{8c_6}{M_0}+c_7M_0\right)r^{-2}(t_2-s)\right)^2\\
&\ge \left(w(s,x)-W(t_2)\right)^2\ge \left(\gamma+A-W(t_2)\right)^2,\quad x\in U^+(s,\gamma) \hbox{ and } s\in [0,t_2].
\end{align*}
Hence, according to this and \eqref{l4-6-6}, we know that for every $0\le t_1<t_2\le r^2$ with $t_2-t_1<\delta$,
$$
W(t_2)-W(t_1)+\frac{
c_6\left(\gamma+A-W(t_2)\right)^2
}{2|Q_{r}^+|M_0}\int_{t_1}^{t_2}|U^+(s,\gamma)|ds\le 0.
$$
Also note that $W:[0,r^2]\to \R$ is non-increasing, so we obtain that for every $0\le t_1<t_2\le r^2$ with $t_2-t_1<\delta$,
\begin{equation}\label{l4-6-7}
\begin{split}
\frac{c_6}{2|Q_{r}^+|M_0}\int_{t_1}^{t_2}|U^+(s,\gamma)|\,ds&\le
\frac{W(t_1)-W(t_2)}{\left(\gamma+A-W(t_2)\right)^2}\\
&\le \frac{W(t_1)-W(t_2)}{\left(\gamma+A-W(t_1)\right)\left(\gamma+A-W(t_2)\right)}\\
&=\frac{1}{\gamma+A-W(t_1)}-\frac{1}{\gamma+A-W(t_2)}.
\end{split}
\end{equation}

Furthermore,  we divide the interval as $[0,r^2)=\sum_{i=0}^{N-1}[t_i,t_{i+1})$, where $t_0=0$, $t_N=r^2$ and
$t_{i+1}-t_i<\delta$ for every $0\le i \le N-1$. Therefore, by \eqref{l4-6-7} it holds that for every
$0\le i \le N-1$,
\begin{align*}
\frac{c_6}{2|Q_{r}^+|M_0}\int_{0}^{r^2}|U^+(s,\gamma)|ds&=
\sum_{i=0}^{N-1}
\frac{c_6}{2|Q_{r}^+|M_0}
\int_{t_i}^{t_{i+1}}|U^+(s,\gamma)|ds\\
&\le \sum_{i=0}^{N-1}
\Big(
\frac{1}{\gamma+A-W(t_i)}-\frac{1}{\gamma+A-W(t_{i+1})}
\Big)
\\
&=\frac{1}{\gamma+A-W(0)}-\frac{1}{\gamma+A-W(r^2)}\le \frac{1}{\gamma}.
\end{align*}
This means that (noting that $w(s,x)=-\log u_{+a}(s,x)
-\left(\frac{8c_6}{M_0}+c_7M_0\right)r^{-2}s$
for every $x\in B_{r}$),
\begin{align*}
\left|Q_{r}^+ \cap \left\{(s,x)\in \R\times \Z^d: \log u_{+a}(s,x)+\left(\frac{8c_6}{M_0}+c_7M_0\right)r^{-2}s< -\gamma-A\right\}\right|\le
\frac{2|Q_{r}^+|M_0}{c_6\gamma}.
\end{align*}
Hence, we obtain
\begin{align*}
&\left|Q_{r}^+ \cap \left\{(s,x)\in \R\times \Z^d: \log u_{+a}(s,x)< -\gamma-A\right\}\right|\\
&\le \left|Q_{r}^+ \cap \left\{(s,x)\in \R\times \Z^d: \log u_{+a}(s,x)+\left(\frac{8c_6}{M_0}+c_7M_0\right)r^{-2}s< -\gamma/2-A\right\}\right|\\
&\quad +\left|Q_{r}^+ \cap \left\{(s,x)\in \R \times \Z^d: \left(\frac{8c_6}{M_0}+c_7M_0\right)r^{-2}s\ge \gamma/2\right\}\right|\\
&\le
\frac{4|Q_{r}^+|M_0}{c_6\gamma}
+|Q_{r}^+|\left(1-\frac{\gamma}{2\left(\frac{8c_6}{M_0}+c_7M_0\right)}\right)_+\le
\frac{c_8(M_0)|Q_{r}^+|}{\gamma}.
\end{align*}
Therefore, we prove the desired assertion \eqref{l4-6-1}.
\end{proof}

Now we can prove the desired weak parabolic Harnack inequality as follows.
\begin{theorem}\label{t4-1}
Suppose
that \eqref{a2-1-1} holds for
some $p,q\in (1,+\infty]$
that satisfy \eqref{eq:newp-q}.
Then, for every $2\le r<R/4$
and $t_0\in \R$ satisfying $[t_0-4r^2,t_0+4r^2]\subset [-R^2,R^2]$,
it holds that
\begin{equation}\label{t4-1-2}
\begin{split}
\frac{1}{|U^-(t_0,r)|}\int_{t_0-2r^2}^{t_0-r^2}\sum_{x\in B_{r}}u(t,x)\,dt
&\le C_8(M_0)\inf_{(t,x)\in U^+(t_0,r)}u(t,x)\\
&\quad +C_8(M_0)\left(\frac{r}{R}\right)^2\sup_{t\in [t_0-4r^2,t_0+4r^2]}\|{\rm Tail}(u_{-}(t,\cdot),R)\|_{p,B_{2r}},
\end{split}
\end{equation}
where $U^-(t_0,r):=[t_0-2r^2,t_0-r^2]\times B_{r}$, $U^+(t_0,r):=[t_0+r^2,t_0+2r^2]\times B_{r}$,
$C_8(M_0)$ is a positive constant which may depend on $M_0:=\sup_{k\ge 1}\|\max\{\mu,1\}\|_{p,B_{k}}\cdot
\|\max\{\nu,1\}\|_{q,B_{k}}$ but is independent of $u$, $t_0$, $r$ and $R$.
\end{theorem}
\begin{proof}
Throughout the proof, all the constants $c_i$ will depend on $M_0$.
As explained before, without loss of generality we may assume that $t_0=0$. For any
$2\le r<R/4$
and $1/2\le \theta \le 1$, let
$$
U_1(\theta r):=\left[2r^2-(\theta r)^2, 2r^2\right]\times B_{\theta r},\quad
U_2(\theta r):=\left[-2r^2, -2r^2+(\theta r)^2\right]\times B_{\theta r}.
$$
Let $A$ be the same constant in Lemma \ref{l4-6},  and
set
$a:=\sup_{t\in [t_0-4r^2,t_0+4r^2]}\left(\frac{2r}{R}\right)^2\|{\rm Tail}(u_{-}(t,\cdot),R)\|_{p,B_{2r}}$, $w_1:=e^{-A}u_{+a}^{-1}$ and $w_2:=e^A u_{+a}$. Below  we use Lemma \ref{l4-6} to $w_{1}$ and $w_2$.
Therefore, by \eqref{l4-6-1} and \eqref{l4-6-2}, for every $\gamma>0$,
$$
\left|U_i(\theta r)\cap\left\{(s,x)\in \R\times \Z^d: \log w_i(s,x)>\gamma\right\}\right|\le \frac{c_1|U_i(\theta r)|}{\gamma},\quad i=1,2,\ 1/2\le \theta \le 1,
$$

Now, according to \eqref{l4-3-1}, we obtain that for every $l\in (0,1]$ and $1/2\le \theta \le 1$,
\begin{align*}
\sup_{(s,x)\in U_1(\theta r)}w_1(s,x)&=
e^{-A}\sup_{(s,x)\in U_1(\theta r)}u_{+a} (s,x)^{-1}\\
&\le e^{-A}\left(\left(\frac{c_2M_0}{(1-\theta)^2}\right)^{\frac{p_*(1+\kappa)}{1+\kappa-p_*}}\right)^{1/l}\|u_{+a}^{-1}\|_{l,l,U_1(r)}\\
&=\left(\left(\frac{c_2M_0}{(1-\theta)^2}\right)^{\frac{p_*(1+\kappa)}{1+\kappa-p_*}}\right)^{1/l}\|w_1\|_{l,l,U_1(r)},
\end{align*}
where  $\kappa:=\frac{\rho}{\rho-1}$ and $\rho:=\frac{d}{d-2+d/q}$.
On the other hand, by \eqref{l4-4-1},  for every $l\in (0,\frac{p_*}{1+\kappa}]$ and $1/2\le \theta \le 1$,
\begin{align*}
\|w_2\|_{1,1,U_2(\theta r)}&=e^A\|u_{+a}\|_{1,1,U_2(\theta r)}\\
&\le e^A\left(\left(\frac{c_2M_0}{(1-\theta)^2}\right)^{\frac{(1+\kappa)(1+\kappa+p_*)}{1+\kappa-p_*}}\right)^{1/l-1}\|u_{+a}\|_{l,l,U_2(r)}\\
&=\left(\left(\frac{c_2M_0}{(1-\theta)^2}\right)^{\frac{(1+\kappa)(1+\kappa+p_*)}{1+\kappa-p_*}}\right)^{1/l-1}\|w_2\|_{l,l,U_2(r)}.
\end{align*}
In particular, conditions \eqref{l2-8-1} and \eqref{l2-8-2} hold for the function $w_1$ with $p_0=+\infty$ and $\sigma=1$. According to Lemma \ref{l2-8}, we know that
$$
\sup_{(s,x)\in U_1(\theta r)}w_1(s,x)\le K_1, \quad 1/2\le \theta\le 1,
$$
where $K_1:=K_1(\theta)$ is a positive constant which may depend on $\theta$.
Similarly, conditions \eqref{l2-8-1} and \eqref{l2-8-2} hold for
$w_2$
with $p_0=1$ and $\sigma=\frac{p_*}{1+\kappa}$. Thus, applying Lemma \ref{l2-8} to $w_2$ yields
$$
\|w_2\|_{1,1,U_2(\theta r)}\le K_2:=K_2(\theta),\quad 1/2\le \theta\le 1.
$$

Choosing $\theta=1$ in the inequality above, we find
\begin{align*}
\frac{1}{|U_2(r)|}
\int_{-2r^2}^{-r^2}\sum_{x\in B_{r}}u(t,x)\,dt&\le
\frac{1}{|U_2(r)|}
\int_{-2r^2}^{-r^2}\sum_{x\in B_{r}}u_{+a}(t,x)\,dt\\
&=e^{-A}
\|w_2\|_{1,1,U_2(r)}
\le e^{-A}K_2 \le \frac{e^{-A} K_1K_2}{\sup_{(s,x)\in U_1(r)}w_1(s,x)}\\
&=\frac{K_1K_2}{\sup_{(s,x)\in U_1(r)}u_{+a}(s,x)^{-1}}=K_1K_2
\inf_{(s,x)\in U_1(r)}
u_{+a}(s,x)\\
&=K_1K_2\inf_{(t,x)\in U_1(r)}u(t,x)\\
&\quad +K_1K_2
\left(\frac{2r}{R}\right)^2
\sup_{t\in [-4r^2,4r^2]}\|{\rm Tail}(u_{-}(t,\cdot),R)\|_{p,B_{2r}},
\end{align*}
which completes the proof of \eqref{t4-1-2}.
\end{proof}

\section{Proof of Theorem \ref{CLT}}
In this section, we will apply the weak parabolic Harnack inequality obtained in
Theorem \ref{t4-1} to establish the following result about
the large scale regularity for caloric functions associated with $\sL$, which
is a crucial step to prove Theorem \ref{CLT}.

\begin{proposition} \label{HR}
Suppose that the function $u:\R\times \Z^d \to \R$ is non-negative on $[-2R^2,2R^2]\times B_R$ for some
constant $R\ge2$
and satisfies that
$$
\partial_t u(t,x)-\sL u(t,\cdot)(x)=0,\quad (t,x)\in [-2R^2,2R^2] \times B_R.
$$
Assume that \eqref{a2-1-1} holds
$p,q\in (1,+\infty]$
that satisfy \eqref{eq:newp-q}.
Then, there
are constants
$C_1>0$ and $\beta\in (0,1)$ $
($depending on  $M_0:=\sup_{k\ge 1}\|\max\{\mu,1\}\|_{p,B_{k}}\cdot
\|\max\{\nu,1\}\|_{q,B_{k}}$$)$
such that for all
$(t,x)\in [-R^2/2,R^2/2]\times B_{R}$
with $|x|+|t|^{1/2}\ge 2$,
\begin{equation}\label{hr1}
|u(t,x)-u(0,0)|\le C_1
\|u\|_{\infty,\infty,[-2R^2,2R^2]\times \Z^d}
\left(\frac{|x|+|t|^{1/2}}{R}\right)^\beta,
\end{equation} where
$$
\|u\|_{\infty,\infty,[-2R^2,2R^2]\times \Z^d}
:=
\sup_{(t,x)\in [-2R^2,2R^2]\times \Z^d}
|u(t,x)|.$$
\end{proposition}

\begin{proof} Set
$Q(R)=[-R^2/2,R^2/2]\times B_{R}$. It suffices to prove that there are constants
$K>0$, $\beta\in (0,1]$,
a non-decreasing sequence $a_k$ and a non-increasing sequence $A_k$ with $k=1,2,\cdots$
so that for all $k\in \Z_+$
with $6^{-k}R>2$,
\begin{equation}\label{e:hol1}
\begin{split}&
a_k\le u(t,x)\le A_k,
\quad  (t,x)\in Q(6^{-k}R),\\
  & A_k-a_k=K 6^{-k\beta}\|u\|_{\infty,\infty,[-2R^2,2R^2]\times \Z^d}.
  \end{split}
  \end{equation}
  Let $A_0=\sup_{(t,x)\in [ -2R^2,2R^2]\times \Z^d}u(t,x)$ and
  $a_0=\inf_{(t,x)\in [ -2R^2,2R^2]\times \Z^d}u(t,x)$.
   Define $A_{-n}=A_0$ and $a_{-n}=a_0$ for all $n\in \Z_+$. It is easy to verify that \eqref{e:hol1} holds for $k=0$ (with the choice of the constant
  $K=\frac{A_0-a_0}{\|u\|_{\infty,\infty,[-2R^2,2R^2]\times \Z^d}}\in [1,2]$).
  Suppose that we have constructed $A_k$ and $a_k$ for all $k\le n-1$ that satisfy \eqref{e:hol1}.
 Define
$$v(t,x)= \frac{  2\cdot 6^{(n-1)\beta}}{K}\left(u(t,x)-\frac{A_{n-1}+a_{n-1}}{2}\right),\quad (t,x)\in [-2R^2,2R^2]\times \Z^d.$$
Since \eqref{e:hol1} holds for $k=n-1$, it is easy to verify that $|v(t,x)|\le 1$ on $Q(6^{-n-1}R)$.
Furthermore, for any $j\ge1$, on $[-(6^{-n+j+1}R)^2/2, (6^{-n+j+1}R)^2/2]\times B_{6^{-n+j+1}R}\setminus B_{6^{-n+j}R}$, it holds that
\begin{align*}\frac{Kv(t,x)}{2\cdot 6^{(n-1)\beta}}=&u(t,x)-\frac{A_{n-1}+a_{n-1}}{2}\\
\le&  A_{n-j-1}-a_{n-j-1}+a_{n-j-1}-\frac{A_{n-1}+a_{n-1}}{2} \\
\le&  A_{n-j-1}-a_{n-j-1}-\frac{A_{n-1}-a_{n-1}}{2}\\
=&K 6^{-(n-j-1)\beta}-\frac{K}{2}6^{-(n-1)\beta},\end{align*} where in the second inequality we used the fact that $a_{n-j-1}\le a_{n-1}$ and the last equality follows from \eqref{e:hol1}. In particular,
\begin{equation}\label{e:hol1a}
v(t,x)\le 2\cdot 6^{j\beta}-1,\quad(t,x)\in [-(6^{-n+j+1}R)^2/2, (6^{-n+j+1}R)^2/2]\times B_{ 6^{-n+j+1}R}\setminus B_{6^{-n+j}R}.
\end{equation}
Similarly,
$$v(t,x)\ge 1-2\cdot 6^{j\beta}, \quad (t,x)\in [-(6^{-n+j+1}R)^2/2, (6^{-n+j+1}R)^2/2]\times B_{6^{-n+j+1}R}\setminus B_{6^{-n+j}R}.$$

Let $r=6^{-n}R$
and  $t_0 =-3(6^{-n}R)^2/2$. Then,
\begin{equation}\label{e:hol3}
\begin{split}
[t_0+r^2,t_0+2r^2]=&[-(6^{-n}R)^2/2,(6^{-n}R)^2/2],\\
[t_0-4r^2,t_0+4r^2]\subset& [-(6^{-(n-1)}R)^2/2,(6^{-(n-1)}R)^2/2].\end{split}
\end{equation}
Set $$D_n:=[t_0-2r^2, t_0-r^2]\times B_{6^{-n}R}=[-7(6^{-n}R)^2/2,-5(6^{-n}R)^2/2]\times B_{6^{-n}R}.$$
 Note that $w:=1-v$ is non-negative in $Q(6^{-(n-1)}R)$. Without loss of generality, we assume that
 $|D_{n}\cap \{v\le 0\}|\ge |D_{n}|/2$;
 otherwise, we will define $w:=1+v$ below instead. When $6^{-n}R>2$, applying \eqref{t4-1-2} to the function $v(t,x)$ with $r$, $t_0$ defined above and
 $R=6^{-(n-1)}R$ (thanks to \eqref{e:hol3}), we can derive that
\begin{equation}\label{e:hol2}
\begin{split}
&\frac{1}{|D_n|}\int_{-7(6^{-n}R)^2/2}^{-5(6^{-n}R)^2/2}\sum_{x\in B_{6^{-n}R}}w(t,x)\,dt\\
&\le c_0(M_0)\left(\inf_{Q(6^{-n}R)}w+\sup_{t\in[-(6^{-(n-1)}R)^2/2, (6^{-(n-1)}R)^2/2]}\|{\rm Tail}(w_{-}(t,\cdot),6^{-(n-1)}R)\|_{p,2\cdot 6^{-n}R}\right).
\end{split}\end{equation}
It is clear that
$w(t,x)\ge 1$ when $v(t,x)\le 0$, which implies that
$$\frac{1}{|D_n|}\int_{-7(6^{-n}R)^2/2}^{-5(6^{-n}R)^2/2}\sum_{x\in B_{6^{-n}R}}w(t,x)\,dt\ge \frac{|D_{n}\cap \{v\le 0\}|}{|D_{n}|}\ge \frac{1}{2}.$$
On the other hand, for all $x\in B_{2\cdot 6^{-n}R}$ and $t\in [-(6^{-(n-1)}R)^2/2,(6^{-(n-1)}R)^2/2]$,
\begin{align*} {\rm Tail}(w_{-},6^{-(n-1)}R)(t,x)=&6^{-2(n-1)}R^2\sum_{y\in B_{6^{-(n-1)}R}^c}{w_-(t,y)} C_{x,y}\\
=& 6^{-2(n-1)}R^2\sum_{j=1}^\infty
\sum_{y\in  B_{6^{-n+j+1}R} \backslash B_{6^{-n+j}R} }
{w_-(t,y)}   C_{x,y}\\
\le& c_1 \sum_{j=1}^\infty  (6^{j\beta}-1)6^{-2j} \mu(x)\\
=&c_2(\beta)\mu(x),\end{align*} where  in the   inequality above we used
the following facts that
can be deduced from \eqref{e:hol1a}
$$w_-(t,x)\le 2(6^{j\beta}-1),\quad(t,x)\in [-(6^{-n+j+1}R)^2/2, (6^{-n+j+1}R)^2/2]\times B_{ 6^{-n+j+1}R} \backslash B_{6^{-n+j)}R},$$
$$|x-y|\ge c_3 6^{-(n-j)}R,\quad x \in B_{2\cdot 6^{-n}R},\ y\in B_{6^{-n+j+1}R} \backslash B_{6^{-n+j}R},$$
and
\begin{equation*}
\sum_{y\in B_{6^{-n+j+1}R} \backslash B_{6^{-n+j}R}}C_{x,y}\le c_46^{2n-2j}\sum_{y\in \Z^d}|x-y|^2C_{x,y}=
c_46^{2n-2j}\mu(x),\ \quad x \in B_{2\cdot 6^{-n}R},\
\end{equation*}
and in the last equality $c_2(\beta)$ satisfies that $c_2(\beta)\to0$ as $\beta\to0$.
Putting
all the estimates above into \eqref{e:hol2}, applying \eqref{a2-1-1} and
choosing $\beta$ small enough if necessary, we
obtain
$$\inf_{Q(6^{-n}R)}w\ge \delta:= \delta(M_0)>0$$ for some $\delta\in (0,1)$ independent of $n$ and $\beta$ (but may depend on $M_0$).
Hence,
\begin{align*}
u(t,x)&=\frac{Kv(t,x)}{2\cdot 6^{(n-1)\beta}}+
\frac{A_{n-1}+a_{n-1}}{2}\\
&\le \frac{K(1-\delta)}{2\cdot 6^{(n-1)\beta}}
+a_{n-1}+\frac{A_{n-1}-a_{n-1}}{2}
\le\frac{K(1-\delta)}{2\cdot 6^{(n-1)\beta}}+
a_{n-1}
+\frac{K}{2}6^{-(n-1)\beta}\\
&\le a_{n-1}+K(1-\delta/2)
6^{-(n-1)\beta},
\quad  (t,x)\in Q(6^{-nR}).
\end{align*}
Now, choosing
$\beta:=\beta(M_0)$ small enough so that $(1-\delta/2)\le 6^{-\beta}$,
we can choose $a_{n}=a_{n-1}$ and $A_n=a_{n-1}+K6^{-n\beta}$, which means \eqref{e:hol1} is true for $k=n$ when
$6^{-n}R>2$, so we have proved
the desired conclusion.
\end{proof}

\begin{proof}[Proof of Theorem $\ref{CLT}$] We adopt the notations
in \cite{CH}. Let $E=F=\R^d$ and $G=\Z^d$ with $d_E(x,y)=d_F(x,y)=d_G(x,y)=|x-y|$, and $G_n=n^{-1}\Z^d$ with $d_{G_n}(x,y)=n|x-y|$. Denote by $\nu$ the Lebesgue measure in $\R^d$, and by $\nu^n$ the counting measure on $G_n$. It is clear that conditions (a)-(c) in \cite[Assumption 1]{CH} hold with $\alpha(n)=n$, $\beta(n)=n^d$ and
$\gamma(n)=n^2$.
Let $X_t^{n,\w}=n^{-1}X_{n^2t}^\w$ for any $t>0$, and denote by $p^{n,\w}(t,x,y)$
the heat kernel of the process  $(X_t^{n,\w})_{t\ge0}$ on $G_n$ with respect to the counting measure $\nu^n$. Then, $p^{n,\w}(t,x,y)=p^\w(n^2t,nx,xy)$ for all $x,y\in G_n$ and $t>0$, where $p^\w(t,x,y)$ is the heat kernel of the process $(X^\w_t)_{t\ge0}$. Note that, though \cite[Theorem 1]{CH} is stated for the local limit theorem for a discrete time Markov process, by checking the arguments therein, we can verify that the proof also works for a continuous time Markov process; see the proof of \cite[Theorem 1.11]{ADS1} for more details.

Next, we show that condition (d) in \cite[Assumption 1]{CH} holds for $p^{n,\w}(t,x,y)$. Indeed, according to \cite[Theorem 2.1]{BCKW}, under Assumption 1.1 the quenched invariance principle holds for $(X_t^\w)_{t\ge0}$ with the limit process being
Brownian motion with diffusion matrix $M$. In particular, this implies that for any $f\in C_b(\R^d)$
and $t>0$,
$$\lim_{n\to \infty}\left|\Ee^{\w}_{0} \left[f(n^{-1}X_{n^2 t}^{\w})\right]-\int_{\R^d} f(z)k_{M}(t,-z)\,dz\right|=0,$$
where
$\Ee^{\w}_x$ denotes the expectation (in the sense of the quenched law) with respect to the law of the process $(X_t^{\w})_{t\ge0}$ with the initial point $x\in G_n$ for fixed $\w\in \Omega$.
Combining this with Proposition \ref{HR}, we can further verify that for all $f\in C_b(\R^d)$ and $0<T_1<T_2$,
$$\lim_{n\to \infty}\sup_{t\in [T_1,T_2]}\left|\Ee^{\w}_{0} \left[f(n^{-1}X_{n^2 t}^{\w})\right]-\int_{\R^d} f(z)k_{M}(t,-z)\,dz\right|=0.$$
Since for every $x_0\in \R^d$ and $r>0$, $\int_{\partial B(x_0,r)} k_{M}(t,x-z)\,dz =0$, the conclusion above yields that
$$\lim_{n\to \infty}\sup_{t\in [T_1,T_2]}\left|\Pp^{\w}_{0} (n^{-1}X_{n^2t}^{\w}\in B(x_0,r))-\int_{B(x_0,r)} k_{M}(t,-z)\,dz\right|=0;$$
see e.g. \cite[Theorem 2.1]{Bi}. Therefore, (d) in \cite[Assumption 1]{CH} is satisfied.

Furthermore,
when \eqref{a2-2-2} holds,
it follows from Proposition \ref{HR} and the on-diagonal upper bound of $p^\w(t,x,y)$ (see Proposition \ref{p3-2}) that there are  $\beta\in (0,1)$
and
$n_0(\w)\in \Z_+$
such that for any $0<T_1\le T_2$, $r>0$, $\gamma\in (0,r]$ and
$n>n_0(\w)$,
\begin{equation}\label{t5-1-1}
\sup_{x,y\in B_{G_n}(0,\alpha(n) r),d_{G_n}(x,y)\le \alpha(n)\gamma}
\sup_{t\in [T_1,T_2]}|p^\w(n^2t,0,[nx])-p^\w(n^2t,0,[ny])|\le c_1(T_1,T_2,r)n^{-d}\gamma^\beta,
\end{equation}
where $c_1>0$
(which may be random)
is independent of $n$ and $r$, and $B_{G_n}(0,r)$ is
the ball on $G_n$ with center $0$ and radius $r$.
By \eqref{t5-1-1}, we know immediately that \cite[Assumption 2]{CH} holds.
On the other hand, when $q=\infty$,  $\inf_{x,y\in \Z^d:|x-y|=1}C_{x,y}\ge c_2>0$. So, by the classical
Nash inequality in $\Z^d$,  we have
\begin{align*}
\left(\sum_{x\in \Z^d}|f(x)|^{2}\right)^{\frac{1}{2}+\frac{1}{d}}&\le c_3\left(\sum_{x\in \Z^d}|f(x)|\right)^{\frac{2}{d}}
\cdot \sqrt{\sD_{\Z^d}(f,f)}\\
&\le c_4\left(\sum_{x\in \Z^d}|f(x)|\right)^{\frac{2}{d}}
\cdot \sqrt{\sD^\w(f,f)},\quad f\in L^2(\Z^d;\lambda),
\end{align*}
where $\sD_{\Z^d}(f,f):=\sum_{x,y\in \Z^d:|x-y|=1}|f(x)-f(y)|^2$. With this, it is standard to verify that (see e.g. \cite{CKS}) the following
upper bound estimate holds for $p^\w(t,x,y)$
\begin{align*}
p^\w(t,x,y)\le c_5t^{-d/2},\quad x,y\in \Z^d,\ t\ge 1.
\end{align*}
Combining this heat kernel upper bound with Proposition \ref{HR}, we can follow the arguments above to see that \eqref{t5-1-1} is still true, and so \cite[Assumption 2]{CH} holds.

With all the conclusions at hand, we can obtain the desired assertion by \cite[Theorem 1.1]{CH}.
\end{proof}

\begin{remark} The quenched local limit theorem on degenerate random media
was first proved in \cite{BH}. However, in order to apply \cite[Corollary 4.3]{BH}, one needs to prove the parabolic Harnack inequality, which often fails for long-range degenerate random media (see \cite[Example 4.9]{XCKW}). We believe that \cite[Corollary 4.3]{BH}
still holds if one replaces \cite[Assumption 4.1(c)]{BH} by the
weak  parabolic Harnack inequality that is proved in this paper.
\end{remark}

Finally, we give the

\begin{proof}[Proof of Example $\ref{exmT}$]
The proof mainly follows from that of \cite[Corollay 2.3]{BCKW}. For any $z\in \Z^d$ and $n\in \N$,  set
\begin{align*}
&W_{z}:=C_{0,z}-\Ee[C_{0,z}],\quad \tilde W_{z}:=C_{0,z}^{2p/(p+1)}-\Ee[C_{0,z}^{2p/(p+1)}],\\
&M_n:=\sum_{n_0\le |z|\le n_0+n}|z|^2 W_{z},\quad \tilde M_n:=\sum_{n_0\le |z|\le n_0+n}|z|^{\gamma}\tilde W_{z},
\end{align*}
where $\gamma:=\frac{d(p-1)+4p}{p+1}$
with $p>1$ to be fixed later such that
$\tilde \mu\in L^p(\Omega;\Pp)$,
and $n_0$ is a positive integer also to be determined later. Since
$\{C_{0,z}\}_{z\in \Z^d}$ is mutually independent, both $\{M_n\}_{n\ge 1}$ and
$\{\tilde M_n\}_{n\ge 1}$ are martingales with respect to the filtration $\mathscr{F}_n:=\sigma\left\{C_{0,z}:|z|\le n_0+n\right\}$,
and the corresponding quadratic variational processes $\{\langle M \rangle_n\}_{n\ge 1}$ and $\{\langle \tilde M \rangle_n\}_{n\ge 1}$ are given by
$$
\langle M\rangle_n=\sum_{n_0\le |z|\le n_0+n}|z|^4 |W_z|^2,\quad \langle \tilde M\rangle_n= \sum_{n_0\le |z|\le n_0+n}|z|^{2\gamma} |\tilde W_z|^2.
$$
According to the Burkholder-Gundy-Davis inequality, for every $l\ge 1$ there is a constant $c_1>0$ such that
\begin{align*}
&\Ee[|M_n|^l]\le c_1\Ee\left[\langle M\rangle_n^{l/2}\right]=c_1\Ee\left[\left(\sum_{n_0\le |z|\le n_0+n}|z|^4 |W_z|^2\right)^{l/2}\right],\\
&\Ee[|\tilde M_n|^l]\le c_1\Ee\left[\langle \tilde M\rangle_n^{l/2}\right]=c_1\Ee\left[\left(\sum_{n_0\le |z|\le n_0+n}|z|^{2\gamma} |\tilde W_z|^2\right)^{l/2}\right].
\end{align*}
This along with \cite[Theorem 1.1]{L} yields that
\begin{equation}\label{ex1-3}
\begin{split}
&\Ee[|M_n|^l]\le c_2\inf\left\{t>0:\sum_{n_0\le |z|\le n_0+n}\log\left(\Ee\left[\left(1+t^{-1}|z|^4|W_{z}|^2\right)^{l/2}\right]\right)\le l/2\right\},\\
&\Ee[|\tilde M_n|^l]\le c_2\inf\left\{t>0:\sum_{n_0\le |z|\le n_0+n}\log\left(\Ee\left[\left(1+t^{-1}|z|^{2\gamma}|\tilde W_{z}|^2\right)^{l/2}\right]\right)\le l/2\right\}.
\end{split}
\end{equation}

(i) If $C_{x,y}$ is defined by \eqref{ex1-1}, it is not difficult to verify that
$\Ee\left[|W_z|^r+|\tilde W_z|^r\right]\le c_3|z|^{-d-s}$ for all $r\ge1$. Combining this with
the fact that
for all $r>0$ there exists $c_4:=c_4(r)>0$ such that $(1+t)^r\le 1+c_4t\I_{\{r\ge 1\}}+c_4t^r$ for all $t>0$
yields that for all $t>0$ and $l\ge1$,
\begin{equation}\label{ex1-4}
\begin{split}
\log \left(\Ee\left[\left(1+t^{-1}|z|^{2r}W_z^2\right)^{l/2}\right]\right)
&\le \log\left(1+c_5t^{-1}|z|^{2r}\Ee[|W_z|^2]\I_{\{l/2\ge 1\}}+c_5t^{-l/2}|z|^{rl}\Ee[|W_z|^l]\right)\\
&\le c_5\left(t^{-1}+t^{-l/2}\right)|z|^{rl-d-s},
\end{split}
\end{equation} where in the last inequality we used the fact that $\ln(1+x)\le x$ for all $x>0$.
The inequality above still holds, when $W_z$ is replaced by $\tilde W_z$.

If $s\in [2d-2,\infty)\cap(d+2,\infty)$, then, by \eqref{ex1-4}, it holds for some
$p\in [d-1,\infty)\cap (d/2+1,\infty)$ and for all $t>0$ that
$$
\sum_{z\in \Z^d}\log \left(\Ee\left[\left(1+t^{-1}|z|^{4}W_z^2\right)^{p/2}\right]\right)<\infty.
$$
Putting this into \eqref{ex1-3} (by taking $t=1$ and choosing $n_0$ large enough if necessary), we get
\begin{align*}
\sup_{n\ge n_0}\Ee[|M_n|^p]<\infty\quad\text{for\ some}\ p\in [d-1,\infty)\cap (d/2+1,\infty).
\end{align*}
This along with Fatou's lemma gives us $$\Ee\left[\left|\sum_{|z|\ge n_0}|z|^2W_z\right|^p\right]<\infty.$$
Hence, for some
$p\in [d-1,\infty)\cap (d/2+1,\infty)$,
\begin{align*}
\Ee\left[|\tilde \mu|^p\right]\le c_6
\left(\Ee\left[\left|\sum_{|z|<n_0}|z|^2C_{0,z}\right|^p\right]+\Ee\left[\left|\sum_{|z|\ge n_0}|z|^2W_z\right|^p\right]+\left|\sum_{|z|\ge n_0}|z|^2\Ee[C_{0,z}]\right|^p\right)<\infty.
\end{align*}
On the other hand, it follows from \eqref{ex1-1} that $C_{x,y}\equiv 1$ for every $x,y\in \Z^d$ with $|x-y|=1$, which in particular implies that $\tilde \nu \in L^\infty(\Omega;\Pp)$.
Therefore, \eqref{a2-2-1} and so Assumption \ref{Assm2} hold,
hence by Theorem \ref{CLT}, the quenched local limit theorem holds.

As explained above, $\tilde \nu \in L^\infty(\Omega;\Pp)$ by \eqref{ex1-1}.
Now suppose that  \eqref{ex1-1} holds with $s>{d^2}/{4}+{d}/{2}$. Since $d\ge2$ and $d^2/4+d/2\ge d$, we can find $p>d/2$ (which is close to $d/2$) so that $s>2p$ and $s>\gamma(p+1)/2>{d^2}/{4}+{d}/{2}$.
Next, we fix this $p>d/2$.
 Since $s>2p$, one can follow the arguments above to verify that $\sup_{n\ge 1}\Ee[|M_n|^p]<\infty$ when \eqref{ex1-1} holds with $s>{d^2}/{4}+{d}/{2}$. Hence, $\tilde \mu\in L^p(\Omega;\Pp)$ holds for
$p>d/2$, which implies that Assumption \ref{a2-1} holds.
Now, note that $\gamma:=\frac{d(p-1)+4p}{p+1}\ge1$ and $(p+1)/2\ge1$. Using the fact that $s>\gamma(p+1)/2$ and repeating the arguments
in
\eqref{ex1-4}, we can get $\sup_{n\ge 1}\Ee[|\tilde M_n|^{(p+1)/2}]<\infty$ if \eqref{ex1-1} holds with $s>{d^2}/{4}+{d}/{2}$. From this we can deduce that
$\tilde \mu_*\in L^{(p+1)/2}(\Omega;\Pp)$ holds for $p>d/2$,
namely \eqref{p3-1-1a} is satisfied. Hence, by Theorem \ref{p3-1}, \eqref{p3-1-1} holds.

(ii)  If $C_{x,y}$ is defined by \eqref{ex1-2}, then we can check directly that for every $r>0$ and $z\in \Z^d$,
$$\Ee\left[|W_z|^r\right]\le c_7|z|^{-r(d+s)},\quad \ \Ee\left[|\tilde W_z|^r\right]\le c_7|z|^{-\frac{2pr(d+s)}{p+1}}$$
as long as $\xi_{0,z}\in L^{2pr/(p+1)}(\Omega;\Pp)$. Applying this property into \eqref{ex1-3} and \eqref{ex1-4} and following the same arguments as
above, we can prove the desired conclusion in (ii).\end{proof}

\ \

\noindent {\bf Note.}\,\, When this paper was nearly finished, we were informed from
Andres and Slowik that their paper \cite{AS23} was posted on arXiv. In \cite{AS23} they
consider quenched invariance principle and quenched local limit theorem for a class of
random conductance models with long-range jumps.
A week after, it was withdrawn since there is a counter example to their main results in
\cite[Proposition 1.5 (ii)]{DF}. While they might be able to recover the main results by putting
additional conditions, we point out that the methods of their paper are quite different
from ours in that they do not prove weak parabolic Harnack inequalities,
which is the main achievement of our paper.
The two works are done completely independently by using different techniques.

\bigskip

\noindent {\bf Acknowledgements.}\,\,  The research of Xin Chen is supported by the National Natural Science Foundation of China
(No.\ 12122111).
The research of Takashi Kumagai is supported by JSPS
KAKENHI Grant Numbers 22H00099 and 23KK0050.
The research of Jian Wang is supported by the National Key R\&D Program of China (2022YFA1006003) and the National Natural Science Foundation of China (Nos. 12071076 and 12225104).


\begin{thebibliography}{99}
\bibitem{AN} M. Aizenman and C. Newman: Discontinuity of the percolation density in one-dimensional $1/|x-y|^2$
percolation models, \emph{Commun. Math. Phys.}, {\bf 107} (1986), 611--647.

\bibitem{ABDH} S. Andres, M.T. Barlow, J.-D. Deuschel and B.M. Hambly:
Invariance principle for the random conductance model, \emph{Probab. Theory Related Fields}, {\bf 156} (2013), 535--580.

\bibitem{ACDS} S. Andres, A. Chiarini, J.-D. Deuschel and M. Slowik:  Quenched invariance principle for random walks with time-dependent ergodic degenerate weights,
\emph{Ann. Probab.}, {\bf 46} (2018), 302--336.

\bibitem{ACS} S. Andres, A. Chiarini and M. Slowik:
Quenched local limit theorem for random walks among time-dependent ergodic degenerate weights,
\emph{Probab. Theory Related Fields}, {\bf 179} (2021), 1145--1181.



\bibitem{ADS}  S. Andres, J.-D. Deuschel and M. Slowik: Invariance principle for the random conductance model in a degenerate ergodic environment,
\emph{Ann. Probab.}, {\bf 43} (2015), 1866--1891.

\bibitem{ADS1} S. Andres, J.-D. Deuschel and M. Slowik:  Harnack inequalities on weighted graphs and some applications to the random conductance model,
\emph{Probab. Theory Related Fields}, {\bf 164} (2016), 931--977.


\bibitem{AS23} S. Andres and M. Slowik:
Invariance principle and local limit theorem for a class of random conductance models with long-range jumps,
arXiv:2311.07472.

\bibitem{BD} M.T. Barlow and J.-D. Deuschel: Invariance principle for the random conductance model with unbounded
conductances, \emph{Ann. Probab.}, {\bf 38} (2010), 234--276.


\bibitem{BH} M. T. Barlow and B. M. Hambly:
Parabolic Harnack inequality and local limit theorem for percolation clusters,
\emph{Electron. J. Probab.}, {\bf 14} (2009), 1--27.


\bibitem{Ba2-1} J. B\"aumler: Recurrence and transience of symmetric random walks with long-range jumps,
\emph{Electron. J. Probab.},  {\bf 28} (2023), 1--24.

\bibitem{Ba2} J. B\"aumler: Distances in $1/|x-y|^{2d}$ percolation models for all dimensions, \emph{Commun. Math. Phys.}, {\bf 404} (2023), 1495--1570.

\bibitem{Ba1} J. B\"aumler: Behavior of the distance exponent for $1/|x-y|^{2d}$ long-range percolation, arXiv:2208.04793.



\bibitem{Bab} J. B\"aumler adn N. Berger: Isoperimetric lower bounds for critical exponents for long-range percolation,
arXiv:2204.12410.


\bibitem{BFO}  P. Bella, B. Fehrman and F. Otto:
A Liouville theorem for elliptic systems with degenerate ergodic coefficients,
\emph{Ann. Appl. Probab.}, {\bf 28} (2018), 1379--1422.

\bibitem{BS1} P. Bella and M. Sch\"affner: Quenched invariance principle for random walks among random degenerate conductances,
\emph{Ann. Probab.}, {\bf 48} (2020), 296--316.

\bibitem{BS} P. Bella and M. Sch\"affner: Non-uniformly parabolic equations and applications to the random conductance model,
\emph{Probab. Theory Related Fields}, {\bf 182} (2022), 353--397.



\bibitem{Be} N. Berger:
Transience, recurrence and critical behavior for long-range percolation,
\emph{Comm. Math. Phys.}, {\bf 226} (2002), 531--558.

\bibitem{BB} N. Berger and M. Biskup: Quenched invariance principle for simple random walk on percolation clusters,
\emph{Probab. Theory Related Fields}, {\bf 137} (2007), 83--120.


\bibitem{Bi}
P. Billingsley: {\it Convergence of Probability Measures}, second edition, Wiley Series in Probability and
Statistics, John Wiley Sons Inc., New York, 1999.


\bibitem{B1} M. Biskup: On the scaling of the chemical distance in long range percolation models, \emph{Ann. Probab.},
{\bf 32} (2004), 2938--2977.

\bibitem{B} M. Biskup: Recent progress on the random conductance model, \emph{Prob. Surveys}, {\bf 8} (2011), 294--373.




\bibitem{BCKW}
 M. Biskup, X. Chen, T. Kumagai and J. Wang: Quenched invariance principle for a class of random conductance models with long-range jumps,
\emph{Probab. Theory Relat. Fields}, {\bf 180} (2021), 847--889.


\bibitem{BisP}
M. Biskup and T.M. Prescott:
Functional CLT for random walk among bounded random conductances,
\emph{Electron. J. Probab.}, {\bf 12} (2007), 1323--1348.


\bibitem{BR} M. Biskup and P.-F. Rodriguez: Limit theory for random walks in degenerate time-dependent random environments,
\emph{J. Funct. Anal.}, {\bf 274} (2018), 985--1046.



\bibitem{BH} Y. Bokredenghel and M. Heida:
Quenched homogenization of infinite range random conductance model on stationary point processes,
\emph{WIAS preprint}, 2023, no.\ 3017

\bibitem{BG} E. Bombieri and E. Giusti: Harnack's inequality for elliptic differential equations on minimal surfaces,
\emph{Invent. Math.}, {\bf 15} (1972), 24--46.



\bibitem{CCK} V.H. Can, D.A. Croydon and T. Kumagai:
Spectral dimension of simple random walk on a long-range percolation cluster, \emph{Electron. J. Probab.}, {\bf 27} (2022), Paper no. 56, 37 pp.



\bibitem{CKS} E.A. Carlen, S. Kusuoka and D.W. Stroock:
Upper bounds for symmetric Markov transition functions, \emph{Ann. Inst. H. Poincar\'e Probab. Statist.}, {\bf 23} (1987), 245--287.

\bibitem{XCKW} X. Chen, T. Kumagai and J. Wang:
Random conductance models with stable-like jumps: heat kernel estimates and Harnack inequalities,
\emph{J. Funct. Anal.}, {\bf 279} (2020), paper no. 108656, 51 pp.


\bibitem{XCKW1} X. Chen, T. Kumagai and J. Wang: Random conductance models with stable-like jumps: quenched invariance principle,
\emph{Ann. Appl. Probab.}, {\bf 31} (2021), 1180--1231.


\bibitem{CKW} Z.-Q. Chen, T. Kumagai and J. Wang: Stability of heat kernel estimates for symmetric non-local Dirichlet forms,
\emph{Mem. Amer. Math. Soc.}, {\bf 271} (2021), no. 1330, v+89 pp.


\bibitem{C} T. Coulhon: Espaces de Lipschitz et in\'egalit\'es de Poincar\'e,
\emph{J. Funct. Anal.}, {\bf 136} (1996), 81--113.



\bibitem{CS1} N. Crawford and A. Sly: Simple random walk on long range percolation clusters I: heat kernel bounds,
\emph{Probab. Theory Relat. Fields}, {\bf 154} (2012), 753--786.

\bibitem{CS2} N. Crawford and A. Sly: Simple random walk on long range percolation clusters II: scaling limits,
\emph{Ann. Probab.}, {\bf 41} (2013), 445--502.


\bibitem{CH} D.A. Crodyon and B.M. Hambly: Local limit theorems for sequences of simple random walks on graphs,
\emph{Potential Anal.}, {\bf 29} (2008), 351--389.


\bibitem{DKP} A. Di Castro, T. Kuusi and G. Palatucci:
Local behavior of fractional p-minimizers,
\emph{Ann. Inst. H. Poincar\'e $(C)$ Anal. Non Lin\'eaire}, {\bf 33} (2016), 1279--1299.


\bibitem{DF}J.-D. Deuschel and R. Fukushima:
Quenched tail estimate for the random walk in random scenery and in random layered conductance II,
\emph{Electron. J. Probab., \bf 25} (2020), 1--28.

\bibitem{DFH} J. Ding, Z.-R. Fan and L.-J. Huang: Uniqueness of the critical long-range percolation metrics, arXiv:2308.00621.

\bibitem{DS} J. Ding and A. Sly: Distances in critical long range percolation, arXiv:1303.3995.



\bibitem{Fa}
A. Faggionato: Stochastic homogenization of random walks on point processes, \emph{Ann. Inst. Henri Poincar\'e Probab. Stat.}, \textbf{59} (2023), 662--705.

\bibitem{FK}
M. Felsinger and M. Kassmann: Local regularity for parabolic nonlocal operators,
\emph{Comm. Partial Differential Equations}, {\bf 38} (2013), 1539--1573.


\bibitem{FHS}  F. Flegel, M. Heida and M. Slowik:
Homogenization theory for the random conductance model with degenerate ergodic weights and unbounded-range jumps,
\emph{Ann. Inst. Henri Poincar\'e Probab. Stat.}, {\bf 55} (2019), 1226--1257.


\bibitem{Giu}
 E. Giusti: \emph{Direct Methods in the Calculus of Variations}. World Scientific Publishing Co.
Inc., River Edge, 2003.

\bibitem{L} R. Latala: Estimation of moments of sums of independent real random variables,
\emph{Ann. Probab.}, \textbf{25} (1997), 1502--1513.

\bibitem{M} P. Mathieu: Quenched invariance principles for random walks with random conductances,
\emph{J. Stat. Phys.}, {\bf 130} (2008), 1025--1046.

\bibitem{MP} P. Mathieu and A.L. Piatnitski: Quenched invariance principles for random walks on percolation clusters,
\emph{Proc. R. Soc. Lond. Ser. A Math. Phys. Eng. Sci.}, {\bf 463} (2007), 2287--2307.

\bibitem{N} T.A. Nguyen: A Liouville principle for the random conductance model under degenerate conditions, arXiv:1908.10691.

\bibitem{PZ2} A. Piatnitski  and E. Zhizhina: Stochastic homogenization of  convolution-type operators, \emph{J. Math. Pures Appl.}, {\bf 134} (2020), 36--71.


\bibitem{Sa} L. Saloff-Coste: \emph{Lectures on Finite Markov Chains}, in: Lectures on Probability Theory and Statistics
(Saint-Flour, 1996), Lecture Notes in Mathematics,  vol. {\bf 1665}, Springer, Berlin, 1997, 301--413.


\bibitem{SS} V. Sidoravicius and A.-S. Sznitman: Quenched invariance principles for walks on clusters of percolation
or among random conductances, \emph{Probab. Theory Related Fields}, {\bf 129} (2004), 219--244.

\bibitem{S}
M. Str\"omqvist: Harnack's inequality for parabolic nonlocal equations,
\emph{Ann. Inst. H. Poincar\'e $(C)$ Anal. Non Lin\'eaire}, {\bf 36} (2019), 1709--1745.
\end{thebibliography}
\end{document}